\documentclass[amsfonts,12pt]{amsart}

\usepackage{amssymb}
\usepackage{epsfig}
\usepackage{color}

\newcounter{Cnmb}
\newcounter{Nnmb}
\newcommand{\cK}{\mathcal K}
\newtheorem{thm}{Theorem}[section]

\newtheorem{lem}[thm]{Lemma}
\newtheorem{cor}[thm]{Corollary}
\newtheorem{rem}[thm]{Remark}
\newtheorem{prop}[thm]{Proposition}

\newcommand{\inte}{{\mathrm{int}}\,}
\newcommand{\relint}{{\mathrm{relint}}\,}

\newcommand{\conv}{{\mathrm{conv}}\,}
\newcommand{\diam}{{\mathrm{diam}}\,}

\newcommand{\var}{{\mathrm{var}}\,}
\newcommand{\cov}{{\mathrm{cov}}\,}
\newcommand{\vN}{{\mathbf{N}}}
\newcommand{\vx}{{\mathbf{x}}}
\newcommand{\vX}{{\mathbf{X}}}
\newcommand{\vz}{{\mathbf{z}}}

\newcommand{\ee}{\varepsilon}

\newcommand{\R}{{\mathbb R}}

\newcommand{\N}{{\mathbb N}}

\newcommand{\Z}{{\mathbb Z}}

\begin{document}
\hfill\today
\bigskip

\title{Phase retrieval for characteristic functions of convex bodies and reconstruction from covariograms}
\author[Gabriele Bianchi, Richard J.~Gardner, and Markus Kiderlen]
{Gabriele Bianchi, Richard J.~Gardner, and Markus Kiderlen}
\address{Dipartimento di Matematica, Universit\`a di Firenze, Viale Morgagni 67/A,
Firenze, Italy I-50134} \email{gabriele.bianchi@unifi.it}
\address{Department of Mathematics, Western Washington University,
Bellingham, WA 98225-9063} \email{Richard.Gardner@wwu.edu}
\address{Department of Mathematical Sciences, University of Aarhus,
Ny Munkegade, DK--8000 Aarhus C, Denmark} \email{kiderlen@imf.au.dk}
\thanks{Supported in part by U.S.~National Science Foundation grant
DMS-0603307.} \subjclass[2010]{Primary: 42--04, 42B10, 52--04, 52A20; secondary:
52B11, 62H35} \keywords{Algorithm, autocorrelation, convex body, convex polytope,
covariogram, geometric tomography, image analysis, least squares, phase
retrieval, quasicrystal, set covariance} \maketitle

\pagestyle{myheadings} \markboth{GABRIELE BIANCHI, RICHARD J. GARDNER, AND
MARKUS KIDERLEN}{PHASE RETRIEVAL FOR CHARACTERISTIC FUNCTIONS OF CONVEX BODIES}

\begin{abstract}
We propose strongly consistent algorithms for reconstructing the characteristic
function $1_K$ of an unknown convex body $K$ in $\R^n$ from possibly noisy
measurements of the modulus of its Fourier transform $\widehat{1_K}$.  This
represents a complete theoretical solution to the Phase Retrieval Problem for
characteristic functions of convex bodies.  The approach is via the closely
related problem of reconstructing $K$ from noisy measurements of its
covariogram, the function giving the volume of the intersection of $K$ with its
translates. In the many known situations in which the covariogram determines a
convex body, up to reflection in the origin and when the position of the body
is fixed, our algorithms use $O(k^n)$ noisy covariogram measurements to
construct a convex polytope $P_k$ that approximates $K$ or its reflection $-K$
in the origin. (By recent uniqueness results, this applies to all planar convex
bodies, all three-dimensional convex polytopes, and all symmetric and most (in
the sense of Baire category) arbitrary convex bodies in all dimensions.) Two
methods are provided, and both are shown to be strongly consistent, in the
sense that, almost surely, the minimum of the Hausdorff distance between $P_k$
and $\pm K$ tends to zero as $k$ tends to infinity.
\end{abstract}

\section{Introduction}\label{intro}

The {\em Phase Retrieval Problem} of Fourier analysis involves determining a
function $f$ on $\R^n$ from the modulus $|\widehat{f}\,|$ of its Fourier
transform $\widehat{f}$. This problem arises naturally and frequently in
various areas of science, such as X-ray crystallography, electron microscopy,
optics, astronomy, and remote sensing, in which only the magnitude of the
Fourier transform can be measured and the phase is lost. (Sometimes, as when
reconstructing an object from its far-field diffraction pattern, it is the
squared modulus $|\widehat{f}\,|^2$ that is directly measured.) In 1984,
Rosenblatt \cite{Ros84} wrote that the Phase Retrieval Problem ``arises in all
experimental uses of diffracted electromagnetic radiation for determining the
intrinsic detailed structure of a diffracting object."  Today, the word ``all" is perhaps too strong in view of recent advances in coherent diffraction imaging. In any case, the literature is vast; see the surveys \cite{Hur89},
\cite{KST}, \cite{LBL}, and \cite{Ros84}, as well as the articles \cite{BSV}
and \cite{Fie82} and the references given there.

Phase retrieval is fundamentally under-determined without additional
constraints, which usually take the form of an a priori assumption that $f$ has
a particular support or distribution of values.  An important example is when
$f=1_K$, the characteristic function of a convex body $K$ in $\R^n$.  In this
setting, phase retrieval is very closely related to a geometric problem
involving the {\em covariogram} of a convex body $K$ in $\R^n$.  This is the
function $g_K$ defined by $$g_K(x)=V_n\left(K\cap (K+x)\right),$$ for $x\in
\R^n$, where $V_n$ denotes $n$-dimensional Lebesgue measure and $K+x$ is the
translate of $K$ by the vector $x$.  It is also sometimes called the {\em set
covariance} and is equal to the {\em autocorrelation} of $1_K$, that is,
$$g_K=1_K\ast 1_{-K},$$
where $\ast$ denotes convolution and $-K$ is the reflection of $K$ in the
origin. Taking Fourier transforms, we obtain the relation
\begin{equation}\label{PRC}
\widehat{g_K}=\widehat{1_K}\,\widehat{1_{-K}}=\widehat{1_K}\,
\overline{\widehat{1_{K}}} =\big{|}\widehat{1_K}\big{|}^2.
\end{equation}
This connects the Phase Retrieval Problem, restricted to characteristic
functions of convex bodies, to the problem of determining a convex body from
its covariogram.  Both the definition of covariogram and this connection extend
to arbitrary measurable sets, but the reason for restricting to convex bodies
will become clear.

The covariogram was introduced by Matheron in his book \cite{M} on random sets.
He showed that for a fixed $u\in S^{n-1}$, the directional derivatives
$\partial g_K(tu)/\partial t$, for all $t>0$, of the covariogram of a convex
body $K$ in $\R^n$ yield the distribution of the lengths of all chords of $K$
parallel to $u$. This explains the utility of the covariogram in fields such as
stereology, geometric tomography, pattern recognition, image analysis, and
mathematical morphology, where information about an unknown object is to be
retrieved from chord length measurements; see, for example, \cite{CabB03},
\cite{Gar95a}, and \cite{Ser82}. The covariogram has also played an
increasingly important role in analytic convex geometry.  For example, it was
used by Rogers and Shephard in proving their famous difference body inequality
(see \cite[Theorem~7.3.1]{Sch93}), by Gardner and Zhang \cite{GZ} in the theory
of radial mean bodies, and by Tsolomitis \cite{T1} in his study of convolution
bodies, which via the work of Schmuckenschl\"{a}ger \cite{Schm92} and Werner
\cite{Wer99} allows a covariogram-based definition of the fundamental notion of
affine surface area.

Here we effectively solve the following three problems.  In each, $K$ is a
convex body in $\R^n$.

\noindent{\bf Problem 1 (Reconstruction from covariograms)}.  Construct an
approximation to $K$ from a finite number of noisy (i.e., taken with error)
measurements of $g_K$.

\noindent{\bf Problem 2 (Phase retrieval for characteristic functions of convex
bodies: squared modulus)}.  Construct an approximation to $K$ (or,
equivalently, to $1_K$) from a finite number of noisy measurements of
$|\widehat{1_K}|^2$.

\noindent{\bf Problem 3 (Phase retrieval for characteristic functions of convex
bodies: modulus)}. Construct an approximation to $K$ from a finite number of
noisy measurements of $|\widehat{1_K}|$.

In order to discuss our results, we must first address the corresponding
uniqueness problems.  In view of (\ref{PRC}), these are equivalent, so we shall
focus on the covariogram. It is easy to see that $g_K$ is invariant under
translations of $K$ and reflection of $K$ in the origin.  Matheron \cite{Mat86}
asked the following question, known as the {\em Covariogram Problem},
to which he conjectured an affirmative answer when $n=2$.

{\em Is a convex body in $\R^n$ determined, among all convex bodies and up to
translation and reflection in the origin, by its covariogram?}

The focus on covariograms of convex bodies is natural.  One reason is that
Mallows and Clark \cite{MalC70} constructed non-congruent convex polygons whose
overall chord length distributions (allowing the directions of the chords to
vary as well) are equal, thereby answering a related question of Blaschke.
Thus the information provided by the covariogram cannot be weakened too much.
Moreover, there exist non-congruent non-convex polygons, even (see
\cite[p.~394]{GGZ05}) horizontally- and vertically-convex polyominoes, with the
same covariogram, indicating that the convexity assumption also cannot be
significantly weakened.

Interest in the Covariogram Problem extends far beyond geometry. For example,
Adler and Pyke \cite{AdlP91} ask whether the distribution of the difference
$X-Y$ of independent random variables $X$ and $Y$, uniformly distributed over a
convex body $K$, determines $K$ up to translations and reflection in the
origin.  Up to a constant, the convolution $1_K\ast 1_{-K}=g_K$ is just the
probability density of $X-Y$, so the question is equivalent to the Covariogram
Problem.  In \cite{AdlP97}, the Covariogram Problem also appears in deciding
the equivalence of measures induced by Brownian processes for different base
sets.

A detailed historical account of the covariogram problem may be found in \cite{AveB}.  The current status is as follows, in which ``determined" always means determined by the covariogram among all convex bodies, up to translation and reflection in the origin.  Averkov and Bianchi \cite{AveB} showed that planar convex bodies are determined, thereby confirming Matheron's conjecture.  Bianchi \cite{Bia} proved, by a long and intricate argument, that three-dimensional convex polyhedra are determined. It is easy to see that centrally symmetric convex bodies are determined. (In the symmetric case, convexity is not essential; see \cite[Proposition~4.4]{GGZ05} for this result, due to Cabo and Jensen.) Goodey, Schneider, and Weil \cite{GSW} proved that most (in the sense of Baire category) convex bodies in $\R^n$ are determined. Nevertheless, the Covariogram Problem in general has a negative answer, as Bianchi \cite{Bia05} demonstrated by constructing convex polytopes in $\R^n$, $n\ge 4$, that are not determined. It is still unknown whether convex bodies in $\R^3$ are determined.

None of the above uniqueness proofs provide a method for actually
reconstructing a convex body from its covariogram. We are aware of only two
papers dealing with the reconstruction problem: Schmitt \cite{Schm93} gives an
explicit reconstruction procedure for a convex polygon when no pair of its
edges are parallel, an assumption removed in an algorithm due to Benassi and
D'Ercole \cite{BenD07}.  In both these papers, all the exact values of the
covariogram are supposed to be available.

In contrast, our first set of algorithms take as input only a finite number of
values of the covariogram of an unknown convex body $K_0$.  Moreover, these
measurements are corrupted by errors, modeled by zero mean random variables with uniformly bounded $p$th moments, where $p$ is at most six and usually four.  It is assumed that $K_0$ is determined by its
covariogram, has its centroid at the origin, and is contained in a known
bounded region of $\R^n$, which for convenience we take to be the unit cube
$C_0^n=[-1/2,1/2]^n$. We provide two different methods for reconstructing, for
each suitable $k\in\N$, a convex polytope $P_k$ that approximates $K_0$ or its
reflection $-K_0$.  Each method involves two algorithms, an initial algorithm
that produces suitable outer unit normals to the facets of $P_k$, and a common
main algorithm that goes on to actually construct $P_k$.

In the first method, the covariogram of $K_0$ is measured, {\em multiple times}, at
the origin and at vectors $(1/k)u_i$, $i=1,\dots,k$, where the $u_i$'s are
mutually nonparallel unit vectors that span $\R^n$.  From these measurements,
the initial Algorithm~NoisyCovBlaschke constructs an $o$-symmetric convex
polytope $Q_k$ that approximates $\nabla K_0$, the so-called Blaschke body of
$K_0$. (See Section~\ref{prelim} for definitions and notation.) The crucial
property of $\nabla K_0$ is that when $K_0$ is a convex polytope, each of its
facets is parallel to some facet of $\nabla K_0$. It follows that the outer
unit normals to the facets of $P_k$ can be taken to be among those of $Q_k$.
Algorithm~NoisyCovBlaschke utilizes the known fact that $-\partial
g_{K_0}(tu)/\partial t$, evaluated at $t=0$, equals the brightness function
value $b_{K_0}(u)$, that is, the $(n-1)$-dimensional volume of the orthogonal
projection of $K_0$ in the direction $u$.  This connection allows most of the
work to be done by a very efficient algorithm, Algorithm~NoisyBrightLSQ, designed earlier by Gardner and Milanfar (see \cite{GKM06}) for reconstructing a $o$-symmetric convex body from finitely many noisy measurements of its brightness function.

The second method achieves the same goal with a quite different approach.  This
time the covariogram of $K_0$ is measured once at each point in a cubic array
in $2C_0^n=[-1,1]^n$ of side length $1/k$.  From these measurements, the
initial Algorithm~NoisyCovDiff($\varphi$) constructs an $o$-symmetric convex
polytope $Q_k$ that approximates $DK_0=K_0+(-K_0)$, the difference body of
$K_0$.  The set $DK_0$ has precisely the same property as $\nabla K_0$, that
when $K_0$ is a convex polytope, each of its facets is parallel to some facet
of $DK_0$. Furthermore, $DK_0$ is just the support of $g_{K_0}$.  The
known property that $g_{K_0}^{1/n}$ is concave (a consequence of the
Brunn-Minkowski inequality \cite[Section~11]{Gar02}) can therefore be combined
with techniques from multiple regression.  Algorithm~NoisyCovDiff($\varphi$)
employs a Gasser-M\"{u}ller type kernel estimator for $g_{K_0}$, with suitable
kernel function $\varphi$, bandwidth, and threshold parameter.

The output $Q_k$ of either initial algorithm forms part of the input to the
main common Algorithm~NoisyCovLSQ.  The covariogram of $K_0$ is now measured
{\em again}, once at each point in a cubic array in $2C_0^n=[-1,1]^n$ of side length
$1/k$.  Using these measurements, Algorithm~NoisyCovLSQ finds a convex polytope
$P_k$, each of whose facets is parallel to some facet of $Q_k$, whose
covariogram fits best the measurements in the least squares sense.

Much effort is spent in proving that these algorithms are strongly consistent.
Whenever $K_0$ is determined among convex bodies, up to translation and reflection in the origin, by its covariogram, we show that, almost surely,
$$\min\{\delta(K_0,P_k), \delta(-K_0,P_k)\}\rightarrow 0$$
as $k\rightarrow \infty$, where $\delta$ denotes Hausdorff distance.  (If
$K_0$ is not so determined, a rare situation in view of the uniqueness results
discussed above, the algorithms still construct a sequence $(P_k)$ whose
accumulation points exist and have the same covariogram as $K_0$.) From a
theoretical point of view, this completely solves Problem~1. Naturally, the
consistency proof leans heavily on results and techniques from analytic convex
geometry, as well as a suitable version of the Strong Law of Large
Numbers.  Some effort has been made to make the proof fairly self-contained, but some arguments from the proof from \cite{GKM06} that Algorithm~NoisyBrightLSQ is strongly consistent are used in proving that Algorithm~NoisyCovBlaschke is strongly consistent. One such argument rests on the Bourgain-Campi-Lindenstrauss stability result for projection bodies.

With algorithms for Problem~1 in hand, we move to Problem~2, assuming that
$K_0$ is an unknown convex body satisfying the same conditions as before. The
basic idea is simple enough: Use (\ref{PRC}) and the measurements of
$|\widehat{1_{K_0}}|^2$ at points in a suitable cubic array to approximate
$g_{K_0}$ via its Fourier series, and feed the resulting values into the
algorithms for Problem~1.  However, two major technical obstacles arise.  The
new estimates of $g_{K_0}$ are corrupted by noise that now involves dependent
random variables, and a new deterministic error appears as well.  A substitute
for the Strong Law of Large Numbers must be proved, and the deterministic error
controlled using Fourier analysis and the fortunate fact that $g_{K_0}$ is
Lipschitz.  In the end the basic idea works, assuming that for suitable
$1/2<\gamma<1$, measurements of $|\widehat{1_{K_0}}|^2$ are taken at the points
in $(1/k^{\gamma})\Z^n$ contained in the cubic window
$[-k^{1-\gamma},k^{1-\gamma}]^n$, whose size increases with $k$ at a rate
depending on the parameter $\gamma$. The three resulting algorithms,
Algorithm~NoisyMod$^2$LSQ, Algorithm~NoisyMod$^2$Blaschke, and
Algorithm~NoisyMod$^2$Diff($\varphi$), are stated in detail and, with suitable
restrictions on $\gamma$, proved to be strongly consistent under the same
hypotheses as for Problem~1.

Our final three algorithms, Algorithm~NoisyModLSQ, Algorithm~NoisyModBlaschke,
and Algorithm~NoisyModDiff($\varphi$) cater for Problem~3.  Again there is a
basic simple idea, namely, to take two independent measurements at each of the
points in the same cubic array as in the previous paragraph, multiply the two,
and feed the resulting values into the algorithms for Problem~2.  No serious
extra technical difficulties arise, and we are able to prove that the three new
algorithms are strongly consistent under the same hypotheses as for Problem~2.
This provides a complete theoretical solution to the Phase Retrieval Problem
for characteristic functions of convex bodies.

To summarize:

\noindent $\bullet$ For Problem~1, {\em first} use {\em either} Algorithm~NoisyCovBlaschke {\em or}
Algorithm~NoisyCovDiff($\varphi$) and {\em then} use Algorithm~NoisyCovLSQ.

\noindent $\bullet$ For Problem~2, {\em first} use {\em either} Algorithm~NoisyMod$^2$Blaschke {\em or}
Algorithm~NoisyMod$^2$Diff($\varphi$) and {\em then} use Algorithm~NoisyMod$^2$LSQ.

\noindent $\bullet$ For Problem~3, {\em first} use {\em either} Algorithm~NoisyModBlaschke {\em or}
Algorithm~NoisyModDiff($\varphi$) and {\em then} use Algorithm~NoisyModLSQ.

These results can also be viewed as a contribution to the literature on the associated uniqueness problems. They show that if a convex body is determined, up to translation and reflection in the origin, by its covariogram, then it is also so determined by its values at certain countable sets of points, even, almost surely, when these values are contaminated with noise.  Similarly, the characteristic function of such a convex body is also determined by certain countable sets of noisy values of the modulus of its Fourier transform.

Our noise model is sufficiently general to apply to all the main cases of practical interest: zero mean Gaussian noise, Poisson noise (unbiased measurements following a Poisson distribution, sometimes called shot noise), or Poisson noise plus zero mean Gaussian noise.  However, the main text of this paper deals solely with theory.  With the exception of Corollary~\ref{cordiffe} and Remark~\ref{cordifferem}, where the method of proof leads naturally to rates of convergence for Algorithm~NoisyCovDiff($\varphi$) and hence for the two related algorithms for phase retrieval, the focus is entirely on strong consistency.  Further remarks about convergence rates, sampling designs, and implementation issues have been relegated to the Appendix. Much remains to be done.  We believe, however, that our algorithms will find applications.  For example, Baake and Grimm \cite{BaaG07} explain how the problem of finding the atomic structure of a quasicrystal from its X-ray diffraction image involves recovering a subset of $\R^n$ called a window from its covariogram, and note that this window is in many cases a convex body.

We are grateful to Jim Fienup, David Mason, and Sara van de Geer for helpful correspondence and to referees for some insightful suggestions that led to significant improvements.

\vspace{-.5in}

\section{Guide to the paper}\label{guide}

\begin{tabular}{p{0in}  p{6in}}
\hspace{-.25in}\S\ref{prelim}.& {\em Definitions, notation, and preliminary results}.\\
& We recommend that the reader skip this section and refer back to it when necessary.\\
\hspace{-.25in}\S\ref{convergence}.& {\em The main algorithm for reconstruction from covariograms}.\\
& This presents the main ({\em second} stage) Algorithm~NoisyCovLSQ for Problem~1 and its strong consistency, established in Theorem~\ref{maincov}.\\
\hspace{-.25in}\S\ref{Algorithm}.&{\em Approximating the Blaschke body via the covariogram}.\\
& The first of the two first-stage algorithms for Problem~1, Algorithm~NoisyCovBlaschke, is stated with proof of strong consistency in Theorem~\ref{covB}. The latter requires the assumption that the vectors $u_i$, $i=1,\dots,k$, are part of an infinite sequence $(u_i)$ that is in a sense evenly spread out in $S^{n-1}$, but this is a weak restriction.\\
\hspace{-.25in}\S\ref{diffe}.& {\em Approximating the difference body via the covariogram}.\\
& In this section, the second of the two first-stage algorithms for Problem~1, Algorithm~NoisyCovDiff($\varphi$), is set out and proved to be strongly consistent in Theorem~\ref{maindiffe}.\\
\hspace{-.25in}\S\ref{PRI}.& {\em Phase retrieval: Framework and technical lemmas}.\\
& Necessary material from Fourier analysis is gathered, and the scene is set for results on phase retrieval.  This does not depend on the previous three sections.\\
\hspace{-.25in}\S\ref{PRII}.& {\em Phase retrieval from the squared modulus}.\\
& The algorithms for Problem~2, Algorithm~NoisyMod$^2$LSQ, Algorithm~NoisyMod$^2$Blaschke, and Algorithm~NoisyMod$^2$Diff($\varphi$) are presented and strong consistency theorems for them are proved.\\
\hspace{-.25in}\S\ref{PRIII}.& {\em Phase retrieval from the modulus}.\\
& The corresponding algorithms for Problem~3, Algorithm~NoisyModLSQ, Algorithm~NoisyModBlaschke, and Algorithm~NoisyModDiff($\varphi$), are presented and shown to be strongly consistent.\\
\hspace{-.3in}\S\ref{Appendix}.& {\em Appendix}.\\
& Rates of convergence and implementation issues are discussed.\\
\end{tabular}

\section{Definitions, notation, and preliminary results}\label{prelim}

\subsection{Basic definitions and notation}

As usual, $S^{n-1}$ denotes the unit sphere, $B^n$ the unit ball, $o$ the
origin, and $|\cdot|$ the norm in Euclidean $n$-space $\R^n$.  It is assumed throughout that $n\ge 2$. We shall also
write $C_0^n=[-1/2,1/2]^n$ throughout. The standard orthonormal basis for
$\R^n$ will be denoted by $\{e_1,\dots,e_n\}$. A {\it direction} is a unit
vector, that is, an element of $S^{n-1}$. If $u$ is a direction, then $u^{\perp
}$ is the $(n-1)$-dimensional subspace orthogonal to $u$ and $l_u$ is the line
through the origin parallel to $u$.  If $x,y\in\R^n$, then $x\cdot y$ is the
inner product of $x$ and $y$, and $[x,y]$ is the line segment with endpoints
$x$ and $y$.

We denote by $\partial A$,  $\inte A$,  $\diam A$, and $1_A$ the {\it boundary}, {\it
interior}, \emph{diameter}, and {\em characteristic function} of a set $A$, respectively. The notation for the usual (orthogonal) {\it projection} of $A$ on a subspace $S$ is $A|S$. A set is {\it $o$-symmetric} if it is centrally symmetric, with center at the origin.

If $X$ is a metric space and $\ee>0$, a finite set $\{x_1,\dots,x_m\}$ is
called an {\em $\ee$-net} in $X$ if for every point $x$ in $X$, there is an
$i\in\{1,\dots,m\}$ such that $x$ is within a distance $\ee$ of $x_i$.

We write $V_k$ for $k$-dimensional Lebesgue measure in $\R^n$, where
$k=1,\dots, n$, and where we identify $V_k$ with $k$-dimensional Hausdorff
measure. If $K$ is a $k$-dimensional convex subset of $\R^n$, then $V(K)$ is
its {\em volume} $V_k(K)$. Define $\kappa_n=V(B^n)$. The notation $dz$ will
always mean $dV_k(z)$ for the appropriate  $k=1,\dots, n$.

If $E$ and $F$ are sets in $\R^n$, then
$$E+F=\{x+y: x\in E, y\in F\}$$
denotes their {\em Minkowski sum} and
\begin{equation}\label{MD}
E\ominus F=\{x\in \R^n: F+x\subset E\}
\end{equation}
their {\em Minkowski difference}.

We adopt a standard definition of the Fourier transform $\widehat{f}$ of a
function $f$ on $\R^n$, namely
$$\widehat{f}(x)=\int_{\R^n}f(y)e^{-ix\cdot y}\,dy.$$

If $f$ and $g$ are real-valued functions on $\N$, then, as usual, $f=O(g)$
means that there is a constant $c$ such that $f(k)\le cg(k)$ for sufficiently
large $k$.  The notation $f\sim g$ will mean that $f=O(g)$ and $g=O(f)$.

\subsection{Convex geometry}

Let ${\mathcal K}^n$ be the class of compact convex sets in $\R^n$, and let ${\mathcal K}^n(A)$ be the subclass of members of ${\mathcal K}^n$ contained in the subset $A$ of $\R^n$. A {\em convex body} in $\R^n$ is a compact convex set with nonempty interior. The notation ${\mathcal K}^n(r,R)$ will be used for the class of convex bodies
containing $rB^n$ and contained in $RB^n$, where $0<r<R$. The treatise of
Schneider \cite{Sch93} is an excellent general reference for convex geometry.

Figures illustrating many of the following definitions can be found in \cite{Gar95a}.

If $K \in {\mathcal K}^n$, then
$$K^*=\{x\in\R^n:x\cdot y\le 1~{\text{for all}}~y\in K\}$$
is the {\em polar set} of $K$.  The function
$$h_K(x)=\max\{x\cdot y: y\in K\},$$
for $x\in\R^n$, is the {\it support function} of $K$ and
$$b_K(u)=V(K|u^{\perp}),$$
for $u\in S^{n-1}$, its {\it brightness function}.  Any $K \in {\mathcal K}^n$
is uniquely determined by its support function. We can regard $h_K$ as a
function on $S^{n-1}$, since $h_K(x)=|x|h_K(x/|x|)$ for $x\neq o$. The {\em
Hausdorff distance} $\delta(K,L)$ between two sets $K,L\in {\mathcal K}^n$ can
then be conveniently defined by
$$\delta(K,L)=\|h_K-h_L\|_{\infty},$$
where $\|\cdot\|_{\infty}$ denotes the supremum norm on $S^{n-1}$.  Equivalently, one can define
$$\delta(K,L)=\min\{\ee\ge 0: K\subset L+\ee B^n,~ L\subset K+\ee B^n\}.$$

The {\em surface area measure} $S(K,\cdot)$ of a convex body $K$ is defined for
Borel subsets $E$ of $S^{n-1}$ by
$$
S(K,E)=V_{n-1}\left(g^{-1}(K,E)\right),
$$
where $g^{-1}(K,E)$ is
the set of points in $\partial K$ at which there is an outer unit normal vector
in $E$. Let $S(K)=S(K,S^{n-1})$. Then $S(K)$ is the {\em surface area} of $K$. The {\em
Blaschke body} $\nabla K$ of a convex body $K$ is the unique $o$-symmetric
convex body satisfying
\begin{equation}\label{BB}
S(\nabla K,\cdot)=\frac12 S(K,\cdot)+\frac12 S(-K,\cdot).
\end{equation}
The {\it projection body} of $K\in \cK^n$ is the $o$-symmetric set $\Pi K\in
\cK^n$ defined by \begin{equation}\label{projbod} h_{\Pi K}=b_K.
\end{equation}
Cauchy's projection formula states that for any $u\in S^{n-1}$,
\begin{equation}\label{Cauchy}
h_{\Pi K}(u)=b_K(u)=\frac12\int_{S^{n-1}}|u\cdot v|\,dS(K,v),
\end{equation}
and Cauchy's surface area formula is
\begin{equation}\label{Cauchynew}
S(K)=\frac{1}{\kappa_{n-1}} \int_{S^{n-1}} b_K(u)du;
\end{equation}
see \cite[(A.45) and (A.49), p.~408]{Gar95a}. By (\ref{BB}) and (\ref{Cauchy}),
we have
\begin{equation}\label{newBB}
b_{\nabla K}=b_K,
\end{equation}
and it can be shown (see \cite[p.~116]{Gar95a}) that $\nabla K$ is the unique
$o$-symmetric convex body with this property.

The {\em difference body} of $K$ is the $o$-symmetric convex body
$DK=K+(-K)$.

\subsection{The covariogram}

The function
$$g_K(x) = V(K\cap (K+x)),$$
for $x\in\R^n$, is called the {\it covariogram} of $K$. Note that
$g_K(o)=V(K)$, and that we have $g_K(x)=0$  if and only if $x\notin \inte DK$,
so the support of $g_K$ is $DK$. Also, $g_K^{1/n}$ is concave on its support;
see, for example, \cite[Lemma~3.2]{GZ}.

Let $K$ be a convex body in $\R^n$ and let $u\in S^{n-1}$. The {\it (parallel)
X-ray of $K$} in the direction $u$ is the function $X_uK$ defined by
$$X_uK(x)=\int_{l_u+x}1_K(y)dy,$$
for $x\in u^{\perp}$. Now define
\begin{equation}\label{ek}E_K(t,u) =\{y\in u^\perp
:X_uK(y)\geq t\}
\end{equation} and
\begin{equation}\label{ak}
a_K(t,u)=V\bigl(E_K(t,u)\bigr),
\end{equation}
for $t\ge 0$ and $u\in S^{n-1}$.  Note that if $u\in S^{n-1}$, then
$E_K(0,u)=K|u^{\perp}$ and $a_K(0,u)=b_K(u)$.

Let $x=tu$, where $t\ge 0$ and $u\in S^{n-1}$, and define $g_K(t,u)=g_K(tu)$.
The simple relationship
\begin{align}\label{simple}
g_K(t,u)=\int_t^\infty a_K(s,u)\,ds
\end{align}
was noticed by Matheron \cite[p.~86]{M} in the form
$$
\frac{\partial g_K(t,u)}{\partial t}=-a_K(t,u),
$$
which also yields
$$\left.\frac{\partial g_K(t,u)}{\partial t}\right|_{t=0}=-b_K(u).$$
(Note that the partial derivative here is one-sided; $g_K$ is not differentiable at the origin.)

\begin{lem}\label{unif1}
Let $r>0$ and let $K$ be a convex body with $rB^n\subset K$. If $0<t\le 2r$,
then
\begin{equation}\label{estimDiffq}
\left(1-\frac{t}{2r}\right)^{n-1} b_K\left(u\right)\le \frac{g_K(o)-g_K(tu)}{t}
\le b_K\left(u\right),
\end{equation}
for all $u\in S^{n-1}$.
\end{lem}

\begin{proof}
Let $u\in S^{n-1}$.  By \eqref{simple}, we have
$$g_K(o)-g_K(tu)=\int_0^ta_K(s,u)\,ds.$$
From this and the fact that $a_K(\cdot,u)$ is decreasing, we obtain
\begin{equation}\label{99}
a_K(t,u)\le\frac{g_K(o)-g_K(tu)}{t}\le a_K(0,u)=b_K(u).
\end{equation}
The set
$$M=\conv\left((K|u^\perp)\cup [-ru,ru]\right)$$
is generally not a subset of $K$, but elementary geometry using
$[-ru,ru]\subset K$ and (\ref{ek}) gives
$$\left(1-\frac{t}{2r}\right)\left(K|u^\perp\right)=E_M(t,u)\subset E_K(t,u).$$
Taking the $(n-1)$-dimensional volumes of these sets and using (\ref{ak})
yields
$$\left(1-\frac{t}{2r}\right)^{n-1}b_K(u) \le a_K(t,u).$$
The lemma follows from the previous inequality and \eqref{99}.
\end{proof}

An inequality similar to \eqref{estimDiffq} was derived in \cite[Theorem
1]{kide:jens03} for $n=2$.

Matheron \cite[p.~2]{Mat86} showed that the covariogram of a convex body is a
Lipschitz function. For the convenience of the reader, we provide a
proof of this fact based on \cite{Gal10}, which yields the  optimal
Lipschitz constant.

\begin{prop}\label{Matprop}
If $K$ is a convex body in $\R^n$ and $x,y\in \R^n$, then
$$
|g_K(x)-g_K(y)|\le \max_{u\in S^{n-1}}b_K(u)|x-y|.
$$
\end{prop}

\begin{proof}
We have
$$
\left(K\cap (K+x)\right)\setminus \left(K\cap (K+y)\right)\subset
(K+x)\setminus (K+y).
$$
This implies
\begin{align*}
  V_n\left(K\cap (K+x)\right)-V_n\left(K\cap (K+y)\right)&\le
  V_n\left(K\setminus(K+y-x)\right)\\&=V_n(K)-V_n\left(K\cap(K+y-x)\right).
\end{align*}
Equivalently,
$g_K(x)-g_K(y)\le g_K(o)-g_K(y-x)=g_K(o)-g_K(x-y)$, and interchanging $x$ and $y$
yields
$$|g_K(x)-g_K(y)|\le g_K(o)-g_K(x-y).$$
Using this and the right-hand inequality
in (\ref{estimDiffq}), we get
$$|g_K(x)-g_K(y)|\le b_K\!\left(\frac{x-y}{|x-y|}\right)|x-y|,$$
and the proposition follows immediately.
\end{proof}

\begin{cor}\label{lipcor}
If $K_0\subset C_0^n$ is a convex body, then for all $x,y\in \R^n$,
$$
|g_{K_0}(x)-g_{K_0}(y)|\le \sqrt{n}|x-y|.
$$
\end{cor}

\begin{proof}
Since $K_0\subset C_0^n$, Proposition~\ref{Matprop} yields
$$
|g_K(x)-g_K(y)|\le \max_{u\in S^{n-1}}b_{C_0^n}(u)|x-y|.
$$
By Cauchy's projection formula (\ref{Cauchy}), for $u=(u_1,u_2,\dots,u_n)\in
S^{n-1}$ we have
$$b_{C_0^n}(u)=V\left(C_0^n|u^{\perp}\right)=\sum_{i=1}^n|u_i|,$$
from which it is easy to see that $b_{C_0^n}(u)\le \sqrt{n}$.
\end{proof}

\subsection{Miscellaneous definitions}

Let $\mu$ and $\nu$ be finite nonnegative Borel measures in $S^{n-1}$.  Define
\begin{equation}\label{Pro}
d_P(\mu,\nu)=\inf\{\ee>0:\mu(E)\le\nu(E_{\ee})+\ee,\
\nu(E)\le\mu(E_{\ee})+\ee,\ E \text{ Borel in } S^{n-1}\},
\end{equation}
where
$$E_\ee=\{u\in S^{n-1}:\exists v\in E: |u-v|<\ee\}.$$
Then $d_P$ is a metric called the {\em Prohorov metric}. As $S^{n-1}$ is a
Polish space, it is enough to take the infimum in \eqref{Pro} over the class of
\emph{closed} sets. In addition, if $\mu(S^{n-1})=\nu(S^{n-1})$, then
\begin{equation}\label{Pro1}
d_P(\mu,\nu)=\inf\{\ee>0:\mu(E)\le\nu(E_{\ee})+\ee,\ E \text{ Borel in }
S^{n-1}\};
\end{equation}
see \cite{Dud}.

We need a condition on a sequence $(u_i)$ in $S^{n-1}$ stronger than denseness
in $S^{n-1}$.  To this end, for $u\in S^{n-1}$ and $0<t\le 2$, let
$$C_t(u)=\{v\in S^{n-1}:|u-v|<t\}$$ be the open spherical cap
with center $u$ and radius $t$. We call $(u_i)$ {\em evenly spread} if for all
$0<t<2$, there is a constant $c=c(t)>0$ and an $N=N(t)$ such that
$$
|\{u_1,\dots,u_k\}\cap C_t(u)|\ge ck,
$$
for all $u\in S^{n-1}$ and $k\ge N$.
Often, we will apply this notion to the symmetrization
\[
  (u_i^*)=(u_1,-u_1,u_2,-u_2,u_3,-u_3,\ldots)
\]
of a sequence $(u_i)$.

Let $p\ge 1$. A family $\{X_{\alpha}:\alpha\in A\}$ of random variables has {\em uniformly bounded $p$th absolute moments} if there is a constant $C$ such that \begin{equation}\label{moments}
E\left(|X_{\alpha}|^p\right)\le C,
\end{equation}
for all $\alpha\in A$.  Of course, if $p$ is an even integer, we can and will omit the word ``absolute."  If $1\le q\le p$ and (\ref{moments}) holds, then it also holds with $p$ replaced by $q$ and $C$ replaced by $C^{q/p}$.

Triangular arrays of random variables of the form $\{X_{ik}: i=1,\dots,m_k; k\in \N\}$  (or, more generally, $\{X_{\alpha k}: \alpha\in A_k; k\in \N\}$) are called {\em row-wise independent} if for each $k$, the family $\{X_{ik}: i=1,\dots,m_k\}$ (or $\{X_{\alpha k}: \alpha\in A_k\}$, respectively) is independent.

\section{The main algorithm for reconstruction
from covariograms}\label{convergence}

We shall assume throughout that the unknown convex body $K_0$ is contained in
the cube $C_0=[-1/2,1/2]^n$, with its centroid at the origin. This assumption
can be justified on both purely theoretical and purely practical grounds. If
the measurements are exact, then from the covariogram, a convex polytope can be
constructed that contains a translate of $K_0$. On the other hand, in practise,
an unknown object whose covariogram is to be measured is contained in some
known bounded region.  In either case, one may as well suppose that $K_0$ is
contained in $C_0^n$, and since in the situations we consider, the covariogram
determines $K_0$ up to translation and reflection in the origin, we can also
fix the centroid at the origin.

We now state the main, second-stage algorithm. Note that it requires, as part of the input, an $o$-symmetric convex polytope that approximates either the Blaschke body $\nabla K_0$ or the difference body $DK_0$ of $K_0$.  These are provided by the first-stage algorithms, Algorithm~NoisyCovBlaschke and Algorithm~NoisyCovDiff($\varphi$), described in Sections~\ref{Algorithm} and~\ref{diffe}, respectively.

The reader should be aware that here, and throughout the paper, double subscripts in expressions such as $x_{ik}$, $M_{ik}$, $N_{ik}$, etc., represent triangular arrays. Thus, for a fixed $k$, the index $i$ varies over a finite set of integers that depends on $k$; and similarly when the first index is labeled by another letter in expressions such as $z_{jk}$, $X_{pk}$, and so on, or is itself represented by a double index, as in $N_{ijk}$.  Phrases such as ``the $N_{ik}$'s are row-wise independent" mean that the corresponding triangular array is row-wise independent, i.e., independent for fixed $k$.

\bigskip

{\large{\bf Algorithm~NoisyCovLSQ}}

\bigskip

{\it Input:} Natural numbers $n\ge 2$ and $k$; noisy covariogram measurements
\begin{equation}\label{meas}
M_{ik}=g_{K_0}(x_{ik})+N_{ik},
\end{equation}
of an unknown convex body $K_0\subset C_0^n$ whose centroid is at the origin,
at the points $x_{ik}$, $i=1,\dots,I_k=(2k+1)^n$ in the cubic array $2C_0^n\cap
(1/k)\Z^n$, where the $N_{ik}$'s are row-wise independent zero mean random
variables with uniformly bounded third absolute moments; an $o$-symmetric convex polytope $Q_k$ in $\R^n$, stochastically
independent of the measurements $M_{ik}$, that approximates either $\nabla K_0$
or $DK_0$, in the sense that, almost surely,
\begin{align}\label{QkdoesIt}
\lim_{k\to\infty} \delta(Q_k,\nabla K_0)=0,\quad\text{ or}\quad
\lim_{k\to\infty} \delta(Q_k,D K_0)=0.
\end{align}

\smallskip

{\it Task:} Construct a convex polytope $P_k$ that approximates $K_0$, up to
reflection in the origin.

\smallskip

{\it Action:}

1.  Compute the outer unit normals $\{\pm u_j: j=1,\dots,s\}$ to the facets of $Q_k$.

2.  For any vector $a=(a_1^+,a_1^-,a_2^+,a_2^-,\dots,a_s^+,a_s^-)$, where
$a_j^+, a_j^-\ge 0$, $j=1,\dots,s$, such that $\sum_{j=1}^s(a_j^+-a_j^-)u_j=o$,
let $P(a)=P(a_1^+,a_1^-,a_2^+,a_2^-,\dots,a_s^+,a_s^-)$ be the convex polytope
with centroid at the origin, facet outer unit normals in $\{\pm u_j: j=1,\ldots,s\}$
and such that the facet with normal $u_j$ (or $-u_j$) has $(n-1)$-dimensional
measure $a_j^+$ (or $a_j^-$, respectively), $j=1,\ldots,s$.

Solve the following  least squares problem:

\begin{equation}\label{obj1}\min
\sum_{i=1}^{I_k}\left(M_{ik}-g_{P(a)\cap C_0^n}(x_{ik})\right)^2
\end{equation}
over the variables $a_1^+,a_1^-,a_2^+,a_2^-,\dots,a_s^+,a_s^-$, subject to the
constraints
$$\sum_{j=1}^s(a_j^+-a_j^-)u_j=o$$ and
$$a_j^+, a_j^-\ge 0, ~~j=1,\dots,s.$$
These constraints guarantee that the output will correspond to a convex
polytope.

3. Let a set of optimal values be
$\hat{a}_1^+,\hat{a}_1^-,\hat{a}_2^+,\hat{a}_2^-,\dots,\hat{a}_s^+,\hat{a}_s^-$,
and call the corresponding polytope $P(\hat{a})$.  Then the output polytope
$P_k$ is the translate of $P(\hat{a})\cap C_0^n$ that has its centroid at the
origin. Note that in this case $-P_k$ also corresponds to a set of optimal
values obtained by switching $a_j^+$ and $a_j^-$, $j=1,\dots,s$.

\bigskip

\begin{lem}\label{polouter}
Let $0<r<R$ and let $Q\in {\mathcal K}^n(r,R)$ be an $o$-symmetric convex
polytope. Then there are facets of $Q$ with outer unit normals
$u_1,\dots,u_n$ such that
\begin{equation}\label{det}
|\det(u_1,\dots,u_n)|>(r/R)^{n(n-1)/2}.
\end{equation}
\end{lem}

\begin{proof}
The polar body $Q^*$ of $Q$ is contained in ${\mathcal K}^n(1/R,1/r)$ and has
its vertices in the directions of the outer unit normals to the facets of $Q$, so it
suffices to prove that there are vertices $v_1,\dots,v_n$ of $Q^*$ such that
with $u_i=v_i/|v_i|$, (\ref{det}) holds.

The proof will be by induction on $n$.  Let $n=2$.  We may assume that $Q^*$
has a vertex, $v_1$ say, on the positive $x_2$-axis.  Since $Q^*\in {\mathcal
K}^2(1/R,1/r)$, there must be another vertex $v_2$ of $Q^*$ with distance at
least $1/R$ from the $x_2$-axis, and by the symmetry of $Q^*$, such that also
$v_2\cdot e_2\ge 0$.  If $\alpha$ is the angle between $v_1$ and $v_2$, we must
then have $\theta\le\alpha\le\pi/2$, where $\theta$ is the angle between the
vectors $(0,1/r)$ and $\left(1/R,\sqrt{(1/r^2)-(1/R^2)}\right)$.  Then, if
$u_i=v_i/|v_i|$ for $i=1,2$, we have
$$|\det(u_1,u_2)|=\sin\alpha\ge \sin\theta=r/R,$$
which proves (\ref{det}) for $n=2$.

Suppose that (\ref{det}) holds with $n$ replaced by $n-1$ and let  $Q^*\in
{\mathcal K}^n(1/R,1/r)$.  We may assume that $Q^*$ has a vertex, $v_1$ say, on
the positive $x_n$-axis, so that $v_1/|v_1|=e_n$.  Since $Q^*|e_n^{\perp}\in
{\mathcal K}^{n-1}(1/R,1/r)$ (where we are identifying $e_n^{\perp}$ with
$\R^{n-1}$), by the inductive hypothesis, there are vertices $w_2,\dots,w_n$ of
$Q^*|e_n^{\perp}$ such that if $z_i=w_i/|w_i|$, $i=2,\dots,n$, then
\begin{equation}\label{detind}
|\det(z_2,\dots,z_n)|\ge (r/R)^{(n-1)(n-2)/2}.
\end{equation}
Let $v_i$ be a vertex of $Q^*$ such that $v_i|e_n^{\perp}=w_i$, $i=2,\dots,n$,
and let $u_i=v_i/|v_i|$, $i=1,\dots,n$.  By the symmetry of $Q^*$, we may also
assume that $v_i\cdot e_n\ge 0$ for $i=2,\dots,n$.  Let $\alpha_i$ be the angle
between $v_i$ and $w_i$, for $i=2,\dots,n$. Using the fact that
$Q^*|e_n^{\perp}\in {\mathcal K}^{n-1}(1/R,1/r)$, we see that each $v_i$,
$i=2,\dots,n$ has distance at least $1/R$ from the $x_n$-axis.  Therefore
$\cos\alpha_i\ge\sin\theta= r/R$ for $i=2,\dots,n$.  Then, using (\ref{detind})
and noting that $u_1=e_n$ and $u_i=u_i|e_n^{\perp}+(u_i\cdot e_n)e_n$ for
$i=2,\dots,n$, we obtain
\begin{eqnarray*}
|\det(u_1,\dots,u_n)|&=&|\det(u_2|e_n^{\perp},\dots,u_n|e_n^{\perp})|\\
&=& |\det(z_2,\dots,z_n)|\prod_{i=2}^n\cos\alpha_i\\
&\ge & (r/R)^{(n-1)(n-2)/2}(r/R)^{n-1}=(r/R)^{n(n-1)/2}.
\end{eqnarray*}
\end{proof}

\begin{lem}\label{inrad}
Let $K\in {\mathcal K}^n(r,R)$, let $0<\ee<\kappa_{n-1}r^{n-1}/2$, and let $L$
be a convex body containing the origin in $\R^n$ such that
\begin{equation}\label{as}
d_P(S(K,\cdot),S(L,\cdot))<\ee.
\end{equation}
Then there is a constant $a_1$ depending only on $\ee$, $r$, and $R$ such that
$L\subset a_1B^n$.  If $L$ is $o$-symmetric, there is also a constant $a_0>0$
depending only on $\ee$, $r$, and $R$ such that $a_0B^n\subset L$.
\end{lem}

\begin{proof}
Using (\ref{projbod}) and \eqref{Cauchy}, we obtain
\begin{equation}\label{rev21}
|h_{\Pi K}(u)-h_{\Pi L}(u)|=|b_K(u)-b_L(u)|\le d_D(S(K,\cdot),S(L,\cdot)).
\end{equation}
Here $d_D$ is the Dudley metric, defined by
$$
d_D(\mu,\nu)=\sup\left\{\left|\int_{S^{n-1}} f\,d(\mu-\nu)\right|:
\|f\|_{BL}\le 1\right\},
$$
where for any real-valued function $f$ on $S^{n-1}$ we define
$$\|f\|_{L}=\sup_{u\ne v}\frac{|f(u)-f(v)|}{|u-v|}\quad {\text{and}}\quad
\|f\|_{BL}=\|f\|_\infty+\|f\|_{L}.$$
(Note that for any $u\in S^{n-1}$, the function $f(v)=|u\cdot v|/2$, $v\in
S^{n-1}$ satisfies $\|f\|_{BL}= 1$.) By \cite[Corollary 2]{Dud}, we have
the relation
\begin{equation}\label{DudPro}
d_D(\mu,\nu)\le 2 d_P(\mu,\nu),
\end{equation} for finite nonnegative Borel
measures $\mu$ and $\nu$ in $S^{n-1}$.  Now (\ref{rev21}), (\ref{DudPro}), and (\ref{as}) yield
$$|h_{\Pi K}(u)-h_{\Pi L}(u)|\le
2d_P(S(K,\cdot),S(L,\cdot))<2\ee,$$ for each
$u\in S^{n-1}$.

Since $K\in {\mathcal K}^n(r,R)$, we have $\Pi K\in {\mathcal
K}^n\left(\kappa_{n-1}r^{n-1},\kappa_{n-1}R^{n-1}\right)$, so $\Pi L\in
{\mathcal K}^n(\kappa_{n-1}r^{n-1}-2\ee,\kappa_{n-1}R^{n-1}+2\ee)$. Now exactly
the same argument as in the proof of Lemma~4.2 of \cite{GarM03}, beginning with
formula (16) in that paper, yields the existence of $a_1$ and $a_0$.  (The
assumption of $o$-symmetry made in \cite{GarM03} is only needed for the latter.
Explicit values for $a_0$ and $a_1$ can be given in terms of $\ee$, $r$, and
$R$, but we do not need them here.)
\end{proof}

\begin{lem}\label{Prohorov}
Let $K$ be a convex body in $\R^n$.  Then there is an $\ee_0>0$ such that for
all $0<\ee<\ee_0$, if $Q$ is an $o$-symmetric convex polytope in $\R^n$ such
that either
\begin{equation}\label{Blasch}
d_P(S(\nabla K,\cdot),S(Q,\cdot))<\ee
\end{equation}
or
\begin{equation}\label{Diff}
d_P(S(DK,\cdot),S(Q,\cdot))<\ee,
\end{equation}
then there is a constant \refstepcounter{Cnmb}\label{ca} $c_{\arabic{Cnmb}}>0$
depending only on $K$ and a convex polytope $J$ whose facets are each parallel
to some facet of $Q$, such that
\begin{equation}\label{Blacon}
d_P(S(K,\cdot),S(J,\cdot))<c_{\ref{ca}}\ee.
\end{equation}
\end{lem}

\begin{proof}
We choose $\ee_0>0$ so that Lemma~\ref{inrad} holds when $\ee$ is replaced by
$\ee_0$ and $K$ is replaced by either $\nabla K$ or $DK$, as appropriate.  Let
$0<\ee<\ee_0$.

Let $\pm u_1,\dots,\pm u_s$ be the outer unit normals to the facets of $Q$ and
for $i=s+1,\dots,2s$, let $u_i=-u_{i-s}$. Set $I=\{1,\dots,2s\}$.

Suppose that (\ref{Blasch}) holds.  By (\ref{Pro}), $S(\nabla K,E)<
S(Q,E_{\ee})+\ee$ for each Borel subset $E$ of $S^{n-1}$. If $E_{\ee}\cap
\cup_{i\in I} \{u_i\}=\emptyset$, we have $S(Q,E_{\ee})=0$.  This implies that
$S(\nabla K,E)<\ee$ and so by (\ref{BB}),
\begin{equation}\label{smallDK}
S(K,E)<2\ee.
\end{equation}
If instead (\ref{Diff}) holds, then (\ref{Pro}) implies that $S(DK,E)<
S(Q,E_{\ee})+\ee$ for each Borel subset $E$ of $S^{n-1}$.  Then, if
$E_{\ee}\cap \cup_{i\in I}\{u_i\}=\emptyset$, we have $S(DK,E)<\ee$. By
\cite[(5.1.17), p.~275]{Sch93},
$$S(DK,E)=S(K+(-K),E)=S(K,E)+\sum_{j=1}^{n-1} {n-1 \choose j}
S(K,n-1-j;-K,j,E),$$ where $S(K,n-1-j;-K,j,\cdot)$ denotes the mixed area
measure of $n-1-j$ copies of $K$ and $j$ copies of $-K$. Since all these terms
are nonnegative, we obtain $S(K,E)<\ee$ and so (\ref{smallDK}) holds again.

For $i\in I$, let
$$V_i=\{u\in S^{n-1} : |u-u_i|\leq |u-u_j|\text{ for each $j\in I$, $j\neq i$}\}$$
be the Voronoi cell in $S^{n-1}$ containing $u_i$.  Choose Borel sets $W_i$
such that $\relint V_i\subset W_i\subset V_i$ for each $i$ and $W_i\cap
W_j=\emptyset$ for $i\neq j$, so that $\{W_i: i\in I\}$ forms a partition of
$S^{n-1}$.

Let $a_i=S(K,W_i)$ and let $w=\sum_{i\in I}a_iu_i$.  Since $S(K,\cdot)$ is
balanced, i.e.,
$$
\int_{S^{n-1}}u\,dS(K,u)=o,
$$
we have
\begin{eqnarray*}
w=\sum_{i\in I}a_iu_i&=&\sum_{i\in I}u_i\int_{W_i}\,dS(K,u)-\int_{S^{n-1}}u\,dS(K,u)\\
&=&\sum_{i\in I}\int_{W_i}(u_i-u)\,dS(K,u).
\end{eqnarray*}
For each $u\in S^{n-1}$ and $t>0$, let $C_{t}(u)=\{v\in S^{n-1}:|u-v|\le t\}$.
Let $W=\cup_{i\in I}(W_i\setminus C_{\ee}(u_i))$. Then $u_i\not\in W_{\ee}$ for
$i\in I$, so (\ref{smallDK}) implies that $S(K,W)<2\ee$.  Using this, we obtain
\begin{eqnarray}
|w|&=&\left|\sum_{i\in I} \int_{W_i\cap C_{\ee}(u_i)}(u_i-u)\,dS(K,u)+
\sum_{i\in I}\int_{W_i\setminus C_{\ee}(u_i)} (u_i-u)\,dS(K,u)\right|\nonumber\\
&\le &\sum_{i\in I} \int_{W_i\cap C_{\ee}(u_i)}|u_i-u|\,dS(K,u)+2
\int_{W} d S(K,u)\nonumber\\
&<&\ee S(K,S^{n-1})+4\ee=(S(K)+4)\ee.\label{w}
\end{eqnarray}

Since $Q$ is $o$-symmetric, we can apply Lemma~\ref{inrad} (with $K$ and $L$
replaced by $\nabla K$ (or $DK$) and $Q$, respectively) and
Lemma~\ref{polouter} to conclude that there exist outer unit normals
$u_{i_1},\dots,u_{i_n}$ to facets of $Q$ such that
\refstepcounter{Cnmb}\label{cr0}
$|\det(u_{i_1},\dots,u_{i_n})|>{c_{\arabic{Cnmb}}}$, where $c_{\ref{cr0}}$ depends
only on $K$.  In particular, $u_{i_1},\dots,u_{i_n}$ forms a basis for $\R^n$,
so there exist real numbers $b_{i_1},\dots,b_{i_n}$ such that
$$-w=\sum_{j=1}^nb_{i_j}u_{i_j}.$$
Replacing $u_{i_j}$ by $-u_{i_j}$, if necessary, we may assume that $b_{i_j}>0$
for $j=1,\dots,n$. By Cramer's rule, we obtain $b_{i_j}\le
|w|/|\det(u_{i_1},\dots,u_{i_n})|<|w|/c_{\ref{cr0}}$, for $j=1,\dots,n$. Define
$b_i=0$ for each $i\in I$ such that $i\not\in\{i_1,\dots,i_n\}$. Then, by
(\ref{w}), \refstepcounter{Cnmb}\label{cr1}
\begin{equation}\label{b}
\sum_{i\in I}b_i\le n|w|/c_{\ref{cr0}}<c_{\arabic{Cnmb}}\ee,
\end{equation}
where $c_{\ref{cr1}}$ depends only on $K$.

Let
$$\mu_0=\sum_{i\in I}a_i\delta_{u_i}~~{\text{and}}~~\mu_1
=\sum_{i\in I}b_i\delta_{u_i},$$ and let $\mu=\mu_0+\mu_1$.  Then the support
of $\mu$ is not contained in a great sphere, and since
$$\int_{S^{n-1}}u\,d\mu(u)=\sum_{i\in I}(a_i+b_i)u_i=w-w=o,$$
$\mu$ is balanced. By Minkowski's existence theorem
\cite[Theorem~A.3.2]{Gar95a}, there is a convex polytope $J$ such that
$S(J,\cdot)=\mu$.  By its definition, each facet of $J$ is parallel to a facet
of $Q$.

It remains to prove (\ref{Blacon}). Using \eqref{b}, we obtain
\begin{eqnarray*}
 d_P(S(J,\cdot),S(K,\cdot))&=&d_P(\mu_0+\mu_1,S(K,\cdot))
\le d_P(\mu_0+\mu_1,\mu_0)+ d_P(\mu_0,S(K,\cdot))\\
&=& d_P(\mu_1,0)+ d_P(\mu_0,S(K,\cdot))< c_{\ref{cr1}}\ee+
d_P(\mu_0,S(K,\cdot)), \end{eqnarray*} where $0$ is the zero measure in
$S^{n-1}$. In view of $\mu_0(S^{n-1})=S(K,S^{n-1})$ and \eqref{Pro1}, it is
therefore enough to find a constant \refstepcounter{Cnmb}\label{cr2}
$c_{\arabic{Cnmb}}$, depending only on $K$, such that
\begin{equation}\label{p1}
\mu_0(E) < S(K,E_{c_{\ref{cr2}}\ee})+c_{\ref{cr2}}\ee,
\end{equation}
for any Borel set $E$ in $S^{n-1}$. Let $X=\cup\{W_i:u_i\in E\} \setminus
E_\ee$. We have
\begin{eqnarray}
S(K,E_\ee)&\geq & S\left(K,E_\ee\cap (\cup\{W_i:u_i\in E\})\right)\nonumber\\
&=&\sum \{S(K,W_i):u_i\in E\} -S(K,X)=\mu_0(E)-S(K,X)\label{y}.
\end{eqnarray}
If $x\in X$, then for some $i$ with $u_i\in E$ we have $x\in W_i$, and so
$|x-u_i|\ge \ee$ since $x\not\in E_{\ee}$.
 Moreover, if $j\neq i$, then $|u_j-x|\ge
|u_i-x|\ge\ee$.  Hence $\cup_{i\in I}\{u_i\}\cap X_{\ee}=\emptyset$,
 and by (\ref{smallDK}), we have $S(K,X)< 2\ee$.
Now (\ref{y}) implies that \eqref{p1} holds with $c_{\ref{cr2}}=2$.
\end{proof}

For a fixed finite set $z_1,\dots,z_q$ of points in $\R^n$, define a pseudonorm
$|\cdot|_q$ by
\begin{equation}\label{pseudo}
|f|_q=\left(\frac{1}{q}\sum_{i=1}^qf(z_i)^2\right)^{1/2},
\end{equation}
where $f$ is any real-valued function on $\R^n$. For a convex body $K$
contained in $C_0^n$, vector $\vz_q=(z_1,\dots,z_q)$ of the points
$z_1,\dots,z_q$ in $\R^n$, and vector $\vX_q=(X_1,\ldots,X_q)$ of random
variables $X_1,\dots,X_q$, let
\begin{equation}\label{psi}
\Psi(K,\vz_q,\vX_q)=\frac{1}{q}\sum_{i=1}^qg_K(z_i)X_i.
\end{equation}

\begin{lem}\label{lem1}
Let $k\in\N$ and let $K_0\subset C_0^n$ be a convex body with its centroid at
the origin.  Suppose that $P_k$ is an output from Algorithm~NoisyCovLSQ as
stated above. Let $P(a)$ be any convex polytope admissible for the minimization
problem (\ref{obj1}). Then
\begin{equation}\label{normbound}
\left|g_{K_0}-g_{P_k}\right|_{I_k}^2\le 2\Psi(P_k,\vx_{I_k},\vN_{I_k})-2\Psi
(P(a)\cap C_0^n,\vx_{I_k},\vN_{I_k})+\left|g_{K_0}-g_{P(a)\cap
C_0^n}\right|_{I_k}^2,
\end{equation}
where for each $k\in\N$, $|\cdot|_{I_k}$ and $\Psi(K,\vx_{I_k},\vN_{I_k})$ are
defined by (\ref{pseudo}) and (\ref{psi}), respectively, with $q=I_k$,
$\vx_{I_k}=(x_{1k},\dots,x_{I_kk})$, and $\vN_{I_k}=(N_{1k},\ldots,N_{I_kk})$.
\end{lem}

\begin{proof}
If $P(\hat{a})\cap C_0^n$ is a solution of (\ref{obj1}), then since
$g_{P_k}=g_{P(\hat{a})\cap C_0^n}$, we obtain
$$
\sum_{i=1}^{I_k}\left(M_{ik}-g_{P_k}(x_{ik})\right)^2\le
\sum_{i=1}^{I_k}\left(M_{ik}-g_{P(a)\cap C_0^n}(x_{ik})\right)^2,
$$
Substituting for $M_{ik}$ from (\ref{meas}) and rearranging, we obtain
\begin{eqnarray*}
\sum_{i=1}^{I_k}\left(g_{K_0}(x_{ik})-g_{P_k}(x_{ik})\right)^2&\le &
2\sum_{i=1}^{I_k}g_{P_k}(x_{ik})N_{ik}-2\sum_{i=1}^{I_k}g_{P(a)\cap
C_0^n}(x_{ik})N_{ik}+\\
& & +\sum_{i=1}^{I_k}\left(g_{K_0}(x_{ik})-g_{P(a)\cap C_0^n}(x_{ik})\right)^2.
\end{eqnarray*}
In view of (\ref{pseudo}) and (\ref{psi}), this is the required inequality.
\end{proof}

Let $K$ be any convex body in $\R^n$ and let $\ee>0$.  The {\em inner parallel
body} $ K\ominus \ee B^n$ is the Minkowski difference of $K$ and $\ee B^n$ as
defined in (\ref{MD}).  Then
$$K\ominus \ee B^n=\bigcap_{y\in \ee B^n}(K-y),$$
so the inner parallel body is convex.  (It may be empty.)  For further
properties, see \cite[pp.~133--137]{Sch93}. The following proposition is an
immediate consequence of the fact that if $K$ is a convex body in $\R^n$, then
\begin{equation}\label{innerp}
V(K)-V(K\ominus \ee B^n)<S(K)\ee.
\end{equation}
This follows directly from either an inequality of Sangwine-Yager or one of
Brannen; see Theorem~1 or Corollary~2 of \cite{Bra}, respectively.  The
estimate (\ref{innerp}) both generalizes and strengthens
\cite[Lemma~4.2]{GarK}, which concerns the case $n=2$.  The authors of the
latter paper were unaware that an even stronger estimate for $n=2$ was found
earlier by Matheron \cite{Mat}.

\begin{prop}\label{prp2}
If $K\subset C_0^n$ is a convex body and $\ee>0$, then
$$
V(K)-V(K\ominus \ee B^n)<2n\ee.
$$
\end{prop}

Let $\mathcal G$ be the class of all nonnegative functions $g$ on $\R^n$ with
support in $2C_0^n$ that are the covariogram of some convex body contained in
$C_0^n$, together with the function on $\R^n$ that is identically zero. Note
that for each $g\in \mathcal G$ and $x\in \R^n$, $g(x)\le g_{C_0^n}(x)\le
V(C_0^n)=1$.

\begin{lem}\label{net}
Let $0<\ee<1$ be given. Then there is a finite set
$\{(g_j^L,g_j^U):\,j=1,\ldots,m\}$ of pairs of functions in $\mathcal G$ such
that
\begin{enumerate}
\item[(i)]
$\|g_j^U-g_j^L\|_1\le \ee$ for $j=1,\ldots,m$ and
\item[(ii)] for each $g\in \mathcal G$, there is an $j\in
\{1,\ldots,m\}$ such that $g_j^L\le g\le g_j^U$.
\end{enumerate}
\end{lem}

\begin{proof}
Let $0<\ee<1$ and let \refstepcounter{Cnmb}\label{caa}
$c_{\arabic{Cnmb}}=c_{\arabic{Cnmb}}(n)\ge 1$ be a constant, to be chosen
later. Since ${\mathcal K}^n(C_0^n)$ with the Hausdorff metric is compact,
there is an $\ee/c_{\ref{caa}}$-net $\{K_1,\ldots,K_m\}$ in ${\mathcal
K}^n(C_0^n)$. For each $j=1,\dots,m$, let
$K_j^U=(K_j+({\ee}/c_{\ref{caa}})B^n)\cap C_0^n$ and $K_j^L=K_j\ominus
({\ee}/c_{\ref{caa}})B^n$. Define $g_j^U=g_{K_j^U}$ and $g_j^L=g_{K_j^L}$,
$j=1,\ldots,m$. Both $g_j^U$ and $g_j^L$ belong to ${\mathcal G}$,
$j=1,\dots,m$.

We first prove (ii).  Let $g\in {\mathcal G}$.  There is a $K\in {\mathcal
K}^n(C_0^n)$ such that $g=g_K$.  Choose $j\in\{1,\dots,m\}$ such that
$\delta(K,K_j)\le \ee/c_{\ref{caa}}$. Since $K\subset C_0^n$ and $K\subset
K_j+({\ee}/c_{\ref{caa}})B^n$, we have $K\subset
(K_j+({\ee}/c_{\ref{caa}})B^n)\cap C_0^n=K_j^U$. Also, we have
$$(K_j\ominus ({\ee}/c_{\ref{caa}})B^n)+({\ee}/c_{\ref{caa}})B^n\subset K_j\subset
K+({\ee}/c_{\ref{caa}})B^n,$$ yielding $K_j^L=K_j\ominus
({\ee}/c_{\ref{caa}})B^n\subset K$. These facts imply that $g_j^L\le g \le
g_j^U$, as required.

It remains to prove (i).  It is easy to prove (see, for example,
\cite[p.~411]{Sch93}) that for any convex body $L$ in $\R^n$,
$$\int_{DL}g_L(x)\,dx=V(L)^2.$$
Applying this, Steiner's formula with quermassintegrals (see \cite[(A.30),
p.~404]{Gar95a}, basic properties of mixed volumes (see \cite[(A.16) and
(A.18), p.~399]{Gar95a}) together with $K_j\subset C_0^n \subset(n/4)^{1/2}B^n$
and $c_{\ref{caa}}\ge 1$, and Proposition~\ref{prp2} with $\ee$ replaced by
$\ee/c_{\ref{caa}}$, we obtain
\begin{eqnarray*}
\|g_j^U-g_j^L\|_1&=& \int_{2C_0^n}
\left(g_j^U(x)-g_j^L(x)\right)\,dx=V\left(K_j^U\right)^2-V\left(K_j^L\right)^2\le
2\left(V\left(K_j^U\right)-V\left(K_j^L\right)\right)\\
&\le &
2\left(\left(V\left(K_j+\frac{\ee}{c_{\ref{caa}}}B^n\right)-V(K_j)\right)+
\left(V(K_j)-V\left(K_j\ominus\frac{\ee}{c_{\ref{caa}}}B^n\right)\right)\right)\\
&\le & 2\left(\kappa_n\sum_{i=1}^n\left(\begin{array}{c}n\\i\end{array}\right)
\left(\frac{n}{4}\right)^{(n-i)/2}+
2n\right)\left(\frac{\ee}{c_{\ref{caa}}}\right)<\ee,
\end{eqnarray*}
provided that $c_{\ref{caa}}$ is chosen sufficiently large.
\end{proof}

By analogy with \cite[Definition~2.2]{vdG00}, we refer to a finite set
$\{(g_j^L,g_j^U):\,j=1,\ldots,m\}$ of pairs of functions in $\mathcal G$
satisfying (i) and (ii) of Lemma~\ref{net} as an {\em $\ee$-net with
bracketing} for the class $\mathcal G$.

The following proposition is a version of the strong law of large numbers that applies to a triangular family, rather than a sequence, of
random variables. A version with the assumptions of full independence and uniformly bounded fourth moments is proved in detail in \cite[Lemma~4.4]{GarK}, with $m_k=k$.  The stronger statement below follows directly from \cite[Corollary~1]{HMT89} (with $p=1$ and $n=m_k$ there); in fact, it is enough to assume the uniform boundedness of $p$th absolute moments where $p=2+\ee$ for some $\ee>0$, but we prefer to avoid this extra parameter in the sequel.

\begin{prop}\label{stronglaw}
Let $X_{ik}$, $k\in \N$, $i=1,\ldots,m_k$, where $m_k\ge k$, be a triangular array of row-wise independent zero mean random variables.  If the array has uniformly bounded third absolute moments, then, almost surely,
\begin{equation}\label{asConv}
\frac{1}{m_k} \sum_{i=1}^{m_k} X_{ik} \to 0
\end{equation}
as $k\to \infty$.
\end{prop}

\begin{lem}\label{bound}
For every $k\in \N$, let  $x_{ik}$, $i=1,\dots,I_k$, be the points in the cubic
array $2C_0^n\cap (1/k)\Z^n$. Let  $N_{ik}$, $k\in \N$, $i=1,\ldots,I_k$, be
row-wise independent zero mean random variables with uniformly bounded third absolute moments. Then, almost surely,
$$\sup_{K\in{\mathcal K}^n(C_0^n)}\Psi(K,\vx_{I_k},\vN_{I_k})\to 0$$
as $k\to\infty$, where for each $k\in\N$, $\Psi(K,\vx_{I_k},\vN_{I_k})$ is
defined by (\ref{psi}) with $q=I_k$, $\vx_{I_k}=(x_{1k},\dots,x_{I_kk})$, and
$\vN_{I_k}=(N_{1k},\ldots,N_{I_kk})$.
\end{lem}

\begin{proof}
Let $0<\ee<1$ and let $\{(g_j^L,g_j^U) :j=1,\dots,m\}$ be an $\ee$-net with
bracketing for $\mathcal G$, as provided by Lemma~\ref{net}. Let $K\in
{\mathcal K}^n(C_0^n)$ and let $g=g_{K}\in \mathcal G$. Choose
$j\in\{1,\dots,m\}$ such that $g_j^L\le g \le g_j^U$. Define
$N_{ik}^+=\max\{N_{ik},0\}$ and $N_{ik}^-=N_{ik}^+-N_{ik}$ for $k\in \N$ and
$i=1,\dots,I_k$. Then for $k\in \N$, we have
\begin{eqnarray*}
\Psi(K,\vx_{I_k},\vN_{I_k})&=&\frac {1}{I_k}\sum_{i=1}^{I_k} g(x_{ik}) N_{ik}^+
-\frac{1}{I_k}\sum_{i=1}^{I_k} g(x_{ik}) N_{ik}^-\\
&\le&\frac{1}{I_k}\sum_{i=1}^{I_k} g_j^U(x_{ik}) N_{ik}^+
-\frac{1}{I_k}\sum_{i=1}^{I_k} g_j^L(x_{ik}) N_{ik}^-\\
&\le &W_k(\ee),
\end{eqnarray*}
where
\begin{equation}\label{WWW}
W_k(\ee)=\max_{j=1,\ldots,m} \left\{\frac  {1}{I_k}\sum_{i=1}^{I_k}
g_j^U(x_{ik}) N_{ik}^+-\frac  {1}{I_k}\sum_{i=1}^{I_k} g_j^L(x_{ik})
N_{ik}^-\right\}
\end{equation}
is independent of $K$.  Consequently,
\begin{equation}\label{max}
\sup_{K\in{\mathcal K}^n(C_0^n)}\Psi(K,\vx_{I_k},\vN_{I_k})\le  W_k(\ee),
\end{equation}
for all $0<\ee<1$.

Fix $j\in \{1,\dots,m\}$, and let
$$
X_{ik}=g_j^U(x_{ik})N^+_{ik}-g_j^U(x_{ik})E(N^+_{ik}),
$$
for $k\in \N$ and $i=1,\dots,I_k$. Since $g_j^U(x_{ik})\le 1$, it is easy to
check that the random variables $X_{ik}$ satisfy the hypotheses of
Proposition~\ref{stronglaw}. By (\ref{asConv}) with $m_k=I_k$, we obtain, almost surely,
$$
\limsup_{k\to \infty}\frac{1}{I_k} \sum_{i=1}^{I_k} g_j^U(x_{ik})N^+_{ik}=
\limsup_{k\to \infty}\frac{1}{I_k} \sum_{i=1}^{I_k}g_j^U(x_{ik})E(N^+_{ik}).
$$
The same argument, with limits superior replaced by limits inferior, applies when $X_{ik}$ is defined by $X_{ik}=g_j^L(x_{ik})N^-_{ik}-g_j^L(x_{ik})E(N^-_{ik})$. Our moment assumption on the random variables $N_{ik}$ implies that there is a constant $C$ such that
$$E(N^+_{ik})=E(N^-_{ik})=\frac{1}{2}E(|N_{ik}|)\le C.$$
Also, by Lemma~\ref{net}(i) we have $\|g_j^U-g_j^L\|_1\le \ee$ and by
Lemma~\ref{net}(ii) we may assume that $g^U_j-g^L_j\ge 0$, for $i=1,\dots,m$.
Therefore, almost surely,
\begin{eqnarray*}
\lim_{k\to \infty}W_k(\ee)&=&\max_{j=1,\ldots,m}
\left\{
\limsup_{k\to \infty}\frac{1}{I_k} \sum_{i=1}^{I_k}g_j^U(x_{ik})E(N^+_{ik})-
\liminf_{k\to \infty}\frac{1}{I_k} \sum_{i=1}^{I_k}g_j^L(x_{ik})E(N^-_{ik})\right\}\\
&\le &\max_{j=1,\ldots,m}
\left\{
C\left(\limsup_{k\to \infty}\frac{1}{I_k} \sum_{i=1}^{I_k}g_j^U(x_{ik})-
\liminf_{k\to \infty}\frac{1}{I_k} \sum_{i=1}^{I_k}g_j^L(x_{ik})\right)\right\}\\
&\le &\max_{j=1,\ldots,m}
\left\{\frac{C}{2^n}\int_{2C_0^n}\left(g_j^U(x)-g_j^L(x)\right)\,dx\right\}
\le \frac{C\ee}{2^{n}}.
\end{eqnarray*}
This and (\ref{max}) complete the proof.
\end{proof}

\begin{lem}\label{pseudcov}
Let $K_0\subset C_0^n$ be a convex body with its centroid at the origin.
Suppose that $P_k$ is an output from Algorithm~NoisyCovLSQ as stated above.
Then, almost surely,
\begin{equation}\label{consis}
\lim_{k\to\infty}\left|g_{K_0}-g_{P_k}\right|_{I_k}=0.
\end{equation}
\end{lem}

\begin{proof}
Let $Q_k$ be the $o$-symmetric polytope from the input of Algorithm~NoisyCovLSQ
that satisfies, almost surely, \eqref{QkdoesIt}. Fix a realization for which
\eqref{QkdoesIt} holds. We may assume that
$$
\lim_{k\to\infty} \delta(Q_k,\nabla K_0)=0,
$$
as the other case is completely analogous. By \cite[Theorem~4.2.1]{Sch93},
$S(Q_k,\cdot)$ converges weakly to $S(\nabla K_0,\cdot)$ as
$k\rightarrow\infty$.  By \cite[Theorem~6.8]{Bil}, weak convergence is
equivalent to convergence in the Prohorov metric, so $S(Q_k,\cdot)$ converges
in the Prohorov metric to $S(\nabla K_0,\cdot)$ as $k\rightarrow\infty$.  Now
Lemma~\ref{Prohorov} ensures that if $J_k$ is the convex polytope corresponding
to $Q_k$ in that lemma, then $S(J_k,\cdot)$ converges in the Prohorov metric to
$S(K_0,\cdot)$ as $k\rightarrow\infty$. We may assume that the centroid of
$J_k$ is at the origin for each $k$.  By Lemma \ref{inrad} (with $K$ and $L$
replaced by $K_0$ and $J_k$, respectively), there are constants $a_1$ and
$k_0\in \N$, depending only on $K_0$, such that $J_k\subset a_1B^n$ for all
$k\ge k_0$. By Blaschke's selection theorem and the fact that a convex body is
determined up to translation by its surface area measure, the sequence $(J_k)$
has an accumulation point and every such accumulation point must be a translate
of $K_0$.  But $J_k$ and $K_0$ have their centroids at the origin and
$K_0\subset C_0^n$, so
$$\lim_{k\to\infty} \delta(K_0,J_k\cap C_0^n)=\lim_{k\to\infty} \delta(K_0,J_k)=0.$$
(This consequence of the fact that $d_P(S(J_k,\cdot),S(K_0,\cdot))\rightarrow
0$ as $k\rightarrow\infty$ can also be derived from a stability estimate of Hug
and Schneider \cite[Theorem~3.1]{HS}, but we do not need the full force of that
result here.) It follows from the continuity of volume that
$\|g_{K_0}-g_{J_k\cap C_0^n}\|_\infty\rightarrow 0$ as $k\rightarrow\infty$ and
hence that
\begin{equation}\label{may}
\lim_{k\to\infty} \left|g_{K_0}-g_{J_k\cap C_0^n}\right|_{I_k}=0.
\end{equation}

Next, we observe that $J_k$ can serve as the $P(a)$ in Lemma~\ref{lem1}. By its
definition, a translate of $P_k$ is contained in $C_0^n$, and the quantity
$\Psi(P_k,\vx_{I_k},\vN_{I_k})$ is unaffected by this translation.  From
Lemma~\ref{bound} we obtain
\begin{equation}\label{may2}
\lim_{k\to\infty}\Psi(P_k,\vx_{I_k},\vN_{I_k})=0\quad{\text{and}}\quad
\lim_{k\to\infty}\Psi(J_k\cap C_0^n,\vx_{I_k},\vN_{I_k})=0.
\end{equation}
Now (\ref{consis}) follows directly from (\ref{normbound}) (with $P(a)$
replaced by $J_k$), (\ref{may}), and (\ref{may2}).
\end{proof}

\begin{thm}\label{maincov}
Suppose that $K_0\subset C_0^n$ is a convex body with its centroid at the
origin.  Suppose also that $K_0$ is determined, up to translation and
reflection in the origin, among all convex bodies in $\R^n$, by its
covariogram. If $P_k$, $k\in\N$, is an output from Algorithm~NoisyCovLSQ as
stated above, then, almost surely,
\begin{equation}\label{mainconsis}
\min\{\delta(K_0,P_k),\delta(-K_0,P_k)\}\to 0
\end{equation}
as $k\to\infty$.
\end{thm}

\begin{proof}
By Lemma~\ref{pseudcov}, almost surely,
\begin{equation}\label{consis1}
\left|g_{K_0}-g_{P_k}\right|_{I_k}\rightarrow 0,
\end{equation}
as $k\to\infty$. Fix a realization for which this statement holds. For each
$k$, $P_k$ has its centroid at the origin and is a translate of a subset of
$C_0^n$, so $P_k\subset 2C_0^n$ and by Blaschke's selection theorem, $(P_k)$
has an accumulation point, $L$, say. Note that $L$ must also have its centroid at the origin and be a translate of a subset of $C_0^n$.

Let $(P_{k'})$ be a subsequence converging to $L$. Then since $g_{K_0}-
g_{P_{k'}}$ converges uniformly to $g_{K_0}-g_L$ as $k'\rightarrow\infty$, we
have
$$\left|g_{K_0}- g_{P_{k'}}\right|_{I_{k'}}^2\to \frac{1}{2^n} \int_{2C_0^n}
(g_{K_0}(x)-g_L(x))^2\,dx,$$ as $k'\to \infty$.  From this and (\ref{consis1}),
we obtain $\left\|g_{K_0}- g_L\right\|_{L^2(2C_0^n)}= 0$, and hence, since covariograms are clearly continuous, $g_{K_0}=g_L$ on $2C_0^n$.  As the supports of $g_{K_0}$ and $g_L$ are contained in $2C_0^n$, we have $g_{K_0}=g_L$ in $\R^n$. The hypothesis on $K_0$ now implies that $L=\pm K_0$. Since $L$ was an arbitrary accumulation point of $(P_k)$, we obtain (\ref{mainconsis}).
\end{proof}

\section{Approximating the Blaschke body via the covariogram}\label{Algorithm}

\bigskip

{\large{\bf Algorithm~NoisyCovBlaschke}}

\bigskip

{\it Input:} Natural numbers $n\ge 2$ and $k$; mutually nonparallel vectors
$u_i\in S^{n-1}$, $i=1,\dots,k$ that span $\R^n$; noisy covariogram
measurements
$$
M_{ijk}^{(1)}=g_{K_0}(o)+N_{ijk}^{(1)}\quad\text{and}\quad M_{ijk}^{(2)}
=g_{K_0}((1/k)u_i)+N_{ijk}^{(2)},
$$
for $i=1,\dots,k$ and $j=1,\dots,k^2$, of an unknown convex body $K_0\subset
C_0^n$ whose centroid is at the origin, where the $N_{ijk}^{(m)}$'s are
row-wise independent (i.e., independent for fixed $k$) zero mean random variables with uniformly bounded sixth moments.

\smallskip

{\it Task:} Construct an $o$-symmetric convex polytope $Q_k$ that approximates
the Blaschke body $\nabla K_0$.

\smallskip

{\it Action:}

1.  For $i=1,\dots,k$ and $j=1,\dots,k^2$, let
$$y_{ik}=\frac{1}{k^2}\sum_{j=1}^{k^2}k(M_{ijk}^{(1)}-M_{ijk}^{(2)}).$$

2. With the natural numbers $n\ge 2$ and $k$, and vectors $u_i\in S^{n-1}$,
$i=1,\dots,k$ use the sample means $y_{ik}$ instead of noisy measurements of
the brightness function $b_K(u_i)$ as input to Algorithm~NoisyBrightLSQ (see
\cite[p.~1352]{GKM06}).  The output of the latter algorithm is $Q_k$.

\bigskip

For a fixed finite set $u_1,\dots,u_q$ of points in $S^{n-1}$, define a
pseudonorm $|\cdot|_q$ by
\begin{equation}\label{pseudo2}
|f|_q=\left(\frac{1}{q}\sum_{i=1}^qf(u_i)^2\right)^{1/2},
\end{equation}
where $f$ is any real-valued function on $S^{n-1}$. For a convex body $K$
contained in $C_0^n$, a sequence $(u_i)$ in $S^{n-1}$, and a vector
$\vX_k=(X_{1k},\ldots,X_{kk})$ of random variables, let
$$\Psi(K,(u_i),\mathbf{X}_k)=\frac{1}k\sum_{i=1}^kb_K(u_i)X_{ik}.$$
The same notations were used for a technically different pseudonorm and
function $\Psi$ in the previous section, but this should cause no confusion.

\begin{lem}\label{LEMbasicinequality}
Let $K_0$ be a convex body in $\R^n$ with centroid at the origin and such that
$rB^n\subset K_0\subset C_0^n$ for some $r>0$. Let $(u_i)$ be a sequence in
$S^{n-1}$. If $Q_k$ is an output from Algorithm~NoisyCovBlaschke as stated
above, then, almost surely, there is a constant
\refstepcounter{Cnmb}\label{cbi} $c_{\arabic{Cnmb}}=c_{\arabic{Cnmb}}(n,r)$
such that
\begin{equation}\label{bi}
|b_{K_0}-b_{Q_k}|_k^2\le 2\Psi(Q_k,(u_i),\mathbf{X}_k)-2
\Psi(K_0,(u_i),\mathbf{X}_k)+ \frac{c_{\arabic{Cnmb}}}{k}| b_{K_0}-b_{Q_k}|_k,
\end{equation}
for all $k\in\N$.  Here $\vX_k=(X_{1k},\ldots,X_{kk})$, with
$$X_{ik}=\frac{1}{k}\sum_{j=1}^{k^2}(N_{ijk}^{(1)}-N_{ijk}^{(2)}),$$
for $i=1,\dots,k$.

\end{lem}

\begin{proof}
For $i=1,\ldots,k$, we have
$$y_{ik}=\frac{g_{K_0}(o)-g_{K_0}((1/k)u_i)}{1/k}+\frac1k
\sum_{j=1}^{k^2}(N_{ijk}^{(1)}-N_{ijk}^{(2)})= \mu_{ik}+X_{ik},$$
where the $X_{ik}$'s are row-wise independent zero mean random variables.  Note that the $y_{ik}$'s are also row-wise independent.  Furthermore, by Khinchine's inequality (see, for example, \cite[(4.32.1), p.~307]{Hof94} with $\alpha=6$), there is a constant $C$ such that
$$E\left(|X_{ik}|^6\right)\le \frac{C}{k^2}\sum_{j=1}^{k^2}E\left(\left|N_{ijk}^{(1)}-
N_{ijk}^{(2)}\right|^6\right),$$
from which we see that the $X_{ik}$'s also have uniformly bounded sixth moments. By Lemma~\ref{unif1},
$$\lim_{k\to\infty}\mu_{ik}=b_{K_0}(u_i).$$
In fact, the convergence is uniform.  This is because for each $u\in S^{n-1}$,
we have
$$b_{K_0}(u)\le b_{C_0^n}(u)\le b_{(\sqrt n/2)B^n}(u)=
(n/4)^{(n-1)/2}\kappa_{n-1}$$ and
\begin{equation}\label{new42}
0\le b_{K_0}(u)-\mu_{ik}\le
\left(1-\left(1-\frac{1}{2rk}\right)^{n-1}\right)b_{K_0}(u)\le \frac{n-1}{2rk}
b_{K_0}(u),\quad k\ge 1/(2r),
\end{equation}
by Lemma~\ref{unif1}, so there is a constant \refstepcounter{Cnmb}\label{c3}
$c_{\arabic{Cnmb}}=c_{\arabic{Cnmb}}(n,r)$ such that
\begin{equation}\label{eq2}
0\le b_{K_0}(u_i)-\mu_{ik}\le \frac{c_{\ref{c3}}}k,
\end{equation}
for all $k\in \N$ and $i=1,\ldots,k$.

By the formulation of Algorithms~NoisyCovBlaschke and~NoisyBrightLSQ (see
\cite[p.~1352]{GKM06} and take \cite[Proposition~2.1]{GKM06} into account), $Q_k$ minimizes
\begin{equation}\label{eq1}
\sum_{i=1}^k\left(b_{K}(u_i)-y_{ik}\right)^2
\end{equation}
over the class of all $o$-symmetric convex bodies $K$ in $\R^n$. By
(\ref{newBB}), for each convex body there is an $o$-symmetric convex body with
the same brightness function.  From this it follows that $Q_k$ is actually a
minimizer over the class of all convex bodies $K$ in $\R^n$. Substituting
$K=Q_k$ and $K=K_0$ in \eqref{eq1}, we obtain
$$
\sum_{i=1}^k \left(b_{Q_k}(u_i)-\mu_{ik}-X_{ik}\right)^2\le
\sum_{i=1}^k\left(b_{K_0}(u_i)-\mu_{ik}-X_{ik}\right)^2.
$$
Rearranging and using (\ref{pseudo2}), we obtain
$$
|b_{K_0}-b_{Q_k}|_k^2\le \frac2k\sum_{i=1}^k\left(
b_{Q_k}(u_i)-b_{K_0}(u_i)\right)(X_{ik}-(b_{K_0}(u_i)-\mu_{ik})).
$$
The definition of $\Psi$ and Cauchy-Schwarz inequality yields
$$
|b_{K_0}-b_{Q_k}|_k^2\le2\Psi(Q_k,(u_i),\mathbf{X}_k)-2
\Psi(K_0,(u_i),\mathbf{X}_k)+ 2|b_{K_0}-b_{Q_k}|_k \left(\frac1k \sum_{i=1}^k
(b_{K_0}(u_i)-\mu_{ik})^2\right)^{1/2}.
$$
In view of \eqref{eq2}, this proves (\ref{bi}) with
$c_{\ref{cbi}}=2c_{\ref{c3}}$.
\end{proof}

\begin{lem}\label{ub}
Suppose that the assumptions of Lemma~\ref{LEMbasicinequality} are satisfied
with a sequence  $(u_i)$ such that $(u_i^*)$ is evenly spread.  Suppose also that the second moments of the $X_{ik}$'s are uniformly bounded by a constant $C>0$. Then, almost surely, there are constants \refstepcounter{Cnmb}\label{c1}
$c_{\arabic{Cnmb}}=c_{\arabic{Cnmb}}(C,n,r,(u_i))$ and
\refstepcounter{Nnmb}\label{Nanewer} $N_{\arabic{Nnmb}}=N_{\arabic{Nnmb}}((X_{ik}),(u_i))$
such that
\begin{equation}\label{unifBound}
S(Q_k)\le c_{\ref{c1}},
\end{equation}
for all $k\ge N_{\ref{Nanewer}}$.
\end{lem}

\begin{proof}
By the Cauchy-Schwarz inequality,
$$
\Psi(Q_k,(u_i),\mathbf{X}_k)-\Psi(K_0,(u_i),\mathbf{X}_k) \le
|b_{K_0}-b_{Q_k}|_k\left(\frac1k\sum_{i=1}^k X_{ik}^2\right)^{1/2}.
$$
This and \eqref{bi} imply that
$$
|b_{K_0}-b_{Q_k}|_k\le 2\left(\frac1k\sum_{i=1}^k X_{ik}^2\right)^{1/2}
+\frac{c_{\ref{cbi}}}k,
$$
for all $k\in \N$.  Since the $X_{ik}$'s have uniformly bounded sixth moments, we can apply Proposition~\ref{stronglaw} with $m_k$ and $X_{ik}$
replaced by $k$ and $X_{ik}^2-E\left(X_{ik}^2\right)$, respectively, to conclude that
the first term on the right-hand side is bounded, almost surely. Thus, almost surely, there are constants \refstepcounter{Cnmb}\label{cn}
$c_{\arabic{Cnmb}}=c_{\arabic{Cnmb}}(C,n,r)$ and
\refstepcounter{Nnmb}\label{Nanew} $N_{\arabic{Nnmb}}=N_{\arabic{Nnmb}}((X_{ik}),(u_i))$ such that
\begin{equation}\label{almostdone}
|b_{K_0}-b_{Q_k}|_k\le c_{\arabic{Cnmb}},
\end{equation}
for all $k\ge N_{\ref{Nanew}}$.
As $(u_i^*)$ is evenly spread, we can apply \cite[Lemma 7.1]{GKM06} with $K$
and $L$ replaced by $\Pi K_0$ and $\Pi Q_k$, respectively.  Using this, the
fact that $\Pi K_0\subset \Pi C_0^n=2C_0^n\subset \sqrt{n}B^n$ (see
\cite[p.~145]{Gar95a}), and \eqref{projbod}, we find that there are constants
\stepcounter{Cnmb} $c_{\arabic{Cnmb}}=c_{\arabic{Cnmb}}((u_i))$ and
\refstepcounter{Nnmb}\label{Na}
$N_{\arabic{Nnmb}}=N_{\arabic{Nnmb}}((u_i))$ such that
\begin{equation}\label{done}
b_{Q_k}\le c_{\arabic{Cnmb}}|b_{K_0}-b_{Q_k}|_k+2\sqrt{n},
\end{equation}
for $k\ge N_{\ref{Na}}$. Finally, (\ref{unifBound}) follows directly from
\eqref{almostdone}, \eqref{done}, and \eqref{Cauchynew}.
\end{proof}

\begin{lem}\label{extend}
Suppose that the assumptions of Lemma~\ref{LEMbasicinequality} are satisfied
with a sequence $(u_i)$ such that $(u_i^*)$ is evenly spread.  Then, almost
surely,
\begin{equation}\label{myConv}
\lim_{k\to \infty}|b_{K_0}-b_{Q_k}|_k=0.
\end{equation}
\end{lem}

\begin{proof}
Choose a constant $C_1$ such that $E\left(|X_{ik}|^2\right)\le C_1$ for all $i$ and $k$.  Due to \eqref{bi} and  \eqref{almostdone}, there is, almost surely, a constant
\refstepcounter{Cnmb}\label{cgood}
$c_{\arabic{Cnmb}}=c_{\arabic{Cnmb}}(C_1,n,r)$ such that
\begin{equation}\label{eqff}
|b_{Q_k}-b_{K_0}|_k^2\le 2\Psi(Q_k,(u_i),\mathbf{X}_k)-2
\Psi(K_0,(u_i),\mathbf{X}_k)+\frac{c_{\ref{cgood}}}k,
\end{equation}
for all $k\ge N_{\ref{Nanew}}$. By Proposition~\ref{stronglaw} with $m_k=k$ and
$X_{ik}$ replaced by $b_{K_0}(u_i)X_{ik}$, the variable
$\Psi(K_0,(u_i),\mathbf{X}_k)$ converges to zero, almost surely, as $k\to
\infty$.

For $m\in \N$, let ${\mathcal H}_m=\{ K\in \cK^n: S(K)\le m\}$. If we can show
that for all $m\in \N$, almost surely,
\begin{equation}\label{unif}
\lim_{k\to \infty} \sup_{K\in {\mathcal H}_m}|\Psi(K,(u_i),\mathbf{X}_k)|=0,
\end{equation}
then by \eqref{unifBound}, almost surely,
$$\lim_{k\to \infty}\Psi(Q_k,(u_i),\mathbf{X}_k)=0.$$
This and \eqref{eqff} will yield \eqref{myConv}, completing the proof.

To prove (\ref{unif}), note first that by \eqref{Cauchy}, we have
$$|\Psi(K,(u_i),\mathbf{X}_k)|=\left|\frac{1}k\sum_{i=1}^k b_K(u_i)X_{ik}\right|
\le\frac12\int_{S^{n-1}}\left|\frac{1}k\sum_{i=1}^k| u_i\cdot v |X_{ik}\right|
\,dS(K,v).$$ Since $S(K)=S(K,S^{n-1})\le m$ for $K\in{\mathcal H}_m$, it is
enough to prove that, almost surely,
\begin{equation}\label{enough}
 \lim_{k\to \infty} \sup_{v\in S^{n-1}}\left|
\frac{1}k\sum_{i=1}^k| u_i\cdot v |X_{ik}\right|=0.
\end{equation}
This follows essentially from the uniform continuity of the function $|u_i\cdot
v|$, $v\in S^{n-1}$, and the fact that $S^{n-1}$ is compact. Indeed, suppose that
\eqref{enough} does not hold almost surely. Choose a constant $C_2$ such that $E(|X_{ik}|)\le C_2$ for all $i$ and $k$. Then there is a $\delta>0$ such
that
\begin{equation}\label{unif2}
\limsup_{k\to \infty} \sup_{v\in S^{n-1}} \frac{1}k\sum_{i=1}^k| u_i\cdot
v|X_{ik}>\delta C_2
\end{equation}
with positive probability. Let $\{w_1,\ldots,w_m\}$ be a $\delta/2$-net in
$S^{n-1}$. For any realization and any $k\in \N$, there is a $v_k\in S^{n-1}$
such that
\begin{equation}\label{zzz}
\frac{1}k\sum_{i=1}^k|u_i\cdot v_k|X_{ik}= \sup_{v\in S^{n-1}}
\frac{1}k\sum_{i=1}^k|u_i\cdot v|X_{ik}.
\end{equation}
Let $A_j$ denote the set of all events such that an accumulation point of
$(v_k)$ has distance at most $\delta/2$ from $w_j$, $j=1,\ldots,m$. For a
realization in $A_j$ and any subsequence $(k')$ of $(k)$ such that
$|v_{k'}-w_j|\le \delta$ holds for sufficiently large $k$, we have, almost
surely,
$$\limsup_{k'\to \infty} \left|\frac{1}{k'}\sum_{i=1}^{k'}|u_i\cdot v_{k'}|X_{ik'}
-\frac{1}{k'}\sum_{i=1}^{k'}|u_i\cdot w_j|X_{ik'} \right|\le \delta
\limsup_{k'\to \infty}\frac{1}{k'}\sum_{i=1}^{k'} |X_{ik'}|\le \delta C_2,
$$
by Proposition~\ref{stronglaw} with $m_k$ and $X_{ik}$ replaced by $k'$ and $|X_{ik'}|-E(|X_{ik'}|)$, respectively. But Proposition~\ref{stronglaw}, with $m_k$ and
$X_{ik}$ replaced by $k'$ and $|u_i\cdot w_j|X_{ik'}$, respectively, also
implies that, almost surely, the second term on the left-hand side converges to
zero, as $k'\rightarrow\infty$. In view of (\ref{zzz}), this yields
$$\limsup_{k'\to \infty}\sup_{v\in S^{n-1}} \frac{1}{k'}\sum_{i=1}^{k'}|u_i\cdot
v|X_{ik'}\le\delta C_2,
$$
for almost all events in $A_j$. As any sequence in $S^{n-1}$ has at least one
accumulation point, the latter inequality holds, almost surely, contradicting
\eqref{unif2}.
\end{proof}

\begin{thm}\label{covB}
Let $K_0\subset C_0^n$ be a convex body with its centroid at the origin. Let
$(u_i)$ be a sequence in $S^{n-1}$ such that $(u_i^*)$ is evenly spread.  If
$Q_k$ is an output from Algorithm~NoisyCovBlaschke as stated above, then,
almost surely,
\begin{equation}\label{cons}
\lim_{k\to\infty} \delta(\nabla K_0,Q_k)=0.
\end{equation}
\end{thm}

\begin{proof}
We have $o\in\inte K_0$, so there is an $r>0$ such that $rB^n\subset K_0$. By
Lemmas~\ref{ub} and~\ref{extend}, we can fix a realization for which both
\eqref{unifBound} and (\ref{myConv}) are true. Using (\ref{projbod}), we
observe that (\ref{myConv}) is equivalent to
\begin{equation}\label{hh}
\lim_{k\to \infty} |h_{\Pi K_0}- h_{\Pi Q_k}|_k=0.
\end{equation}
We also have $h_{\Pi Q_k}=b_{Q_k}\le S(Q_k)$, so by \eqref{unifBound}, the sets
$\Pi Q_k$ are uniformly bounded.  With these observations and the fact that $(u_1,-u_1,u_2,-u_2,\dots)$ is evenly spread, we can follow the proof of \cite[Theorem~6.1]{GKM06}), from the fourth line, with $K$ and $\hat{P}_k$ replaced by $\Pi K_0$ and $\Pi Q_k$, respectively, to conclude that
\begin{equation}\label{picon}
\lim_{k\rightarrow\infty} \delta(\Pi K_0,\Pi Q_k)=0.
\end{equation}
Now $rB^n\subset K_0\subset C_0^n$ yields $sB^n\subset \Pi K_0\subset tB^n$
with $s=\kappa_{n-1}r^{n-1}$ and $t=\sqrt{n}$.  Moreover, (\ref{projbod}) and
(\ref{newBB}) give $\Pi(\nabla K_0)=\Pi K_0$.  Hence (\ref{picon}) implies that
$$\frac{s}{2}B^n\subset \Pi(\nabla K_0), \Pi Q_k \subset \frac{3t}{2}B^n,$$
for sufficiently large $k$, where $s$ and $t$ depend only on $n$ and $r$.
Exactly as in the proof from (48) to (49) of \cite[Theorem~7.2]{GKM06} (which
in turn follows the proof of \cite[Lemma~4.2]{GarM03}), this leads to
$$r_0B^n\subset \nabla K_0, Q_k \subset R_0B^n,$$
for sufficiently large $k$, where $r_0>0$ and $R_0$ depend only on $n$ and $r$.
Then (\ref{cons}) follows from (\ref{picon}) and the
Bourgain-Campi-Lindenstrauss stability result for projection bodies (see
\cite{BouL88} and \cite{Cam88}, or \cite[Remark~4.3.13]{Gar95a}).
\end{proof}

\section{Approximating the difference body via the covariogram}\label{diffe}

Throughout this section, $\varphi$ will be a nonnegative bounded measurable
function on $\R^n$ with support in $C_0^n$, such that
$\int_{\R^n}\varphi(x)\,dx=1$.

\bigskip

{\large{\bf Algorithm~NoisyCovDiff($\varphi$)}}

\bigskip

{\it Input:} Natural numbers $n\ge 2$ and $k$; positive reals $\delta_k$ and
$\ee_k$; noisy covariogram measurements
\begin{equation}\label{meas22}
M_{ik}=g_{K_0}(x_{ik})+N_{ik},
\end{equation}
of an unknown convex body $K_0\subset C_0^n$ at the points $x_{ik}$,
$i=1,\dots,I_k$ in the cubic array $2C_0^n\cap (1/k)\Z^n$, where the $N_{ik}$'s are row-wise independent zero mean random variables with uniformly bounded fourth moments.

\smallskip

{\it Task:} Construct an $o$-symmetric convex polytope $Q_k$ in $\R^n$ that
approximates the difference body $DK_0$.

\smallskip

{\it Action:}

1.  Let $\varphi_{\ee_k}(x)=\ee_k^{-n} \varphi(x/\ee_k)$ for $x\in \R^n$, and
let
\begin{equation}\label{kkernel}
g_k(x)=\sum_{i=1}^{I_k} M_{ik}\int_{(1/k)C_0^n+x_{ik}}\varphi_{\ee_k}(x-z)\,dz
=\left(\sum_{i=1}^{I_k} M_{ik}\,1_{(1/k)C_0^n+x_{ik}}\right)\ast
\varphi_{\ee_k}\,(x).
\end{equation}

2. Define the finite set
\begin{equation}\label{Sk}
S_k=\{x\in 2C_0^n\cap (1/k)\Z^n : g_k(x)\ge \delta_k\}.
\end{equation}
The output is the convex polytope $Q_k=(1/2)(\conv S_k+(-\conv S_k))$.

\bigskip

The input $\delta_k$ in the algorithm is a threshold parameter. The function
$g_k(x)$ is a Gasser-M\"uller type kernel estimator for $g_{K_0}$ with kernel
function $\varphi$ and bandwidth $\ee_k$. As the design points $x_{ik}$ are
deterministic, $g_k$ is a multivariate fixed design kernel estimator. Such
estimators are common in multivariate regression and are discussed in detail by
Ahmad and Lin \cite{ahm:lin84}. Among other things, strong pointwise
consistency and a bound for the rate of weak pointwise convergence are given
there. We shall need uniform bounds and establish them in the next two lemmas.
By \cite[Theorem 1]{ahm:lin84}, for any $x\in \R^n$, $g_k(x)$ is an
asymptotically unbiased estimator for $g_{K_0}(x)$, if $\ee_k\to 0$ as
$k\to\infty$. We shall show that this holds uniformly in $x$.

\begin{lem}\label{unifAsympUnbias}
Suppose that $K_0$, $\ee_k$, and $g_k$ are as in
Algorithm~NoisyCovDiff($\varphi$). For each $k\in \N$ and $x\in \R^n$,
$$
|E\left(g_k(x)\right)-g_{K_0}(x)|\le n(\ee_k+1/k).
$$
Consequently, $g_k$ is uniformly asymptotically unbiased whenever
$\lim_{k\to\infty}\ee_k=0$.
\end{lem}

\begin{proof}
Using (\ref{meas22}), (\ref{kkernel}), and the definition of $\varphi_{\ee_k}$,
we obtain
\begin{equation}\label{aaa}
|E\left(g_k(x)\right)-g_{K_0}(x)|\le \sum_{i=1}^{I_k}
|g_{K_0}(x_{ik})-g_{K_0}(x)|\int_{(1/k)C_0^n+x_{ik}}\varphi_{\ee_k}(x-z)\,dz,
\end{equation}
for all $x\in \R^n$.  The support of $\varphi_{\ee_k}$ is contained in
$\ee_kC_0^n$, so for fixed $x$, the support of the integrand
$\varphi_{\ee_k}(x-z)$ is contained in $\ee_kC_0^n+x$. Now if $x_{ik}\not\in
(\ee_k+1/k)C_0^n+x$, then $\ee_kC_0^n+x$ and $(1/k)C_0^n+x_{ik}$ are disjoint,
so the corresponding summand in \eqref{aaa} vanishes. Moreover, for $x_{ik}\in
(\ee_k+1/k)C_0^n+x$, Corollary~\ref{lipcor} and the fact that the diameter of
$C_0^n$ is $\sqrt{n}$ imply that
$$|g_{K_0}(x_{ik})-g_{K_0}(x)|\le n(\ee_k+1/k).$$
Consequently,
\begin{eqnarray*}
|E\left(g_k(x)\right)-g_{K_0}(x)|&\le & n (\ee_k+1/k) \sum_{i=1}^{I_k}
\int_{(1/k)C_0^n+x_{ik}}\varphi_{\ee_k}(x-z)\, dz\\
&\le & n (\ee_k+1/k)\int_{\R^n}\varphi_{\ee_k}(x-z)\,dz=n (\ee_k+1/k),\\
\end{eqnarray*}
as required.
\end{proof}

In \cite[Lemma 1]{ahm:lin84}, a polynomial rate of convergence result in the
weak sense is established for independent identically distributed measurement errors with polynomial tails. In contrast, we assume only uniformly bounded fourth moments and obtain a convergence rate that holds uniformly, using the Lipschitz continuity of the covariogram.

\begin{lem}\label{expo}
Suppose that $K_0$, $\ee_k$, and $g_k$ are as in
Algorithm~NoisyCovDiff($\varphi$) and let $\delta>0$ and $\lim_{k\to \infty}
\ee_k=0$. Then there are constants \refstepcounter{Cnmb}\label{caa222}
$c_{\arabic{Cnmb}}=c_{\arabic{Cnmb}}(\varphi)$ \refstepcounter{Nnmb}\label{na}
and $N_{\arabic{Nnmb}}=N_{\arabic{Nnmb}}((\ee_k),n)\in \N$ such that
\begin{equation}\label{exp1}
{\mathrm{Pr}}\,(|g_k(x)-g_{K_0}(x)|>\delta)\le
c_{\ref{caa222}}(2k+1)^n\delta^{-4}(k\ee_k)^{-3n},
\end{equation}
for all  $k\ge N_{\ref{na}}$ and all $x\in \R^n$.
\end{lem}

\begin{proof}
Let  $x\in \R^n$ and $k\in \N$ be fixed and define
\begin{equation}\label{beta}
\beta_{ik}=\beta_{ik}(x)=\int_{(1/k)C_0^n+x_{ik}}\varphi_{\ee_k}(x-z)\,dz,
\end{equation}
for $i=1,\ldots,I_k$. Then
\begin{equation}\label{betabound}
\beta_{ik}\le \|\varphi_{\ee_k}\|_\infty V\left((1/k)C_0^n\right)
=\|\varphi\|_\infty (k\ee_k)^{-n}
\end{equation}
and
\begin{equation}\label{sumbetabound}
\sum_{i=1}^{I_k}\beta_{ik}\le \int_{\R^n}\varphi_{\ee_k}(x-z)\,dz=1.
\end{equation}
In view of \eqref{meas22}, (\ref{kkernel}), and (\ref{beta}),
$$
g_k(x)-E\left(g_k(x)\right)=\sum_{i=1}^{I_k} \beta_{ik}N_{ik}
$$
is a sum of zero mean independent random variables.  The assumption that the $N_{ik}$'s have uniformly bounded fourth moments implies that $E\left(|N_{ik}|^4\right)\le C$ for some constant $C$ and all $i$ and $k$.  Now, using Markov's inequality, Khinchine's inequality (see, for example, \cite[(4.32.1), p.~307]{Hof94} with $\alpha=4$), (\ref{betabound}), and (\ref{sumbetabound}), we obtain
\begin{eqnarray}\label{markhi}
{\mathrm{Pr}}\left(|g_k(x)-E\left(g_k(x)\right)|\ge \delta/2\right)&\le &
(\delta/2)^{-4}E\left(\left|\sum_{i=1}^{I_k} \beta_{ik}N_{ik}\right|^4\right)\nonumber\\
&\le &
c\delta^{-4}I_k\sum_{i=1}^{I_k}E
\left(\left|\beta_{ik}N_{ik}\right|^4\right)\nonumber\\
&\le &
cC\delta^{-4}I_k\sum_{i=1}^{I_k}\beta_{ik}^4\nonumber\\
&\le &
cC\delta^{-4}I_k\left(\|\varphi\|_\infty(k\ee_k)^{-n}\right)^{3}
\sum_{i=1}^{I_k}\beta_{ik}\nonumber\\
&\le &
c_{\ref{caa222}}(2k+1)^n\delta^{-4}(k\ee_k)^{-3n},
\end{eqnarray}
for all $\delta>0$, where $c$ is a constant and $c_{\ref{caa222}}=cC\|\varphi\|_\infty^{3}$.
By Lemma~\ref{unifAsympUnbias}, there is a constant
$N_{\ref{na}}=N_{\ref{na}}((\ee_k),n)\in \N$ such that for all $k\ge
N_{\ref{na}}$ and $x\in \R^n$, we have $|E\left(g_k(x)\right)-g_{K_0}(x)|\le
\delta/2$ and therefore
\begin{eqnarray*}
{\mathrm{Pr}}\,(|g_k(x)-g_{K_0}(x)|>\delta) &\le &
{\mathrm{Pr}}\,\left(|g_k(x)-E\left(g_{k}(x)\right)|+|E\left(g_{k}(x)\right)
-g_{K_0}(x)|>\delta\right)\\
&\le & {\mathrm{Pr}}\,\left(|g_k(x)-E\left(g_k(x)\right)|
>\delta/2\right).
\end{eqnarray*}
Now \eqref{exp1} follows from this and \eqref{markhi}.
\end{proof}

For a convex body $K$ in $\R^n$ and $\delta>0$, let $K(\delta)=\{x\in \R^n:
g_{K}(x)\ge \delta\}$. Since $g_{K}^{1/n}$ is concave on its support,
$K(\delta)$ is a compact convex set, sometimes called a convolution body of
$K$. References to results on convolution bodies can be found in
\cite[p.~378]{Gar95a}.

\begin{lem}\label{sDelta}
Let $K$ be a convex body in $\R^n$. If $0<\delta<V(K)$, then
$$
\left(1-\frac{\delta^{1/n}}{V(K)^{1/n}}\right) DK\subset K(\delta).
$$
\end{lem}

\begin{proof}
Let $t=(\delta/V(K))^{1/n}$ and let $x\in (1-t)DK$. Since $DK$ is the support
of $g_{K}$, there is a $y$ in the support of $g_{K}$ such that $x=(1-t)y+t o$.
As $g_{K}^{1/n}$ is concave on its support, we have
$$
g_{K}(x)^{1/n}\ge(1-t) g_{K}(y)^{1/n}+t g_{K}(o)^{1/n} \ge t
V(K)^{1/n}=\delta^{1/n}.
$$
It follows that $x\in K(\delta)$.
\end{proof}

\begin{thm}\label{maindiffe}
Suppose that $K_0$, $\delta_k$, $\ee_k$, and $g_k$ are as in
Algorithm~NoisyCovDiff($\varphi$). Assume that
$\lim_{k\to\infty}\ee_k=\lim_{k\to\infty}\delta_k=0$ and that
\begin{equation}\label{essent}
\liminf_{k\to\infty}\delta^4_k\ee_k^{3n}k^{n-3/2}>0.
\end{equation}
Let \refstepcounter{Cnmb}\label{cleadingconst}
$c_{\arabic{Cnmb}}>\sqrt{n}(2/V(K_0))^{1/n}$. If $Q_k$ is an output from
Algorithm NoisyCovDiff($\varphi$) as stated above, then, almost surely,
\begin{equation}\label{convRate}
\delta(DK_0,Q_k)\le c_{\ref{cleadingconst}}\delta_k^{1/n},
\end{equation}
for sufficiently large $k$. In particular, almost surely, $Q_k$ converges to
$DK_0$, as $k\rightarrow\infty$.
\end{thm}

\begin{proof}
Let
$$
a_k=\max_{x\in 2C_0^n\cap (1/k) \Z^n} |g_k(x)-g_{K_0}(x)|.
$$
By Lemma~\ref{expo} and~\eqref{essent}, we have
\begin{eqnarray*}
{\mathrm{Pr}}\,(a_k\ge\delta_k)&\le &\sum_{x\in 2C_0^n\cap
(1/k)\Z^n} {\mathrm{Pr}}\,(|g_k(x)-g_{K_0}(x)|\ge \delta_k)\\
&\le &c_{\ref{caa222}}(2k+1)^{2n}\delta_k^{-4}(k\ee_k)^{-3n}=O\left(k^{-3/2}\right).
\end{eqnarray*}
Therefore, by the
Borel-Cantelli lemma, we see that, almost surely, $a_k< \delta_k$  for
sufficiently large $k$. Fix a realization and a $k\in \N$ such that $a_k<\delta_k$ and
\begin{align}\label{st1}
\left(\frac{2\delta_k}{V(K_0)}\right)^{1/n}+\frac{3}{s(K_0)k}\le 1,
\end{align}
where $s(K_0)=\max\{\rho\ge 0: \rho C_0^n\subset DK_0\}$. As $a_k< \delta_k$,
the definition \eqref{Sk} of $S_k$ implies
$$
K_0(2\delta_k)\cap \frac1k\Z^n\subset S_k\subset DK_0.
$$
The set on the left is $o$-symmetric, and $DK_0$ is convex and $o$-symmetric,
so
\begin{align}
\label{halfway} \conv\left( K_0(2\delta_k)\cap \frac1k\Z^n\right)\subset Q_k
\subset DK_0.
\end{align}
We claim that
\begin{align}
\label{include} K_0(2\delta_k)\ominus \frac 3k C_0^n\subset
\conv\left(K_0(2\delta_k)\cap \frac1k\Z^n\right),
\end{align}
where Minkowski difference $\ominus$ is defined by (\ref{MD}). Indeed, let
$x\in K_0(2\delta_k)\ominus (3/k) C_0^n$. As $\{y+(1/k)C_0^n:y\in (1/k)\Z^n\}$
is a covering of $\R^n$, there is a $y\in (1/k)\Z^n$ with $x\in (1/k)C_0^n+y$
and hence $y\in (1/k)C_0^n+x$.  It follows that
$$
x\in\frac1k (2C_0^n)+y\subset \frac3k C_0^n+x \subset K_0(2\delta_k).
$$
As the vertices of $(1/k)(2C_0^n)+y$ are in $(1/k)\Z^n$, we have $x\in
\conv\left(K_0(2\delta_k)\cap(1/k)\Z^n\right)$, proving the claim.

Let $t_k=(2\delta_k/V(K_0))^{1/n}$. The fact that $DK_0$ is convex and contains
the origin, \eqref{st1}, Lemma~\ref{sDelta} (with $\delta=2\delta_k$), and the
definition of $s(K_0)$ imply that
\begin{align*}
\left(1-\left(t_k+\frac3{s(K_0)k}\right)\right)DK_0&=
\left(1-t_k\right)DK_0\ominus \left(\frac3{s(K_0)k} DK_0\right) \subset
K_0(2\delta_k)\ominus \frac3{k} C_0^n.
\end{align*}
From this, \eqref {include}, and \eqref{halfway}, we obtain
$$
\left(1-\left(t_k+\frac3{s(K_0)k}\right)\right)DK_0\subset Q_k\subset DK_0.
$$
As $DK_0\subset \sqrt n B^n$, this yields
$$
\delta(DK_0,Q_k)\le \sqrt n\left(t_k+\frac3{s(K_0)k}\right)= \left(\sqrt
n\left(\frac2{V(K_0)}\right)^{1/n} + \frac{3\sqrt
n}{s(K_0)k\delta_k^{1/n}}\right)\delta_k^{1/n}.
$$
By \eqref{essent}, $k\delta_k^{1/n}\to\infty$ as $k\to\infty$, and
\eqref{convRate} follows.
\end{proof}

The estimate \eqref{convRate} reveals that the rate of convergence of $Q_k$ to
$DK_0$ depends on the asymptotic behavior of the threshold parameter
$\delta_k$, which is linked to the bandwidth $\ee_k$ by \eqref{essent}. If we

assume that $V(K_0)$ is bounded from below by a known constant, then
$c_{\ref{cleadingconst}}$ in the statement of Theorem~\ref{maindiffe} can be
chosen independent of $K_0$. We note the resulting rate of convergence as a
corollary, where we choose $\ee_k$ and $\delta_k$ as appropriate powers
of $k$. In particular, it shows that a convergence rate of $k^{-p}$ can be
attained, where $p$ is arbitrarily close to $1/4-3/(8n)$.

\begin{cor}\label{cordiffe}
Suppose that $K_0$, $\delta_k$, $\ee_k$, and $g_k$ are as in
Algorithm~NoisyCovDiff($\varphi$). Let $0<b<V(K_0)$, let
$\delta_k=k^{-(n-3\alpha n-3/2)/4}$, and let $\ee_k=k^{-\alpha}$, for some
$0<\alpha<1/3-1/(2n)$. If $Q_k$ is an output from Algorithm~NoisyCovDiff($\varphi$) as stated above, then, almost surely,
$$
\delta(Q_k,DK_0)\le \sqrt n \left(\frac{2}{b}\right)^{1/n}
k^{-(1-3\alpha-3/(2n))/4},
$$
for sufficiently large $k$.
\end{cor}

\begin{rem}\label{cordifferem}
{\em  Here we outline how a stronger assumption, but one that still applies to all the noise models of practical interest, on the random variables in Algorithm~NoisyCovDiff($\varphi$) leads to a better convergence rate in Corollary~\ref{cordiffe}.

Consider a family $\{X_{\alpha}: \alpha\in A\}$ of zero mean random variables
with variances $\sigma_{\alpha}^2$ that satisfy the hypothesis of Bernstein's inequality (see \cite[Theorem~5.2, p.~27]{BulK00} or \cite[Lemma~2.2.11]{VdVW}), that is,
\begin{equation}\label{Bernie1}
\left|E\left(X_{\alpha}^m\right)\right|\le \frac{m!}{2}\sigma_{\alpha}^2H^{m-2},
\end{equation}
for some $H>0$ and all $\alpha\in A$ and $m=2,3,\dots$, and also have uniformly bounded variances, that is,
\begin{equation}\label{Bernie2}
\sigma_{\alpha}^2\le \sigma^2,
\end{equation}
say, for all $\alpha\in A$.  If the family $\{X_1,\dots,X_r\}$ of independent zero mean random variables satisfies (\ref{Bernie1}) with $A=\{1,\dots,r\}$, then Bernstein's inequality states that
$${\mathrm{Pr}}\left(\left|\sum_{i=1}^rX_i\right|\ge \delta\right)\le
2\exp\left(-\frac{\delta^2}{2\left(\delta H+\sum_{i=1}^r\sigma_i^2\right)}\right),
$$
for all $\delta>0$.

Suppose that the random variables $N_{ik}$ in Algorithm~NoisyCovDiff($\varphi$) are row-wise independent, zero mean, and satisfy (\ref{Bernie1}) and (\ref{Bernie2}).  Then Bernstein's inequality can be applied in the proof of Lemma~\ref{expo}, together with (\ref{betabound}) and (\ref{sumbetabound}), to show that
\begin{equation}\label{exp1Bern}
{\mathrm{Pr}}( |g_k(x)-E\left(g_k(x)\right)|\ge \delta/2)\le
2\exp\left(-\frac{\delta^2(k\ee_k)^n}{4\|\varphi\|_\infty
(\delta H+2\sigma^2)}\right),
\end{equation}
for all $\delta>0$. (Compare the weaker upper bound in (\ref{markhi}).) As at the end of the proof of Lemma~\ref{expo}, this results in the same upper bound for ${\mathrm{Pr}}\,(|g_k(x)-g_{K_0}(x)|>\delta)$.  The improved bound (\ref{exp1Bern}), combined with the argument of Theorem~\ref{maindiffe}, leads to the assumption \refstepcounter{Cnmb}\label{444}
\begin{equation}\label{essentBern}
\liminf_{k\to\infty}\frac{\delta^2_k(k\ee_k)^n}{\log
k}> c_{\arabic{Cnmb}}(n+2),
\end{equation}
where $c_{\ref{444}}= 12\|\varphi\|_\infty \sigma^2$, instead of (\ref{essent}).  In Corollary~\ref{cordiffe} we take instead $\delta_k=k^{-n(1-\alpha)/2}\log k$ and  $\ee_k=k^{-\alpha}$, for some $0<\alpha<1$. The final conclusion is that if $Q_k$ is an output from Algorithm~NoisyCovDiff($\varphi$), then, almost surely,
$$
\delta(Q_k,DK_0)\le \sqrt n \left(\frac{2}{b}\right)^{1/n}
k^{-(1-\alpha)/2}(\log k)^{1/n},
$$
for sufficiently large $k$.  In particular, a convergence rate of $k^{-p}$ can be
attained, where $p$ is arbitrarily close to $1/2$.

Note that families of zero mean Gaussian and centered Poisson random variables satisfy (\ref{Bernie1}) and (\ref{Bernie2}).  Also, if two independent families with the same index set satisfy (\ref{Bernie1}) and (\ref{Bernie2}), the same is true for their sums (with possibly different constants $H$ and $\sigma^2$).}
\end{rem}

\section{Phase retrieval: Framework and technical lemmas}\label{PRI}

In this section we set the scene for our results on phase retrieval, beginning
with the necessary material from Fourier analysis.

Let $g$ be a continuous function on $\R^n$ whose support is contained in
$[-1,1]^n$ and let $L\ge 1$. By the classical theory, the Fourier series of $g$
is
$$\sum_{z\in\Z^n}c_ze^{i\pi z\cdot x/L},$$
for $x\in [-L,L]^n$, where
$$c_z=\frac{1}{(2L)^n}\int_{[-L,L]^n}g(t)e^{-i\pi z\cdot t/L}\,dt=
\frac{1}{(2L)^n}\int_{\R^n}g(t)e^{-i\pi z\cdot t/L}\,dt=\frac{1}{(2L)^n}
\widehat{g}(\pi z/L).$$
Let
$$\Z^n_k=\{z\in \Z^n: z=(z_1,\dots,z_n), |z_j|\le k, j=1,\dots,n\}.$$
If $g$ is also Lipschitz, then by \cite[Theorem~3]{Liu72}, the square partial
sums $\sum_{z\in\Z^n_k}c_ze^{i\pi z\cdot x/L}$ of the Fourier series of $g$
converge uniformly to $g$. Therefore, if $g$ is also an even function, we can
write
\begin{equation}\label{2}
g(x)=\frac{1}{(2L)^n}\sum_{z\in\Z^n}\widehat{g}(\pi z/L)e^{i\pi z\cdot x/L}=
\frac{1}{(2L)^n}\sum_{z\in\Z^n}\widehat{g}(\pi z/L)\cos\frac{\pi z\cdot x}{L},
\end{equation}
for all $x\in [-L,L]^n$, where equality is in the sense of uniform convergence
of square partial sums.

Let $\Z^n_k(+)$ be a subset of $\Z^n_k$ such that
\begin{equation}\label{halfarray}
\Z^n_k(+)\cap \left(-\Z^n_k(+)\right)=\emptyset\quad{\text{and}}\quad
\Z^n_k=\{o\}\cup\Z^n_k(+)\cup \left(-\Z^n_k(+)\right).
\end{equation}
Suppose that $g$ is even and for some fixed $0<\gamma<1$ and each $k\in \N$, we
can obtain noisy measurements
\begin{equation}\label{1}
\widetilde{g}_{z,k}=\widehat{g}(z/k^\gamma)+X_{z,k},
\end{equation}
of $\widehat{g}$, for $z\in \{o\}\cup \Z^n_k(+)$, where the $X_{z,k}$'s are
row-wise independent (i.e., independent for fixed $k$) zero mean random variables. Define $X_{z,k}=X_{-z,k}$, for
$z\in \left(-\Z^n_k(+)\right)$ and note that then $X_{z,k}=X_{-z,k}$ for all
$z\in \Z^n_k$. Since $g$ is even, $\widehat{g}$ is also even, and we have
$\widetilde{g}_{z,k}=\widetilde{g}_{-z,k}$ for $z\in \Z^n_k$. Using these
facts, (\ref{2}) with $L=\pi k^{\gamma}$, and (\ref{1}), we obtain
\begin{equation}\label{esti}
\frac{1}{(2\pi k^\gamma)^n}\sum_{z\in \Z^n_k}\widetilde{g}_{z,k}\cos
\frac{z\cdot x}{k^\gamma}=g(x)+\frac{1}{(2\pi k^\gamma)^n}\left(\sum_{z\in \Z^n_k}X_{z,k}\cos \frac{z\cdot
x}{k^\gamma}-\sum_{z\in \Z^n\setminus \Z^n_k}\widehat{g}\left(\frac{z}{k^\gamma}\right)\cos
\frac{z\cdot x}{k^\gamma}\right),
\end{equation}
for all $x\in [-\pi k^{\gamma},\pi k^{\gamma}]^n$.  Here the left-hand side is an
estimate of $g(x)$ and the second and third terms on the right-hand side are a random error and a deterministic error, respectively.

Since it has all the required properties, we can apply the previous equation to
the covariogram $g=g_{K_0}$ of a convex body $K_0$ contained in $C_0^n$, in
which case $\widehat{g_{K_0}}=|\widehat{1_{K_0}}|^2$.  In order to move closer
to the notation used earlier, we now use $i$ as an index and again list the
points in $[-1,1]^n\cap (1/k)\Z^n=(1/k)\Z^n_k$, but this time a little
differently.  We let $x_{0k}=o$, list the points in $(1/k)\Z^n_k(+)$ as
$x_{ik}$, $i=1,\dots,I_k'=\left((2k+1)^n-1\right)/2$, and then let
$x_{ik}=-x_{(-i)k}$ for $i=-I_k',\dots,-1$.  Now let
$z_{ik}=k^{1-\gamma}x_{ik}$, so that
$$(1/k^{\gamma})\Z^n_k=\{z_{ik}:i=-I_k',\dots,I_k'\}.$$
Setting $\widetilde{g}_{jk}=\widetilde{g_{K_0}}_{z_{jk},k}$ and
$X_{jk}=X_{z_{jk},k}$, we use (\ref{1}) to rewrite (\ref{esti}) as
\begin{equation}\label{newmeas}
M_k(x)=g_{K_0}(x)+N_k(x)-d_k(x),
\end{equation}
where
\begin{equation}\label{newmeasexp}
M_k(x)=\frac{1}{(2\pi k^\gamma)^n} \sum_{j=-I_k'}^{I_k'}\cos(z_{jk}\cdot
x)\widetilde{g}_{jk}
\end{equation}
is an estimate of $g_{K_0}$,
\begin{equation}\label{N}
N_k(x)=\frac{1}{(2\pi k^\gamma)^n}\sum_{j=-I_k'}^{I_k'}\cos(z_{jk}\cdot
x)X_{jk}
\end{equation}
is a random variable, and
\begin{equation}\label{M}
d_k(x)=\frac{1}{(2\pi k^\gamma)^n}\sum_{z\in \Z^n\setminus \Z^n_k}
\cos\left(\frac{z\cdot x}{k^\gamma}\right)\widehat{g_{K_0}}(z/k^\gamma)
\end{equation}
is a deterministic error.

We shall need three technical lemmas.  The first of these provides a control on
the deterministic error.

\begin{lem}\label{lemA}
Let $d_k=\sup\{|d_k(x)|: x\in \R^n\}$.  Then $d_k=O(k^{\gamma-1}(\log k)^n)$ as
$k\rightarrow\infty$.
\end{lem}

\begin{proof}
From (\ref{M}), the fact that  $\widehat{g_{K_0}}=|\widehat{1_{K_0}}|^2$ is
nonnegative, and (\ref{2}) with $g=g_{K_0}$ and $L=\pi k^{\gamma}$, we have
\begin{equation}\label{est_1}
 d_k \leq\frac{1}{(2\pi k^\gamma)^n}\sum_{z\in \Z^n\setminus \Z^n_k}
\widehat{g_{K_0}}(z/k^\gamma)=g_{K_0}(0)-\frac{1}{(2\pi k^\gamma)^n}\sum_{z\in
\Z^n_k} \widehat{g_{K_0}}(z/k^\gamma).
\end{equation}
For $t\in\R$, let
$$D_k(t)=\sum_{l=-k}^k e^{i lt}=\frac{\sin((k+1/2)t)}{\sin(t/2)}$$ be the
Dirichlet kernel. Note that for $x=(x_1,\dots,x_n)\in\R^n$, we have
$$
\sum_{z\in \Z^n_k}e^{i z\cdot x}=\prod_{l=1}^n\left(\sum_{l=-k}^k e^{i l
x_l}\right)=\prod_{l=1}^nD_k(x_l).
$$
Using this and the fact that $g_{K_0}$ is even, with support in $[-1,1]^n$, we
obtain
\begin{eqnarray}\label{step1}
\frac{1}{(2\pi k^\gamma)^n}\sum_{z\in \Z^n_k}\widehat{g_{K_0}}(z/k^\gamma)&=&
\frac{1}{(2\pi k^\gamma)^n}\sum_{z\in \Z^n_k}\int_{[-\pi k^\gamma,\pi
k^\gamma]^n}
g_{K_0}(x)e^{-i z\cdot x/k^\gamma}\,dx\nonumber\\
&=&\frac{1}{(2\pi k^\gamma)^n}\int_{[-\pi k^\gamma,\pi k^\gamma]^n}g_{K_0}(x)
\prod_{l=1}^nD_k(-x_l/k^\gamma)\,dx\nonumber\\
&=&\frac{1}{(2\pi)^n}\int_{[-1,1]^n}g_{K_0}(y
k^\gamma)\prod_{l=1}^nD_k(y_l)\,dy.
\end{eqnarray}

Since $\int_{-\pi}^\pi D_k(t)\,dt=2\pi$, we have
\begin{equation}\label{step2}
 g_{K_0}(0)=\frac{1}{(2\pi)^n}\int_{[-\pi,\pi]^n}g_{K_0}(0)\prod_{l=1}^nD_k(y_l)\,dy.
\end{equation}
Thus, by \eqref{est_1}, \eqref{step1}, and \eqref{step2},
\begin{eqnarray}\label{est_2}
 d_k&\leq& \left|\frac{1}{(2\pi)^n}\int_{[-1,1]^n}(g_{K_0}(0)-g_{K_0}(y k^\gamma))
 \prod_{l=1}^nD_k(y_l)\,dy\right|+\nonumber\\
&+&g_{K_0}(0)\left|\frac{1}{(2\pi)^n}\int_{[-\pi,\pi]^n\setminus[-1,1]^n}
\prod_{l=1}^nD_k(y_l)\,dy\right|.
\end{eqnarray}
By Proposition~\ref{Matprop}, $g_{K_0}$ is Lipschitz and hence the Lipschitz
norm of $g_{K_0}(y k^\gamma)$ is $O(k^\gamma)$.  Now \cite[Theorem~1]{Liu72}
implies that \refstepcounter{Cnmb}\label{Liuc}
\begin{equation}\label{est_3}
\left|\frac{1}{(2\pi)^n}\int_{[-1,1]^n}(g_{K_0}(0)-g_{K_0}(y k^\gamma))
\prod_{l=1}^nD_k(y_l)\,dy\right|\leq c_{\arabic{Cnmb}}
k^{\gamma-1}\sum_{l=0}^{n-1} (\log k)^{n-l},
\end{equation}
for some constant $c_{\ref{Liuc}}$ independent of $k$.  (In the statement of
\cite[Theorem~1]{Liu72}, $D_j(Y)$ should be $D_J(Y)$.  In that theorem we are
taking $\alpha=1$ and $J=(k,k,\dots,k)\in\Z^n$.)

In view of \eqref{est_2} and \eqref{est_3}, the proof will be complete if we
show that
\begin{equation}\label{case_n}
\int_{[-\pi,\pi]^n\setminus[-1,1]^n}\prod_{l=1}^nD_k(x_l)dx=O(1/k),
\end{equation}
as $k\to\infty$. To this end, observe that, by trigonometric addition formulas
and integration by parts,
\begin{eqnarray}\label{case_n=1}
\int_{-\pi}^{-1} D_k(t)\,dt=\int_1^\pi D_k(t)\,dt&=&\int_1^\pi
\frac{\sin(kt)\cos(t/2)}{\sin(t/2)}\,dt+\int_1^\pi \cos(kt)\,dt\nonumber\\
&=&\frac{\cos k
\cot(1/2)}{k}+\int_1^\pi\frac{\cos(kt)}{k}\frac{d}{dt}\left(\cot(t/2)\right)
\,dt-\frac{\sin k}{k}\nonumber\\
&=&O(1/k).
\end{eqnarray}
Now
$$
[-\pi,\pi]^n\setminus[-1,1]^n=\cup_{i=1}^n (A_i\cup B_i),$$ where
$$
A_i=\{(x_1,\dots,x_n)\ : -1\leq x_j\leq 1\text{ for $j<i$ }, 1\leq x_i\leq \pi,
-\pi\leq x_j\leq \pi\text{ for $j>i$}\}
$$
and $B_i=-A_i$. By \eqref{case_n=1}, we have, for each $i$,
\begin{align*}
\int_{A_i}\prod_{l=1}^nD_k(x_l)dx&=\left(\int_{-1}^1D_k(t) dt\right)^{i-1}
\int_{1}^\pi D_k(t) dt \left(\int_{-\pi}^\pi D_k(t) dt\right)^{n-i}\\
&=(2\pi-O(1/k))^{i-1}\ O(1/k)\ (2\pi)^{n-i}.
\end{align*}
Since $\inte(A_i)\cap\inte(A_j)=\emptyset$, for each $i,j$ with $i\neq j$,
$\inte(A_i)\cap\inte (B_j)=\emptyset$, for each $i,j$, and
$\prod_{l=1}^nD_k(x_l)$ is even, the previous estimate proves \eqref{case_n}.
\end{proof}

It is possible that the previous lemma could also be obtained via some
estimates proved in \cite{BHI03} for the rate of decay of
$\int_{S^{n-1}}|\widehat{1_{K_0}}(ru)|^2\,du$ as $r\rightarrow\infty$.

The next two lemmas will allow us to circumvent Proposition~\ref{stronglaw},
the version of the Strong Law of Large Numbers used earlier.

\begin{lem}\label{lemC}
Let $Y_{jk}$, $j=1,\dots,m_k$, $k\in\N$ be a triangular array of row-wise independent
zero mean random variables with uniformly bounded fourth moments, where $m_k\sim k^n$ as $k\rightarrow \infty$. Let $\nu$ and $a_{pqk}$, $p,q=1,\dots,m_k$ be constants such that
$|a_{pqk}|=O(k^{\nu})$ as $k\rightarrow\infty$ uniformly in $p$ and $q$, where
$2n-4n\gamma+2\nu<-1$. Then, almost surely,
$$Z_k=\frac{1}{(2\pi k^\gamma)^{2n}}\sum_{p,q=1}^{m_k}a_{pqk}Y_{pk}Y_{qk}\rightarrow 0,$$
as $k\rightarrow\infty$.
\end{lem}

\begin{proof}
Note that $E(Y_{pk}Y_{qk})=E(Y_{pk})E(Y_{qk})=0$ unless $p=q$. Therefore
$$
E(Z_k)=\frac{1}{(2\pi k^\gamma)^{2n}}\sum_{p,q=1}^{m_k}a_{pqk}E(Y_{pk}Y_{qk})=
\frac{1}{(2\pi k^\gamma)^{2n}}\sum_{p=1}^{m_k}a_{ppk}E(Y_{pk}^2).
$$
Since the $Y_{pk}$'s have uniformly bounded second moments,
$|E(Z_k)|=O(k^{n-2n\gamma+\nu})$ and hence $E(Z_k)$ converges to zero as
$k\rightarrow\infty$.

Let
$$v^{(k)}_{pqrs}=\cov(Y_{pk}Y_{qk},Y_{rk}Y_{sk})=E(Y_{pk}Y_{qk}Y_{rk}Y_{sk})
-E(Y_{pk}Y_{qk})E(Y_{rk}Y_{sk}).$$ If the cardinality of the set $\{p,q,r,s\}$
is 3 or 4, then at least one of the indices, say $p$, is different from all the
others and
$$
 v^{(k)}_{pqrs}=E(Y_{pk})E(Y_{qk} Y_{rk} Y_{sk})-E(Y_{pk})E(Y_{qk}) E( Y_{rk} Y_{sk})=0-0=0.
$$
If the cardinality of the set $\{p,q,r,s\}$ is 1, then
$$
 v^{(k)}_{pqrs}=v^{(k)}_{pppp}=E(Y_{pk}^4)-E(Y_{pk}^2)^2.
$$
If the cardinality of the set $\{p,q,r,s\}$ is 2, then either $p=q$, $r=s$ and
$p\neq r$, and
$$
 v^{(k)}_{pqrs}=v^{(k)}_{pprr}=E(Y_{pk}^2Y_{rk}^2)-E(Y_{pk}^2)E(Y_{rk}^2)=0,
$$
or $p=r$, $q=s$ and $p\neq q$, and
$$
 v^{(k)}_{pqrs}=v^{(k)}_{pqpq}=E(Y_{pk}^2Y_{qk}^2)-E(Y_{pk}Y_{qk})^2=
 E(Y_{pk}^2)E(Y_{qk}^2)-E(Y_{pk})^2E(Y_{qk})^2,
$$
or $p=s$, $q=r$ and $p\neq q$, and
$$
 v^{(k)}_{pqrs}=v^{(k)}_{pqqp}=E(Y_{pk}^2Y_{qk}^2)-E(Y_{pk}Y_{qk})^2=
 E(Y_{pk}^2)E(Y_{qk}^2)-E(Y_{pk})^2E(Y_{qk})^2.
$$
In view of the fact that the $Y_{jk}$'s have uniformly bounded fourth moments, the covariances $v^{(k)}_{pqrs}$ are also uniformly bounded, and hence
\begin{align*}
\var(Z_k)&=\frac{1}{(2\pi k^\gamma)^{4n}}\sum_{p,q,r,s=1}^{m_k}a_{pqk}a_{rsk}
v^{(k)}_{pqrs}\\
&=\frac{1}{(2\pi k^\gamma)^{4n}}\sum_{p=1}^{m_k}a_{ppk}^2v^{(k)}_{pppp}+
\frac{1}{(2\pi k^\gamma)^{4n}}\left(\sum_{p\neq
q=1}^{m_k}a_{pqk}^2v^{(k)}_{pqpq}+
\sum_{p\neq q=1}^{m_k}a_{pqk}a_{qpk}v^{(k)}_{pqqp}\right)\\
&= O\left(k^{2n-4n\gamma+2\nu}\right).
\end{align*}
Let $\ee>0$. For sufficiently large $k$, we have $\ee-E(Z_k)>0$, and for such
$k$, by Chebyshev's inequality,
$$
\Pr(Z_k>\ee)=\Pr\bigl(Z_k-E(Z_k)>\ee-E(Z_k)\bigr)\leq
\frac{\var(Z_k)}{(\ee-E(Z_k))^2}=O\left(k^{2n-4n\gamma+2\nu}\right).
$$
Our hypothesis and the Borel-Cantelli Lemma imply that, almost surely, $Z_k$
converges to zero, as $k\rightarrow\infty$.
\end{proof}

\begin{lem}\label{lemCs}
Let $Y_{jk}^{(r)}$, $j=1,\dots,m_k$, $r=1,2$, $k\in\N$, be a triangular array of row-wise independent (i.e., independent for fixed $k$) zero mean random variables with uniformly bounded fourth moments, where $m_k\sim k^n$ as $k\rightarrow \infty$. Let $\nu$ and $a_{pqk}$, $p,q=1,\dots,m_k$ be constants such that $|a_{pqk}|=O(k^{\nu})$ as $k\rightarrow\infty$ uniformly in $p$ and $q$, where
$2n-4n\gamma+2\nu<-1$. Then, almost surely,
$$\overline{Z}_k=\frac{1}{(2\pi k^\gamma)^{2n}}\sum_{p,q=1}^{m_k}a_{pqk}Y_{pk}^{(1)}
Y_{pk}^{(2)} Y_{qk}^{(1)}Y_{qk}^{(2)}\rightarrow 0,$$ as $k\rightarrow\infty$.
\end{lem}

\begin{proof}
As in the proof of Lemma~\ref{lemC}, we have
$$
E(\overline{Z}_k)= \frac{1}{(2\pi
k^\gamma)^{2n}}\sum_{p=1}^{m_k}a_{ppk}E\left(\left(Y_{pk}^{(1)}\right)^2\right)
E\left(\left(Y_{pk}^{(2)}\right)^2\right),
$$
so $|E(\overline{Z}_k)|=O(k^{n-2n\gamma+\nu})$ and hence $E(\overline{Z}_k)$
converges to zero as $k\rightarrow\infty$.

Let
$$w^{(k)}_{pqrs}=\cov\left(Y_{pk}^{(1)}Y_{pk}^{(2)}Y_{qk}^{(1)}Y_{qk}^{(2)},
Y_{rk}^{(1)}Y_{rk}^{(2)}Y_{sk}^{(1)}Y_{sk}^{(2)}\right).$$
Straightforward modifications to the proof of Lemma~\ref{lemC} and the assumption of uniformly bounded fourth moments yield
\begin{align*}
\var(\overline{Z}_k)&=\frac{1}{(2\pi k^\gamma)^{4n}}\sum_{p,q,r,s=1}^{m_k}a_{pqk}a_{rsk}
w^{(k)}_{pqrs}\\
&=\frac{1}{(2\pi k^\gamma)^{4n}}\sum_{p=1}^{m_k}a_{ppk}^2w^{(k)}_{pppp}+
\frac{1}{(2\pi k^\gamma)^{4n}}\left(\sum_{p\neq
q=1}^{m_k}a_{pqk}^2w^{(k)}_{pqpq}+
\sum_{p\neq q=1}^{m_k}a_{pqk}a_{qpk}w^{(k)}_{pqqp}\right)\\
&= O\left(k^{2n-4n\gamma+2\nu}\right).
\end{align*}
The proof is concluded as in Lemma~\ref{lemC}.
\end{proof}

\section{Phase retrieval from the squared modulus}\label{PRII}

This section addresses Problem~2 in the introduction.

\bigskip

{\large{\bf Algorithm~NoisyMod$^2$LSQ}}

\bigskip

{\it Input:} Natural numbers $n\ge 2$ and $k$; a real number $\gamma$ such that
$0<\gamma<1$; noisy measurements
\begin{equation}\label{measPR}
\widetilde{g}_{ik}=|\widehat{1_{K_0}}(z_{ik})|^2+X_{ik},
\end{equation}
of the squared modulus of the Fourier transform of the characteristic function
of an unknown convex body $K_0\subset C_0^n$ whose centroid is at the origin,
at the points in
$$\{z_{ik}: i=0,1,\dots,I_k'\}=\{o\}\cup (1/k^{\gamma})\Z^n_k(+),$$
where $\Z^n_k(+)$ satisfies (\ref{halfarray}) and where the $X_{ik}$'s are row-wise
independent zero mean random variables with uniformly bounded fourth moments; an $o$-symmetric convex polytope $Q_k$ in $\R^n$, stochastically independent of the measurements
$\widetilde{g}_{ik}$, that approximates either $\nabla K_0$ or $DK$, in the
sense that, almost surely,
$$
\lim_{k\to\infty} \delta(Q_k,\nabla K_0)=0,\quad\text{ or}\quad
\lim_{k\to\infty} \delta(Q_k,D K_0)=0.
$$

\smallskip

{\it Task:} Construct a convex polytope $P_k$ that approximates $K_0$, up to
reflection in the origin.

\smallskip

{\it Action:}

1.  Let $\widetilde{g}_{ik}=\widetilde{g}_{(-i)k}$, for $i=-I_k',\dots,-1$, let
$x_{ik}=k^{\gamma-1}z_{ik}$, $i=-I_k',\dots,I_k'$ be the points in the cubic
array $2C_0^n\cap (1/k)\Z^n$, and let
\begin{equation}\label{mod2meas}
M_k(x_{ik})=\frac{1}{(2\pi k^\gamma)^n} \sum_{j=-I_k'}^{I_k'} \cos(z_{jk}\cdot
x_{ik})\widetilde{g}_{jk},
\end{equation}
for $i=-I_k',\dots,I_k'$.

2.  Run Algorithm~NoisyCovLSQ with inputs $n$, $k$, $Q_k$, and with $M_{ik}$
replaced by $M_k(x_{ik})$, for $i=-I_k',\dots,I_k'$ and with the obvious
re-indexing in $i$. The resulting output $P_k$ of that algorithm is also the
output of the present one.

\bigskip

The main result in this section corresponds to Theorem~\ref{maincov} above.  We
first state it, and then show that it can be proved by suitable modifications
to the proof of Theorem~\ref{maincov} if in addition $\gamma>1/2+1/(4n)$.

\begin{thm}\label{maincovPR}
Suppose that $K_0\subset C_0^n$ is a convex body with its centroid at the
origin.  Suppose also that $K_0$ is determined, up to translation and
reflection in the origin, among all convex bodies in $\R^n$, by its
covariogram.  Let
\begin{equation}\label{gamma}
1/2+1/(4n)<\gamma<1.
\end{equation}
If $P_k$, $k\in\N$, is an output from Algorithm~NoisyMod$\,^2$LSQ as stated
above, then, almost surely,
$$
\min\{\delta(K_0,P_k),\delta(-K_0,P_k)\}\to 0
$$
as $k\to\infty$.
\end{thm}

As we shall now show, the proof of this theorem basically follows the analysis
given in Section~\ref{convergence}.  Of course, alterations must be made, since
the measurements $M_{ik}$ in Algorithm~NoisyCovLSQ have been replaced by the
new measurements $M_k(x_{ik})$ defined by (\ref{mod2meas}) or equivalently by
(\ref{newmeasexp}) with $x=x_{ik}$.  In view of (\ref{newmeas}), we have
$$M_k(x_{ik})=g_{K_0}(x_{ik})+N_k(x_{ik})-d_k(x_{ik}),$$
$i=-I_k',\dots,I_k'$, where $N_k(x_{ik})$ and $d_k(x_{ik})$ are given by
(\ref{N}) and (\ref{M}), respectively, with $x=x_{ik}$.

We begin with a lemma. Note that $I_k=2I_k'+1$, so the expression in the lemma
is the sample mean.  Also, recall that by their definition, the random
variables $X_{ik}$ have uniformly bounded fourth moments, and $X_{pk}$ and $X_{qk}$ are independent unless $p=\pm q$, in which case they are equal.

\begin{lem}\label{lemB}
Let $N_k(x_{ik})^+=\max\{N_k(x_{ik}),0\}$ for all $i$ and $k$.  If
(\ref{gamma}) holds, then, almost surely,
$$
 \frac1{I_k}\sum_{i=-I_k'}^{I_k'}N_k(x_{ik})^+\to0,
$$
as $k\rightarrow\infty$.
\end{lem}

\begin{proof}
Note firstly that
$$
\frac1{I_k}\sum_{i=-I_k'}^{I_k'} N_k(x_{ik})^+\leq
\frac1{I_k}\sum_{i=-I_k'}^{I_k'} |N_k(x_{ik})| \leq
\left(\frac1{I_k}\sum_{i=-I_k'}^{I_k'} N_k(x_{ik})^2\right)^{1/2}.
$$
Thus it suffices to prove that, almost surely,
$$
S_k=\frac1{I_k}\sum_{i=-I_k'}^{I_k'} N_k(x_{ik})^2\rightarrow 0,
$$
as $k\rightarrow\infty$.

We have
\begin{align*}
S_k&=\frac1{I_k}\sum_{i=-I_k'}^{I_k'} \left(\frac{1}{(2\pi
k^\gamma)^n}\sum_{p=-I_k'}^{I_k'}\cos(z_{pk}\cdot x_{ik})
X_{pk}\right)^2\\
&=\frac{1}{(2\pi k^\gamma)^{2n}}\sum_{p,q=-I_k'}^{I_k'}\left( \frac1{I_k}
\sum_{i=-I_k'}^{I_k'} \cos(z_{pk}\cdot x_{ik})\cos(z_{qk}\cdot x_{ik})\right)
X_{pk}X_{qk}\\
&=\frac{1}{(2\pi k^\gamma)^{2n}}\sum_{p,q=-I_k'}^{I_k'}c_{pqk}X_{pk}X_{qk},
\end{align*}
say. Since $c_{(-p)qk}=c_{p(-q)k}=c_{pqk}$, it is clearly enough to show that,
almost surely,
$$\frac{1}{(2\pi k^\gamma)^{2n}}\sum_{p,q=1}^{I_k'}c_{pqk}X_{pk}X_{qk}\rightarrow 0,$$
as $k\rightarrow\infty$.  In view of (\ref{gamma}) and the fact that
$|c_{pqk}|=O(1)$, this follows from Lemma~\ref{lemC} with $Y_{jk}=X_{jk}$,
$m_k=I_k'$, $a_{pqk}=c_{pqk}$ for all $p$, $q$, and $k$, and $\nu=0$.
\end{proof}

\noindent{\it{Proof of Theorem~\ref{maincovPR}}}. We shall indicate the
modifications needed in Section~\ref{convergence}.  No changes are required in
the lemmas before Lemma~\ref{lem1}. For the latter, we shall use the same
notation as before, with the understanding that the indexing has changed and
the new random variables $N_k(x_{ik})$ replace the random variables $N_{ik}$ of
Section~\ref{convergence}. Thus we write
$$|f|_{I_k}=\left(\frac{1}{I_k}\sum_{i=-I_k'}^{I_k'}f(z_i)^2\right)^{1/2},$$
with corresponding changes in indexing in the definitions of $\vx_{I_k}$,
$\vN_{I_k}$, and $\Psi$. With the same proof as Lemma~\ref{lem1}, we now have
the inequality
\begin{eqnarray}\label{normboundnew}
\left|g_{K_0}-g_{P_k}\right|_{I_k}^2&\le &
2{\Psi}(P_k,\vx_{I_k},\vN_{I_k})-2{\Psi} (P(a)\cap
C_0^n,\vx_{I_k},\vN_{I_k})+\left|g_{K_0}-g_{P(a)\cap C_0^n}\right|_{I_k}^2+\nonumber\\
& & +\frac{2}{I_k}\sum_{i=-I_k'}^{I_k'}\left(g_{P(a)\cap C_0^n}(x_{ik})
-g_{P_k}(x_{ik})\right)d_k(x_{ik}),
\end{eqnarray}
instead of (\ref{normbound}).

Proposition~\ref{prp2} and Lemma~\ref{net} are unchanged. We do not require
Proposition~\ref{stronglaw} in order to conclude as in Lemma~\ref{bound} that,
almost surely,
\begin{equation}\label{38}
\sup_{K\in{\mathcal K}^n(C_0^n)}{\Psi}(K,\vx_{I_k},\vN_{I_k})\to 0,
\end{equation}
as $k\to\infty$. Indeed, it is enough to show that, almost surely, the new
expression corresponding to (\ref{WWW}), namely,
$${W}_k(\ee)=\max_{j=1,\ldots,m} \left\{\frac  {1}{I_k}\sum_{i=-I_k'}^{I_k'}
g_j^U(x_{ik}) N_k(x_{ik})^{\,+}-\frac  {1}{I_k}\sum_{i=-I_k'}^{I_k'}
g_j^L(x_{ik}) N_k(x_{ik})^{\,-}\right\},
$$
converges to zero, as $k\rightarrow\infty$.  This follows from
Lemma~\ref{lemB}, because the coefficients $g^U_j(x_{ik})$ and $g^U_j(x_{ik})$
are uniformly bounded by 1 and Lemma~\ref{lemB} holds both when such
coefficients are inserted and when $N_k(x_{ik})^+$ is replaced by
$N_k(x_{ik})^-=N_k(x_{ik})-N_k(x_{ik})^+=\max\{-N_k(x_{ik}),0\}$.

All this is enough to ensure that Lemma~\ref{pseudcov} still holds.  Indeed,
since a translate of $P_k$ is contained in $C_0^n$, and
${\Psi}(P_k,\vx_{I_k},\vN_{I_k})$ is unchanged by such a translation, we know
from (\ref{38}) that, almost surely, the first and second terms on the
right-hand side of (\ref{normboundnew}) converge to zero, as
$k\rightarrow\infty$. We have $g_{P(a)\cap C_0^n}(x_{ik})\le 1$ and
$g_{P_k}(x_{ik})\le V(2C_0^n)$, since $P_k\subset 2C_0^n$, and then
Lemma~\ref{lemA} implies that the new fourth term on the right-hand side of
(\ref{normboundnew}) converges to zero as $k\rightarrow\infty$. The rest of the
proof of Lemma~\ref{pseudcov} proceeds as before.

The proof of the main Theorem~\ref{maincov} now applies without change. \qed

\medskip

The user of Algorithm~NoisyMod$^2$LSQ must supply as input an $o$-symmetric
convex polytope $Q_k$ in $\R^n$ that approximates either $\nabla K_0$ or $DK$.
For this purpose we provide two algorithms that do the work of
Algorithm~NoisyCovBlaschke and Algorithm~NoisyCovDiff($\varphi$).

\bigskip

{\large{\bf Algorithm~NoisyMod$^2$Blaschke}}

\bigskip

{\it Input:} Natural numbers $n\ge 2$ and $k$; a positive real number $h_k$;
mutually nonparallel vectors $u_i\in S^{n-1}$, $i=1,\dots,k$ that span $\R^n$;
noisy measurements
\begin{equation}\label{measPRblas}
\widetilde{g}_{ik}=|\widehat{1_{K_0}}(z_{ik})|^2+X_{ik},
\end{equation}
of the squared modulus of the Fourier transform of the characteristic function
of an unknown convex body $K_0\subset C_0^n$ whose centroid is at the origin,
at the points in
$$\{z_{ik}: i=0,1,\dots,I_k'\}=\{o\}\cup (1/k^{\gamma})\Z^n_k(+),$$
where $\Z^n_k(+)$ satisfies (\ref{halfarray}) and where the $X_{ik}$'s are
row-wise independent zero mean random variables with uniformly bounded fourth moments.

\smallskip

{\it Task:} Construct an $o$-symmetric convex polytope $Q_k$ that approximates
the Blaschke body $\nabla K_0$.

\smallskip

{\it Action:}

1.  Let $\widetilde{g}_{ik}=\widetilde{g}_{(-i)k}$, for $i=-I_k',\dots,-1$, and
let
$$
M_k(o)=\frac{1}{(2\pi k^\gamma)^n} \sum_{j=-I_k'}^{I_k'}\widetilde{g}_{jk}
\quad{\text{and}}\quad M_k(h_ku_i)=\frac{1}{(2\pi k^\gamma)^n}
\sum_{j=-I_k'}^{I_k'}\cos(z_{jk}\cdot h_ku_i)\widetilde{g}_{jk},
$$
for $i=1,\dots,k$. Then for $i=1,\dots,k$, let
\begin{equation}\label{yblas}
y_{ik}=\frac{M_k(o)-M_k(h_ku_i)}{h_k}.
\end{equation}

2. With the natural numbers $n\ge 2$ and $k$, and vectors $u_i\in S^{n-1}$,
$i=1,\dots,k$ use the quantities $y_{ik}$ instead of noisy measurements of the
brightness function $b_K(u_i)$ as input to Algorithm~NoisyBrightLSQ (see
\cite[p.~1352]{GKM06}).  The output of the latter algorithm is $Q_k$.

\bigskip

We shall show that the argument of Section~\ref{Algorithm} can be modified to
yield a convergence result corresponding to Theorem~\ref{covB}.  It is clear
that any such result must require the input $h_k$ to satisfy $h_k\rightarrow 0$
as $k\rightarrow\infty$, but we need a stronger condition phrased in terms of
parameters $\ee$ and $\gamma$ that satisfy (\ref{eeg}).  Since the second
inequality in (\ref{eeg}) is equivalent to $\gamma>(2n+5-4\ee)/(4n+4)$, which
decreases as $n$ increases and equals $(9-4\ee)/12$ when $n=2$, it is possible
to choose $\gamma$ and $\ee$ so that (\ref{eeg}) is satisfied.  Specifically,
one can choose $3/4\le \gamma<1$ and $0<\ee<1-\gamma$. Note also that
(\ref{eeg}) implies (\ref{gamma}).

There is considerable flexibility in the choice of the parameter $h_k$, and it
would be possible to introduce a further parameter $q_k$ by working with input
vectors $u_i\in S^{n-1}$, $i=1,\dots,q_k$, where $q_k\rightarrow\infty$ as
$k\rightarrow\infty$.  To avoid overcomplicating the exposition, however, we
shall not discuss this any further.

\begin{thm}\label{covBPR}
Let $K_0\subset C_0^n$ be a convex body with its centroid at the origin. Let
$(u_i)$ be a sequence in $S^{n-1}$ such that $(u_i^*)$ is evenly spread.
Suppose that $h_k\sim k^{\gamma-1+\ee}$, $k\in \N$, where $\ee$ and $\gamma$
satisfy
\begin{equation}\label{eeg}
0<\ee<1-\gamma\quad{\text{and}}\quad 2n-4n\gamma+4(1-\gamma-\ee)<-1.
\end{equation}
If $Q_k$ is an output from Algorithm~NoisyMod$\,^2$Blaschke as stated above,
then, almost surely,
$$
\lim_{k\to\infty} \delta(\nabla K_0,Q_k)=0.
$$
\end{thm}

\begin{proof} We shall indicate the changes needed in Section~\ref{Algorithm}.
Note that by (\ref{yblas}), and (\ref{newmeas}) with $x=o$ and $x=h_ku_i$, we
have
$$
y_{ik}=\frac{M_k(o)-M_k(h_ku_i)}{h_k}
=\frac{g_{K_0}(o)-g_{K_0}(h_ku_i)}{h_k}+\frac{N_k(o)-N_k(h_ku_i)}{h_k}
-\frac{d_k(o)-d_k(h_ku_i)}{h_k},
$$
for $i=1,\dots,k$, where $N_k(o)$, $d_k(o)$, $N_k(h_ku_i)$, and $d_k(h_ku_i)$
are given by (\ref{N}) and (\ref{M}) with $x=o$ or $x=h_ku_i$, as appropriate.

Lemma~\ref{unif1} is unchanged. Turning to the proof of
Lemma~\ref{LEMbasicinequality}, we now have
$${y}_{ik}=\zeta_{ik}+T_{ik},$$
where
\begin{equation}\label{XX}
\zeta_{ik}=\frac{g_{K_0}(o)-g_{K_0}(h_ku_i)}{h_k}-\frac{d_k(o)-d_k(h_ku_i)}{h_k}
\quad{\text{and}}\quad T_{ik}=\frac{N_k(o)-N_k(h_ku_i)}{h_k},
\end{equation}
for $i=1,\dots,k$. Since $h_k\sim k^{\gamma-1+\ee}$ for $0<\ee<1-\gamma$, the
second term in the previous expression for $\zeta_{ik}$ converges to zero as
$k\rightarrow\infty$, by Lemma~\ref{lemA}, and hence $\zeta_{ik}\rightarrow
b_{K_0}(u_i)$ as $k\rightarrow\infty$, as before, for $i=1,\dots,k$.  Moreover,
$$b_{K_0}(u_i)-\zeta_{ik}=\left(b_{K_0}(u_i)-\frac{g_{K_0}(o)-g_{K_0}(h_ku_i)}
{h_k}\right)+\frac{d_k(o)-d_k(h_ku_i)}{h_k},$$ so arguing as in the proof of
Lemma~\ref{LEMbasicinequality}, we use Lemma~\ref{unif1} with $t=h_k$ to obtain
(\ref{new42}) with $t=h_k$, that is,
$$0\le b_{K_0}(u_i)-\frac{g_{K_0}(o)-g_{K_0}(h_ku_i)}{h_k}\le \frac{(n-1)h_k}{2r}
b_{K_0}(u_i),$$ if $h_k\le 2r$.  We also have
$$\frac{d_k(o)-d_k(h_ku_i)}{h_k}=O(k^{-\ee}),$$
by Lemma~\ref{lemA}, so there is a constant \refstepcounter{Cnmb}\label{PR12}
$c_{\arabic{Cnmb}}=c_{\arabic{Cnmb}}(n,r)$ such that
$$|b_{K_0}(u_i)-\zeta_{ik}|\le c_{\ref{PR12}}k^{-\beta},$$
for $\beta=\min\{\ee,1-\gamma+\ee\}$, and all $k\in\N$ and $i=1,\dots,k$. The
rest of the proof of Lemma~\ref{LEMbasicinequality} can be followed, yielding
that, almost surely, there is a constant \refstepcounter{Cnmb}\label{PR11}
$c_{\arabic{Cnmb}}=c_{\arabic{Cnmb}}(n,r)$ such that
\begin{equation}\label{binew}
|b_{K_0}-b_{Q_k}|_{k}^2\le 2{\Psi}(Q_k,(u_i),{\mathbf{T}}_k)-2
{\Psi}(K_0,(u_i),{\mathbf{T}}_k)+ \frac{c_{\ref{PR11}}}{k^{\beta}}|
b_{K_0}-b_{Q_k}|_{k},
\end{equation}
for all $k\in\N$.  (Again, we assume that the obvious changes are made in the
notation.)

The next task is to check that Lemma~\ref{ub} still holds.  With (\ref{binew})
in hand, this rests on proving that, almost surely,
$$V_k=\frac{1}{k}\sum_{i=1}^{k}T_{ik}^2$$
is bounded.  In fact we claim that, almost surely, $V_k\rightarrow 0$ as
$k\rightarrow\infty$.  To see this, note that
\begin{eqnarray*}
V_k&=&\frac{1}{k}\sum_{i=1}^{k}\left(\frac{N_k(o)-N_k(h_ku_i)}{h_k}\right)^2\\
&=&\frac{1}{k}\sum_{i=1}^{k}\left( \frac{1}{(2\pi
k^\gamma)^n}\sum_{j=-I_k'}^{I_k'}
\left(\frac{1-\cos(z_{jk}\cdot h_ku_i)}{h_k}\right)X_{jk}\right)^2\\
&=&\frac{1}{(2\pi k^\gamma)^{2n}}\sum_{p,q=-I_k'}^{I_k'}a_{pqk}X_{pk}X_{qk},
\end{eqnarray*}
where
\begin{equation}\label{anew}
a_{pqk}=\frac{1}{kh_k^2}\sum_{i=1}^{k}\bigl(1-\cos(z_{pk}\cdot h_ku_i)\bigr)
\bigl(1-\cos(z_{qk}\cdot h_ku_i)\bigr)
\end{equation}
and hence $|a_{pqk}|\le 4/h_k^2$.  As in the proof of Lemma~\ref{lemB}, we may
take the indices $p,q$ from 1 to $I_k'$, and then, by (\ref{eeg}), the claim
follows from Lemma~\ref{lemC} with $m_k=I_k'$ and $\nu=2(1-\gamma-\ee)$.

At this stage the work for Lemma~\ref{extend} is already done.  Indeed, by the
Cauchy-Schwarz inequality,
$${\Psi}(Q_k,(u_i),{\mathbf{T}}_k)-
{\Psi}(K_0,(u_i),{\mathbf{T}}_k)\le |b_{K_0}-b_{Q_k}|_k
\left(\frac{1}{k}\sum_{i=1}^{k}T_{ik}^2\right)^{1/2}=|b_{K_0}-b_{Q_k}|_k\,
V_k^{1/2}.$$ Using this and (\ref{binew}) we see that, almost surely,
$$|b_{K_0}-b_{Q_k}|_k\le 2V_k^{1/2}+\frac{c_{\ref{PR11}}}{k^{\beta}}\rightarrow 0,$$
as $k\rightarrow\infty$.

Finally, the proof of Theorem~\ref{covB} can be applied without change.
\end{proof}

The next algorithm corresponds to Algorithm~NoisyCovDiff($\varphi$).  As for
that algorithm, $\varphi$ is a nonnegative bounded measurable function on
$\R^n$ with support in $C_0^n$, such that $\int_{\R^n}\varphi(x)\,dx=1$.

\bigskip

{\large{\bf Algorithm~NoisyMod$^2$Diff($\varphi$)}}

\bigskip

{\it Input:} Natural numbers $n\ge 2$ and $k$; positive reals $\delta_k$ and
$\ee_k$; a real number $\gamma$ satisfying $0<\gamma<1$; noisy measurements
\begin{equation}\label{measPR3}
\widetilde{g}_{ik}=|\widehat{1_{K_0}}(z_{ik})|^2+X_{ik},
\end{equation}
of the squared modulus of the Fourier transform of the characteristic function
of an unknown convex body $K_0\subset C_0^n$ whose centroid is at the origin,
at the points in
$$\{z_{ik}: i=0,1,\dots,I_k'\}=\{o\}\cup (1/k^{\gamma})\Z^n_k(+),$$
where $\Z^n_k(+)$ satisfies (\ref{halfarray}) and where the $X_{ik}$'s are
row-wise independent zero mean random variables with uniformly bounded fourth moments.

\smallskip

{\it Task:} Construct an $o$-symmetric convex polytope $Q_k$ in $\R^n$ that
approximates the difference body $DK_0$.

\smallskip

{\it Action:}

1.  Let $\widetilde{g}_{ik}=\widetilde{g}_{(-i)k}$, for $i=-I_k',\dots,-1$, let
$x_{ik}=k^{\gamma-1}z_{ik}$, $i=-I_k',\dots,I_k'$ be the points in the cubic
array $2C_0^n\cap (1/k)\Z^n$, and let
\begin{equation}\label{mod2meas3}
M_k(x_{ik})=\frac{1}{(2\pi k^\gamma)^n} \sum_{j=-I_k'}^{I_k'} \cos(z_{jk}\cdot
x_{ik})\widetilde{g}_{jk},
\end{equation}
for $i=-I_k',\dots,I_k'$.

2. Run Algorithm~NoisyCovDiff($\varphi$) with inputs $n$, $k$, $\delta_k$,
$\ee_k$, and $M_{ik}$ replaced by $M_k(x_{ik})$, for $i=-I_k',\dots,I_k'$ and
with the obvious re-indexing in $i$. The output $Q_k$ of that algorithm is also
the output of the present one.

\bigskip

We shall show that the argument in Section~\ref{diffe} used to prove
Theorem~\ref{maindiffe} can be modified to yield the following convergence
result.

\begin{thm}\label{maindiffePR}
Suppose that $K_0$, $\delta_k$, $\ee_k$, and $g_k$ are as in
Algorithm~NoisyMod$\,^2$Diff($\varphi$). Assume that
$\lim_{k\to\infty}\ee_k=\lim_{k\to\infty}\delta_k=0$ and that
\begin{equation}\label{essentphase}
\liminf_{k\to\infty}\delta^4_kk^{4\gamma n-3n-3/2}>0,
\end{equation}
where $\gamma>3(1+1/(2n))/4$. If $Q_k$
is an output from Algorithm~NoisyMod$\,^2$Diff($\varphi$) as stated above,
then, almost surely,
$$
\delta(DK_0,Q_k)\le c_{\ref{cleadingconst}}\delta_k^{1/n},
$$
for sufficiently large $k$. In particular, almost surely, $Q_k$ converges to
$DK_0$ as $k\rightarrow\infty$.
\end{thm}

\begin{proof}
Algorithm~NoisyMod$^2$Diff($\varphi$) can be regarded formally as
Algorithm~NoisyCovDiff($\varphi$) with $M_{ik}$ and $N_{ik}$ replaced by
$M_k(x_{ik})$ defined by (\ref{mod2meas3}) and $N_k(x_{ik})-d_k(x_{ik})$ defined by
(\ref{N}) and (\ref{M}) with $x=x_{ik}$, respectively. We follow the arguments of
Section~\ref{diffe} with this substitution in mind.

For Lemma~\ref{unifAsympUnbias}, we note first that by (\ref{N}),
$E(N_k(x_{ik}))=0$ for all $i$ and $k$. The same calculations as in the proof
of Lemma~\ref{unifAsympUnbias} lead to
$$|E(g_k(x))-g_{K_0}(x)|\le n(\ee_k+1/k)+d_k,$$
where $d_k$ is as in Lemma~\ref{lemA}.  By that lemma, $d_k\rightarrow 0$ as
$k\rightarrow\infty$ and hence the second statement in
Lemma~\ref{unifAsympUnbias} still holds.

Next, for Lemma~\ref{expo}, recall the definition (\ref{beta}) of $\beta_{ik}(x)$.
Then we have, by (\ref{N}),
\begin{eqnarray*}
g_k(x)-E(g_k(x))&=&\sum_{i=-I_k'}^{I_k'}\beta_{ik}(x)N_k(x_{ik})\\
&=&\frac{1}{(2\pi k^{\gamma})^n}\sum_{j=-I_k'}^{I_k'}
\left(\sum_{i=-I_k'}^{I_k'}\beta_{ik}(x)\cos(z_{jk}\cdot x_{ik})\right)X_{jk}\\
&=&\frac{1}{(2\pi k^{\gamma})^n}\sum_{j=-I_k'}^{I_k'}\xi_{jk}(x)X_{jk},
\end{eqnarray*}
say.  This is a weighted sum of independent random variables, so we can apply
Khinchine's inequality (see, for example, \cite[(4.32.1), p.~307]{Hof94} with $\alpha=4$) to obtain
$$E\left(\left|\sum_{i=-I_k'}^{I_k'}\beta_{ik}(x)N_k(x_{ik})\right|^4\right)\le
\frac{c(2k+1)^n}{(2\pi k^{\gamma})^{4n}}\sum_{j=-I_k'}^{I_k'}E\left|\xi_{jk}(x)X_{jk}\right|^4.$$
for some constant $c>0$. Also,
$$
|\xi_{jk}(x)|^4\le \left(\sum_{i=-I_k'}^{I_k'}\beta_{ik}(x)\right)^4
\le 1,
$$
by (\ref{sumbetabound}). The same argument as in the proof of Lemma~\ref{expo} now leads to the conclusion that there are constants \refstepcounter{Cnmb}\label{caa22PR} $c_{\arabic{Cnmb}}=c_{\arabic{Cnmb}}(\varphi)$ and
\refstepcounter{Nnmb}\label{naPR}
$N_{\arabic{Nnmb}}=N_{\arabic{Nnmb}}((\ee_k),n)\in \N$ such that if $\delta>0$, then
\begin{equation}\label{newprob}
\Pr(|g_k(x)-g_{K_0}(x)|>\delta)\le c_{\ref{caa22PR}}(2k+1)^{2n}k^{-4\gamma n}\delta^{-4},
\end{equation}
for all $k\ge N_{\ref{naPR}}$ and all $x\in \R^n$.   (Compare (\ref{exp1}).)

Lemma~\ref{sDelta} is unchanged.  With (\ref{essentphase}) instead of the hypothesis (\ref{essent}) of Theorem~\ref{maindiffe}, and the new estimate (\ref{newprob}), we arrive in the proof of Theorem~\ref{maindiffe} at the estimate
$$\Pr(a_k\ge \delta_k)\le
c_{\ref{caa22PR}}(2k+1)^{3n}k^{-4\gamma n}\delta_k^{-4}=O(k^{-3/2}),$$
so the Borel-Cantelli lemma can be used as before. This is all that is required to allow the proof of Theorem~\ref{maindiffe} to go through until near the end, when we use the fact that $k\delta_k^{1/n}\rightarrow\infty$ as
$k\rightarrow\infty$.   By (\ref{essentphase}) and the fact that $\gamma<1$,  this still holds.  Then the conclusion is the same, namely that, almost surely,
$$\delta(DK_0,Q_k)\le c_{\ref{cleadingconst}}\delta_k^{1/n},$$
for sufficiently large $k$.
\end{proof}

Concerning Corollary~\ref{cordiffe}, by using $\gamma>3(1+1/(2n))/4$ and (\ref{essentphase}) instead of (\ref{essent}), we can achieve a convergence rate arbitrarily close to $k^{-1/4+3/(8n)}$, the same as before.  If we assume instead that the random variables $X_{ik}$ in Algorithm~NoisyMod$^2$Diff($\varphi$) are row-wise independent, zero mean, and satisfy (\ref{Bernie1}) and (\ref{Bernie2}), that $\gamma>1/2$, and that \refstepcounter{Cnmb}\label{cassump}
\begin{equation}\label{finalassump}
\liminf_{k\rightarrow\infty}\frac{\delta_k^2 k^{n(2\gamma-1)}}{\log k}>c_{\arabic{Cnmb}}(n+2),
\end{equation}
where $c_{\ref{cassump}}=c_{\ref{cassump}}(n,\sigma)=(3^{n+2}\sigma^2)/((2\pi)^{2n})$,
then a rate arbitrarily close to $k^{-1/2}$ can be obtained by the methods outlined in Remark~\ref{cordifferem}.

\section{Phase retrieval from the modulus}\label{PRIII}

This section addresses Problem~3 in the introduction.  A simple trick converts
Problem~3 into one very closely related to Problem~2, considered in the
previous section.

Suppose, more generally, that noisy measurements are taken of
$\sqrt{\widehat{g}}$, where $g$ is an even continuous real-valued function on
$\R^n$ with support in $[-1,1]^n$.  The just-mentioned trick is to take two
independent measurements at each point, multiply the two, and use the resulting
quantities in place of the measurements of $\widehat{g}$ considered earlier.
Thus instead of (\ref{1}) above we have, for $r=1,2$, measurements
$$
\overline{g}_{z,k}^{(r)}=\sqrt{\widehat{g}(z/k^\gamma)}+X_{z,k}^{(r)},
$$
of $\sqrt{\widehat{g}}$, for $z\in \{o\}\cup \Z^n_k(+)$, where $\Z^n_k(+)$
satisfies (\ref{halfarray}) and where the $X_{z,k}^{(r)}$'s are row-wise independent
(i.e., independent for fixed $k$) zero mean random variables with uniformly bounded fourth moments.  Then we replace $\tilde{g}_{z,k}$ in (\ref{1}) by
\begin{equation}\label{2s}
\overline{g}_{z,k}=\overline{g}_{z,k}^{(1)}
\overline{g}_{z,k}^{(2)}=\widehat{g}(z/k^\gamma)+
\sqrt{\widehat{g}(z/k^\gamma)}\left(X_{z,k}^{(1)}+X_{z,k}^{(2)}\right)+
X_{z,k}^{(1)}X_{z,k}^{(2)}.
\end{equation}
Setting $\overline{g}_{jk}=\overline{g_{K_0}}_{z_{jk},k}$ and
$X_{jk}=X_{z_{jk},k}$, the same notation and analysis that gave
(\ref{newmeas}), but now using (\ref{esti}) and (\ref{2s}), leads instead to
$$
\overline{M}_k(x)=g_{K_0}(x)+\overline{N}_k(x)-d_k(x),
$$
where
\begin{equation}\label{newmeasexps}
\overline{M}_k(x)=\frac{1}{(2\pi k^\gamma)^n}
\sum_{j=-I_k'}^{I_k'}\cos(z_{jk}\cdot x)\overline{g}_{jk}
\end{equation}
is an estimate of $g_{K_0}(x)$,
\begin{equation}\label{Ns}
\overline{N}_k(x)=\frac{1}{(2\pi k^\gamma)^n}\sum_{j=-I_k'}^{I_k'}
\sqrt{\widehat{g_{K_0}}(z_{jk}/k^\gamma)}\cos(z_{jk}\cdot
x)\left(X_{jk}^{(1)}+X_{jk}^{(2)}\right)+\frac{1}{(2\pi
k^\gamma)^n}\sum_{j=-I_k'}^{I_k'}\cos(z_{jk}\cdot x)X_{jk}^{(1)}X_{jk}^{(2)}
\end{equation}
is a random variable, and the deterministic error $d_k(x)$ is given as before
by (\ref{M}).

For our analysis it will be convenient to let
\begin{equation}\label{N1s}
\overline{N}_{k1}(x)=\frac{1}{(2\pi k^\gamma)^n}\sum_{j=-I_k'}^{I_k'}
\sqrt{\widehat{g_{K_0}}(z_{jk}/k^\gamma)}\cos(z_{jk}\cdot
x)\left(X_{jk}^{(1)}+X_{jk}^{(2)}\right)
\end{equation}
and
\begin{equation}\label{N2s}
\overline{N}_{k2}(x)=\frac{1}{(2\pi
k^\gamma)^n}\sum_{j=-I_k'}^{I_k'}\cos(z_{jk}\cdot x)X_{jk}^{(1)}X_{jk}^{(2)},
\end{equation}
so that $\overline{N}_k(x)=\overline{N}_{k1}(x)+\overline{N}_{k2}(x)$.

To keep the exposition brief, we shall not give a formal presentation of our
algorithms, called {\bf Algorithm~NoisyModLSQ}, {\bf
Algorithm~NoisyModBlaschke}, and {\bf Algorithm \newline
NoisyModDiff($\varphi$)}, since they are very similar to
Algorithm~NoisyMod$^2$LSQ, Algorithm NoisyMod$^2$Blaschke, and Algorithm
NoisyMod$^2$Diff($\varphi$), respectively.  In each case the input is as
before, except that instead of (\ref{measPR}), (\ref{measPRblas}), and
(\ref{measPR3}), we now have measurements
$$
\overline{g}_{ik}^{(r)}=|\widehat{1_{K_0}}(z_{ik})|+X_{ik}^{(r)},
$$
for $r=1,2$, of the modulus of the Fourier transform of the characteristic
function of $K_0$, where the $X_{ik}^{(r)}$'s are row-wise independent zero mean random variables with uniformly bounded fourth moments.  The task is the same in each case.  For the actions, we first let $\overline{g}_{ik}=\overline{g}_{ik}^{(1)}
\overline{g}_{ik}^{(2)}$ and then follow the actions of the appropriate
algorithms in the previous section, replacing $\widetilde{g}$ by
$\overline{g}$. Thus in the action of each algorithm, we replace $M_k(x)$ by
$\overline{M}_k(x)$ defined by (\ref{newmeasexps}), for the appropriate $x$.

\begin{thm}\label{maincovPRs}
Theorem~\ref{maincovPR} holds when Algorithm~NoisyMod$\,^2$LSQ is replaced by
Algorithm NoisyModLSQ.
\end{thm}

\begin{proof}
In the action of Algorithm~NoisyModLSQ, the measurements used in
Algorithm~NoisyCovLSQ are now $\overline{M}_k(x_{ik})$, $i=-I_k',\dots,I_k'$,
where $\overline{M}_k(x_{ik})$ is given by (\ref{newmeasexps}) with $x=x_{ik}$.
Thus we have
$$\overline{M}_k(x_{ik})=g_{K_0}(x_{ik})+\overline{N}_k(x_{ik})-d_k(x_{ik}),$$
$i=-I_k',\dots,I_k'$, where $\overline{N}_k(x_{ik})$ and $d_k(x_{ik})$ are
given by (\ref{Ns}) and (\ref{M}), respectively, with $x=x_{ik}$.

We claim that Lemma~\ref{lemB} holds when $N_k(x_{ik})$ is replaced by
$\overline{N}_k(x_{ik})$. To see this, use the triangle inequality to obtain
\begin{eqnarray*}
\frac1{I_k}\sum_{i=-I_k'}^{I_k'} \overline{N}_k(x_{ik})^+&\leq &
\left(\frac1{I_k}\sum_{i=-I_k'}^{I_k'} \overline{N}_k(x_{ik})^2\right)^{1/2}\\
&\le & \left(\frac1{I_k}\sum_{i=-I_k'}^{I_k'}
\overline{N}_{k1}(x_{ik})^2\right)^{1/2}
+\left(\frac1{I_k}\sum_{i=-I_k'}^{I_k'}
\overline{N}_{k2}(x_{ik})^2\right)^{1/2},
\end{eqnarray*}
where $\overline{N}_{k1}(x_{ik})$ and $\overline{N}_{k2}(x_{ik})$ are given by
(\ref{N1s}) and (\ref{N2s}), respectively, with $x=x_{ik}$. Since
$\widehat{g_{K_0}}$ is bounded, the same analysis as in the proof of
Lemma~\ref{lemB}, up to a constant, applies to the first of the two sums in the
previous expression.  So it suffices to prove that, almost surely,
$$
\overline{S}_k=\frac1{I_k}\sum_{i=-I_k'}^{I_k'}
\overline{N}_{k2}(x_{ik})^2\rightarrow 0,
$$
as $k\rightarrow\infty$.  As in the proof of Lemma~\ref{lemB}, it is enough to
show that, almost surely,
$$\frac{1}{(2\pi k^\gamma)^{2n}}\sum_{p,q=1}^{I_k'}c_{pqk}X_{pk}^{(1)}X_{pk}^{(2)}
X_{qk}^{(1)}X_{qk}^{(2)}\rightarrow 0,$$ as $k\rightarrow\infty$.  This follows
from Lemma~\ref{lemCs} and proves the claim.

With this in hand, we can conclude exactly as in the proof of
Theorem~\ref{maincovPR} that Algorithm NoisyCovLSQ works with the new
measurements under the same hypotheses.
\end{proof}

We remark that the computation of $E(\overline{Z}_k)$ in Lemma~\ref{lemCs}
shows why we take two independent measurements of $\sqrt{\widehat{g_{K_0}}}$
and multiply, rather than taking a single measurement and squaring it. In the
latter case we would be led to
$$
E(\overline{Z}_k)= \frac{1}{(2\pi
k^\gamma)^{2n}}\sum_{p,q=1}^{m_k}a_{pqk}E(Y_{pk}^2)E(Y_{qk}^2)=O(k^{2n-2n\gamma+\nu}),
$$
which may be unbounded as $k\rightarrow\infty$.

\begin{thm}\label{covBPRs}
Theorem~\ref{covBPR} holds when Algorithm~NoisyMod$\,^2$Blaschke is replaced by
Algorithm NoisyModBlaschke.
\end{thm}

\begin{proof}
We now have
$$\overline{y}_{ik}=\zeta_{ik}+\overline{T}_{ik},$$
where $\zeta_{ik}$ is as in (\ref{XX}) and
\begin{equation}\label{XXs}
\overline{T}_{ik}=\frac{\overline{N}_k(o)-\overline{N}_k(h_ku_i)}{h_k}
=\frac{\overline{N}_{k1}(o)-\overline{N}_{k1}(h_ku_i)}{h_k}+
\frac{\overline{N}_{k2}(o)-\overline{N}_{k2}(h_ku_i)}{h_k},
\end{equation}
for $i=1,\dots,k$, where $\overline{N}_{k1}$ and $\overline{N}_{k2}$ are given
by (\ref{N1s}) and (\ref{N2s}).  The proof of Theorem~\ref{covBPR} can be
followed, except that for Lemma~\ref{ub}, one now shows that, almost surely,
$$\overline{V}_k=\frac{1}{k}\sum_{i=1}^{k}\overline{T}_{ik}^2\rightarrow 0$$ as
$k\rightarrow\infty$.  Using the fact that the earlier analysis applies to
$\overline{N}_{k1}$, and using also the triangle inequality, as we did in the
proof of Theorem~\ref{maincovPRs}, with (\ref{XXs}), we see that it suffices to
examine
$$\frac{1}{(2\pi k^\gamma)^{2n}}\sum_{p,q=1}^{I_k'}a_{pqk}X_{pk}^{(1)}
X_{pk}^{(2)}X_{qk}^{(1)}X_{qk}^{(2)},$$ where $a_{pqk}$ is given by
(\ref{anew}).  Then Lemma~\ref{lemCs} shows that it is possible to choose
$\gamma$ and $\ee$ exactly as in Theorem~\ref{covBPR} to ensure that
Lemma~\ref{ub} holds. No further changes are required, so
Algorithm~NoisyCovBlaschke works with the new measurements under the same
hypotheses as in Theorem~\ref{covBPR}.
\end{proof}

\begin{thm}\label{maindiffePRs}
Theorem~\ref{maindiffePR} holds when Algorithm~NoisyMod$\,^2$Diff($\varphi$) is
replaced by Algorithm NoisyModDiff($\varphi$).
\end{thm}

\begin{proof}
Note that by (\ref{Ns}), we have $E(\overline{N}_k(x_{ik}))=0$ for all $i$ and
$k$. Therefore the same calculations as in the proof of
Theorem~\ref{maindiffePR} show that the second statement in
Lemma~\ref{unifAsympUnbias} still holds.

In Lemma~\ref{expo}, it is enough in view of the proof of
Theorem~\ref{maindiffePR} to consider the contribution to $g_k(x)-E(g_k(x))$ from
$\overline{N}_{k2}(x_{ik})$, namely,
$$
\frac{1}{(2\pi k^{\gamma})^n}\sum_{j=-I_k'}^{I_k'}
\sum_{i=-I_k'}^{I_k'}\beta_{ik}(x)\cos(z_{jk}\cdot x_{ik})X_{jk}^{(1)}
X_{jk}^{(2)}.
$$
This allows the same estimate as before, up to a constant.  No further changes
are required, so Algorithm~NoisyCovDiff($\varphi$) works with the new
measurements under the same hypotheses as in Theorem~\ref{maindiffePR}.
\end{proof}

The previous result provides a convergence rate for
Algorithm~NoisyModDiff($\varphi$) arbitrarily close to $k^{-1/4+3/(8n)}$, as was noted for Algorithm~NoisyMod$\,^2$Diff($\varphi$) after Theorem~\ref{maindiffePR}.  If we assume instead that the random variables $X_{ik}$ in Algorithm~NoisyModDiff($\varphi$) are row-wise independent, zero mean, and satisfy (\ref{Bernie1}) and (\ref{Bernie2}), that $\gamma>1/2$, and that (\ref{finalassump}) holds, then a rate arbitrarily close to $k^{-1/2}$ can be obtained by the methods outlined in Remark~\ref{cordifferem}.

\section{Appendix}\label{Appendix}
\subsection{Convergence rates}

Rates of convergence for Algorithm~NoisyCovDiff($\varphi$), and hence for the two related algorithms for phase retrieval, are provided in Corollary~\ref{cordiffe} and Remark~\ref{cordifferem}. For the other algorithms, however, rates of convergence are more difficult to obtain.  To explain why, it will be necessary to describe some results from \cite{GKM06}, where convergence rates were obtained for algorithms for reconstructing convex bodies from finitely many noisy measurements of either their support functions or their brightness functions.  The algorithms are called Algorithm~NoisySupportLSQ and Algorithm~NoisyBrightnessLSQ, respectively.

In \cite{GKM06}, an unknown convex body $K$ is assumed to be contained in a known ball $RB^n$, $R>0$, in $\R^n$.  An infinite sequence $(u_i)$ in $S^{n-1}$ is selected, and one of the algorithms is run with noisy measurements from the first $k$ directions in the sequence as input.  The noise is modeled by Gaussian $N(0,\sigma^2)$ random variables.  With an assumption on $(u_i)$ slightly stronger than the condition that it is evenly spread (but still mild and satisfied by many natural sequences), and another unimportant assumption on the relation between $R$ and $\sigma$, it is proved in \cite[Theorem~6.2]{GKM06} that if $P_k$ is the corresponding output from Algorithm~NoisySupportLSQ, then, almost surely, there are constants $C=C(n,(u_i))$ and $N=N(\sigma,n,R,(u_i))$ such
that
\begin{equation}\label{PWConv2}
\delta_2(K,P_k)\le
C\,\sigma^{4/(n+3)}R^{(n-1)/(n+3)}k^{-2/(n+3)},
\end{equation}
for $k\ge N$, provided that the dimension $n\le 4$.  Here $\delta_2$ is the $L_2$ metric, so that $\delta_2(K,P_k)=\|h_K-h_{P_k}\|_{2}$, where $\|\cdot\|_2$ denotes the $L_2$ norm on $S^{n-1}$.  Convergence rates for the Hausdorff metric are then obtained by using the known relations between the two metrics.

It is an artifact of the method that while convergence rates can also be obtained for $n\ge 5$, neither these nor those for the Hausdorff metric are expected to be optimal. In contrast, it has recently been proved by Guntuboyina \cite{Gun10+} that the rate given in (\ref{PWConv2}) for $n\le 4$ is the best possible in the minimax sense.  With the additional assumption that $K$ is $o$-symmetric, corresponding rates for Algorithm~NoisyBrightLSQ are obtained in \cite[Theorem~7.6]{GKM06} from those for Algorithm~NoisySupportLSQ by exploiting (\ref{projbod}) and the Bourgain-Campi-Lindenstrauss stability theorem for projection bodies.

There are two principal ingredients in the proof of (\ref{PWConv2}).  The first is \cite[Corollary~4.2]{GKM06}, a corollary of a deep result of van de Geer \cite[Theorem~9.1]{vdG00}.  This corollary provides convergence rates for least squares estimators of an unknown function in a class $\mathcal{G}$, based on finitely many noisy measurements of its values, where the noise is uniformly sub-Gaussian.  The result and the rates depend on having a suitable estimate for the size of $\mathcal{G}$ in terms of its $\ee$-entropy with respect to a suitable pseudo-metric.  The second ingredient is a known estimate (see \cite[Proposition~5.4]{GKM06}) of the $\ee$-entropy of the class of support functions of compact convex sets contained in $B^n$, with respect to the $L_{\infty}$ metric.

It should be possible to apply this method to obtain convergence rates for Algorithm~NoisyCovBlaschke and the two related algorithms for phase retrieval.  With Gaussian noise, or more generally uniformly sub-Gaussian noise, this requires a modification to \cite[Theorem~9.1]{vdG00} that, in our situation, allows (\ref{eqff}) to be used instead of the same inequality without the term $c_{\ref{cgood}}/k$. (Compare \cite[(9.1), p.~148]{vdG00}.) This would yield the same convergence rates given in \cite[Theorem~7.6]{GKM06} for Algorithm~NoisyBrightLSQ.
To cover the case of Poisson noise, however, one can make the general assumption that the random variables are row-wise independent, zero mean, and satisfy (\ref{Bernie1}) and (\ref{Bernie2}), as in Remark~\ref{cordifferem}.  This creates considerable further technical difficulties.  It may well be possible to overcome these, using the machinery behind another result of van de Geer \cite[Theorem~9.2]{vdG00}.  But, as van de Geer points out in \cite[p.~134]{vdG00}, there is a price to pay: One now requires a uniform bound on the class $\mathcal{G}$ of functions, as well as estimates of $\ee$-entropy ``with bracketing."  The former condition might be dealt with by (\ref{unifBound}), which implies that the sets $\Pi Q_k$ are uniformly bounded for any fixed realization.  It should also be possible to obtain the latter, by combining suitable modifications of the bracketing argument of Lemma~\ref{net} and of the proof in \cite[Theorem~7.3]{GKM06} of the $\ee$-entropy estimate for the class of zonoids contained in $B^n$.

But we have not carried out a complete investigation into convergence rates for Algorithm~NoisyCovBlaschke and the related algorithms for phase retrieval, despite having a strategy for doing so, described in the previous paragraph.  The main reason is that there are more serious technical obstacles in achieving convergence rates for Algorithm~NoisyCovLSQ, even for the case of Gaussian noise.  In principal, the method outlined above could be applied by taking $\mathcal{G}$ to be the class of covariograms of compact convex subsets of the unit ball in $\R^n$.  However, an estimate would be required of the $\ee$-entropy of this class with respect to the $L_{\infty}$ metric or some other suitable pseudo-metric. Even if this were available, an application of the theory of empirical processes as described above would yield convergence rates not for $\delta_2(K,P_k)$ but rather for $\|g_K-g_{P_k}\|_{2}$.  To obtain rates for $\delta_2(K,P_k)$, one would then also need suitable stability versions of the uniqueness results for the Covariogram Problem described in the Introduction. In view of the difficulty of these uniqueness results, proving such stability versions will presumably be very challenging.

In summary, a full study of convergence rates for the other algorithms proposed here must remain a project for future study.

\subsection{Implementation issues}

The study undertaken in this paper is a theoretical one. Although we propose algorithms in enough detail to allow implementation, the laborious task of writing all the necessary programs, carrying out numerical experiments, and comparing with other algorithms, largely lies ahead.

At the present time we only have a rudimentary implementation of Algorithms~NoisyCovBlaschke and~NoisyCovLSQ.  The programs were written, mainly in Matlab, by Michael Sterling-Goens while he was an undergraduate student at Western Washington University, and are confined to the planar case.  Algorithm~NoisyCovBlaschke seems to be very fast; this is to be expected, since it is based on Algorithm~NoisyBrightnessLSQ, which is also fast even in three dimensions.  Behind both of these latter two algorithms is a linear least squares problem (cf.~\cite[(18) and (19)]{GarM03}).  In contrast, the least squares problem (\ref{obj1}) in Algorithm~NoisyCovLSQ is nonlinear.  Preliminary experiments indicate that reasonably good reconstructions, such as those depicted in Figures~\ref{poly1}--\ref{ellipse2} (based on Gaussian $N(0,\sigma^2)$ noise, $k=60$ equally spaced directions in Algorithm~NoisyCovBlaschke and $k=8$ in Algorithms~NoisyCovLSQ), can usually be obtained in a reasonable time in the planar case.  Occasionally, however, reconstructions can be considerably worse, particularly for regular $m$-gons for very small $m$. Better and faster reconstructions, also in higher dimensions, will probably require bringing to bear the usual array of techniques for nonlinear optimization, such as simulated annealing.
\vspace{-.1in}

\begin{figure}[ht]
\centering
\begin{tabular}{cc}
\begin{minipage}{3in}
\centering
\includegraphics[width=8cm]{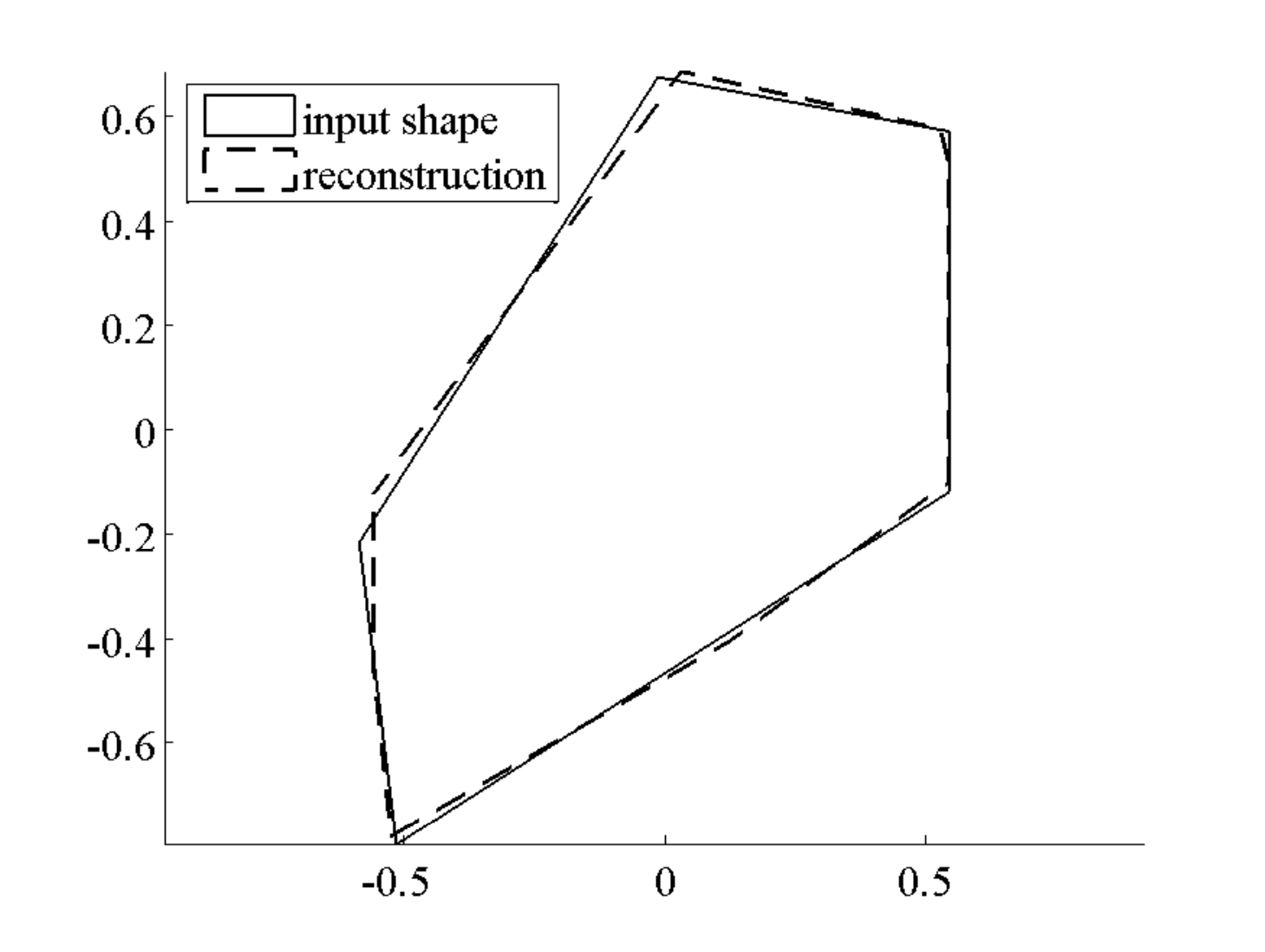}
\vspace{-.4in}\caption {Pentagon, no noise}\label{poly1}
\end{minipage}
& \begin{minipage}{3in}
\includegraphics[width=8cm]{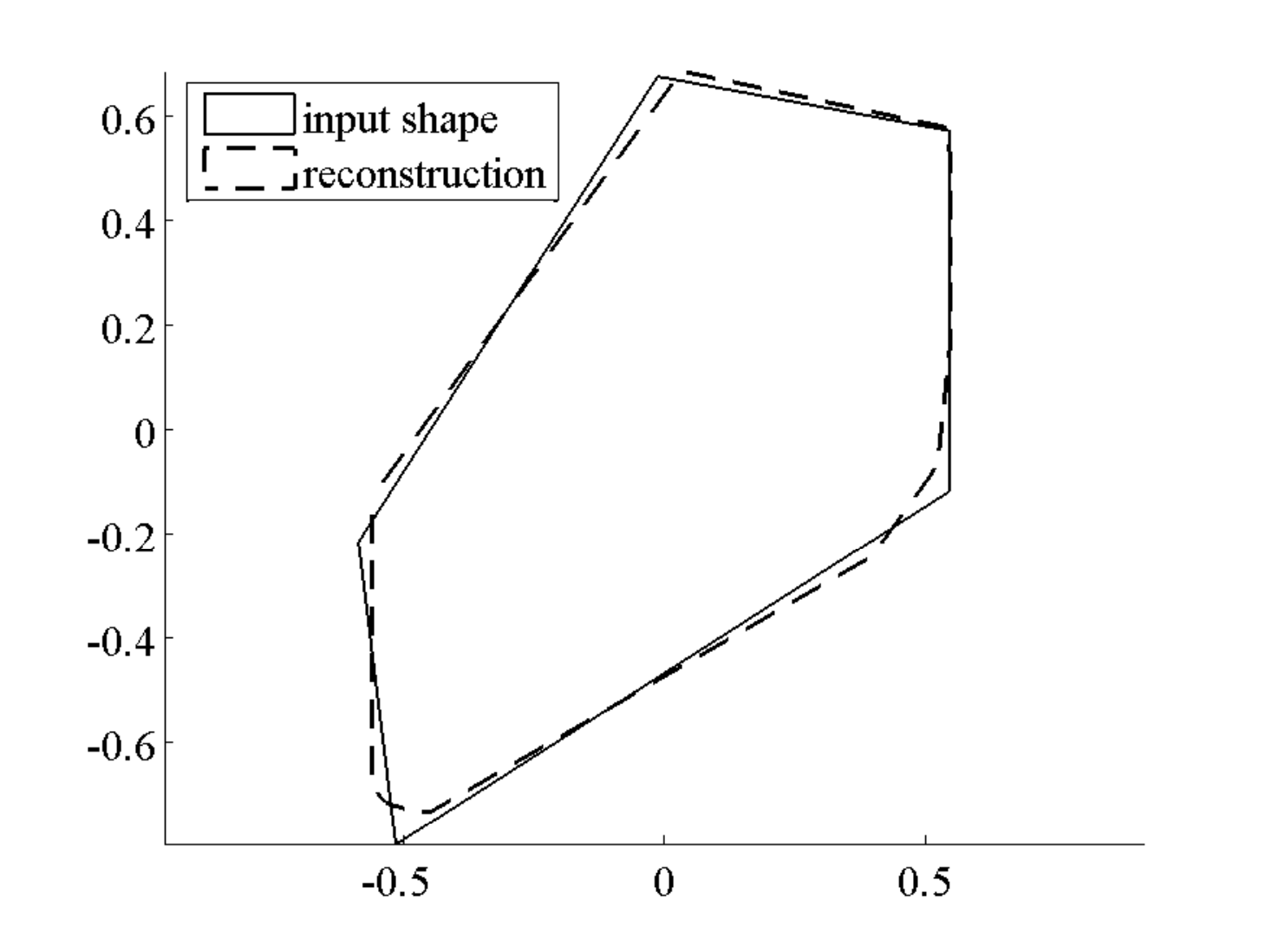}
\vspace{-.4in} \caption {Pentagon, $\sigma=0.01$}\label{poly2}
\end{minipage}\\
\end{tabular}
\end{figure}

\vspace{-.3in}

\begin{figure}[ht]
\centering
\begin{tabular}{cc}
\begin{minipage}{3in}
\centering
\includegraphics[width=8cm]{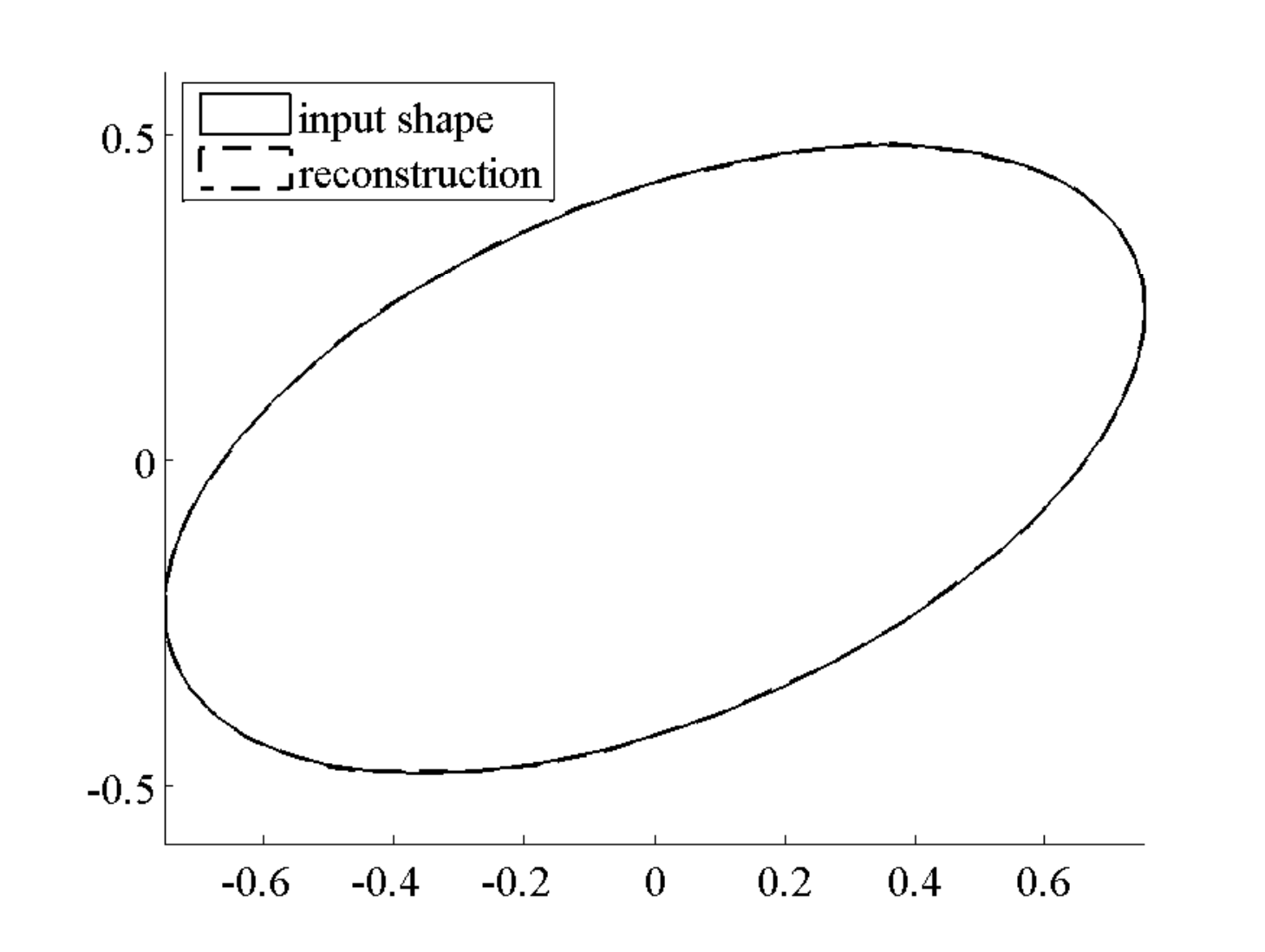}
\vspace{-.4in}\caption {Ellipse, no noise}\label{ellipse1}
\end{minipage}
& \begin{minipage}{3in}
\includegraphics[width=8cm]{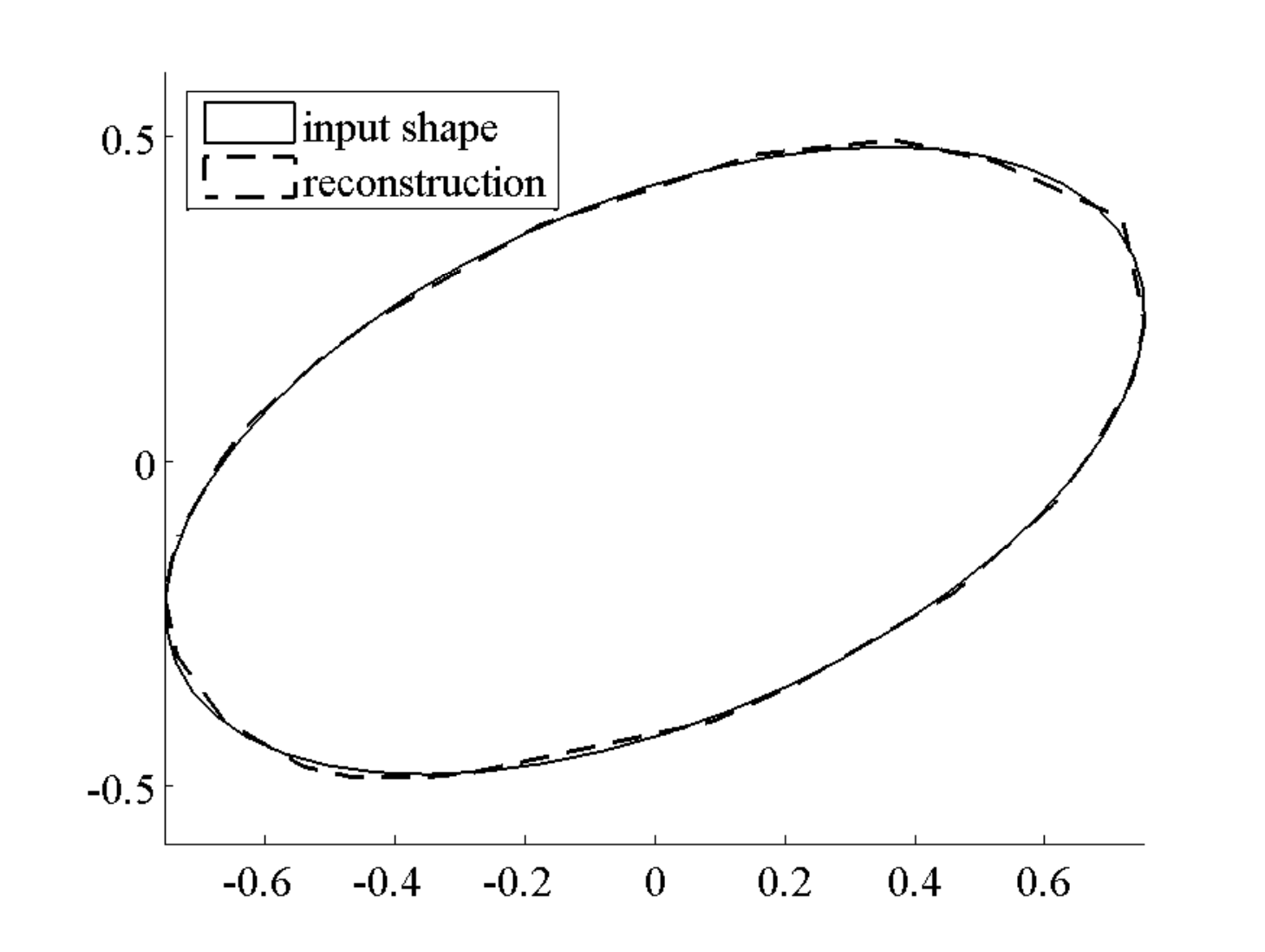}
\vspace{-.4in} \caption {Ellipse, $\sigma=0.01$}\label{ellipse2}
\end{minipage}\\
\end{tabular}
\end{figure}

Since the least squares problem (\ref{obj1}) is nonlinear, it is important to control the number of variables, that is, the number of facets of the approximation $Q_k$ to the Blaschke body $\nabla K_0$ of $K_0$.  To a large extent, Algorithm~NoisyCovBlaschke already does this; the potential $O(k^{n-1})$ variables that would otherwise be required (see \cite[p.~1335]{GKM06}) is, as experiments show, considerably reduced.  In fact, if there is little or no noise, a linear programming version of the brightness function reconstruction program due to Kiderlen (see \cite[p.~289]{GarM03}, where it is stated for measurements without noise) is not only even faster, but also produces approximations $Q_k$ to $\nabla K_0$ with at most $2k$ facets.  Beyond this, there is the possibility of using the pruning techniques discussed in \cite[Section~3.3]{PMG06}.

There is also the possibility of changing the variables in the least squares problem (\ref{obj1}). A convex polytope $P$ whose facet outer unit normals are a subset of a prescribed set $\{\pm u_j: j=1,\dots, s\}$ of directions can be specified by the vector $h=(h_1^+,h_1^-,\dots,h_s^+,h_s^-)$ such that
$$P=P(h)=\{x\in\R^n: -h_j^-\le x\cdot u_j\le h_j^+, j=1,\dots,s\}.$$
The possible advantage in using these variables arises from the fact that, by the Brunn-Minkowski inequality (cf.~\cite[Section~11]{Gar02}), the covariogram $g_{P(h)}(x)$ turns out to be $(1/n)$-concave (i.e., $g_{P(h)}(x)^{1/n}$ is concave) on its support in the combined variable $(h,x)$. One may therefore try solving the problem
\begin{equation}\label{obj1app}\min
\sum_{i=1}^{I_k}\left(M_{ik}-g_{P(h)}(x_{ik})\right)^2
\end{equation}
over the variables $h_1^+,h_1^-,\dots,h_s^+,h_s^-$.  By expanding the square in (\ref{obj1app}), approximating the sums by integrals, and using the Pr\'{e}kopa-Leindler inequality \cite[Section~7]{Gar02}, the objective function can be seen as an approximation to the difference of two log-concave functions. These admittedly weak concavity properties may help.

Regularization is often used to improve Fourier inversion in the presence of noise.  We expect this to be of benefit in implementing the phase retrieval algorithms, where preliminary investigations indicate that regularization will allow the restriction on the parameter $\gamma$ to be considerably relaxed.

Corresponding to the two basic approaches to reconstruction---one via the Blaschke body and one via the difference body---there are two different sampling designs.  For the former, measurements are made first at the origin and at points in a small sphere centered at the origin, and then again at points in a cubic array.  For the latter, measurements are made twice, each time at points in cubic array.  These sampling designs are a matter of convenience, at least regarding the cubic array. It should be possible to use a variety of different sets of measurement points, at least for reconstructing from covariogram measurements, with appropriate adjustments in the consistency proofs.

\end{document}